\pdfoutput=1

\documentclass[11pt]{article}

\usepackage{amsfonts}
\usepackage{amsmath,amsthm,mathrsfs}
\usepackage{amssymb}

\usepackage{vmargin}
\setmarginsrb{2.5cm}{2cm}{2.5cm}{2cm}{.4cm}{4mm}{0cm}{7.5mm}

\usepackage{times}

\usepackage[T1]{fontenc}
\usepackage{epsfig}
\usepackage{graphicx}
\usepackage{amssymb}
\usepackage{amsmath}
\usepackage[utf8]{inputenc}
\usepackage{amsfonts}
\usepackage{amsthm}
\usepackage{paralist}
\usepackage{stmaryrd}
\usepackage[toc,page]{appendix}
\usepackage{bbm}
\usepackage{bm}
\usepackage{mathrsfs}
\usepackage{comment}
\usepackage{longtable}
\usepackage[labelfont=bf,font=it]{caption} 

\usepackage[
  colorlinks=true,
  pdfencoding=auto, %
  psdextra,
]{hyperref}[2012/08/13]
\usepackage[
backend=biber,
]{biblatex}

\usepackage{centernot}
\usepackage[capitalize]{cleveref}
\usepackage{soul}

\theoremstyle{plain}
\newtheorem{corollary}{Corollary}[section]
\newtheorem{theorem}[corollary]{Theorem}
\newtheorem{lemma}[corollary]{Lemma}
\newtheorem{claim}[corollary]{Claim}
\newtheorem*{conjecture*}{Conjecture}

\newtheorem*{theorem*}{Theorem}
\newtheorem*{lemma*}{Lemma}
\newtheorem*{definition*}{Definition}
\newtheorem*{corollary*}{Corollary}
\theoremstyle{definition}

\newtheorem{definition}[corollary]{Definition}

\newtheorem{remark}[corollary]{Remark}
\newtheorem{remark*}{Remark}

\newcommand{\R}{\mathbb{R}}
\newcommand{\Z}{\mathbb{Z}} 
\newcommand{\T}{\mathbb{T}} 
\newcommand{\N}{\mathbb{N}}

\newcommand{\C}{\mathcal{C}}

\newcommand{\comp}{\mathfrak{C}}

\newcommand{\proba}{\mathbb{P}}
\newcommand{\E}{\mathbb{E}}

\newcommand{\X}{\chi}
\newcommand{\1}{\mathbf{1}}

\newcommand{\off}{\hbox{ {\rm off} }}
\newcommand{\be}{\begin{eqnarray}} 
\newcommand{\ee}{\end{eqnarray}}

\newcommand{\Op}{O_{\proba}}

\newcommand{\cD}{\mathcal{D}}

\usepackage{xcolor} 
\usepackage[many]{tcolorbox}
\newcommand{\red}[1]{{\color{red} #1}}

\newcommand{\eps}{\varepsilon}

\newcommand{\ZZ}{\mathbf{Z}}  %

\renewcommand{\l}{\langle}
\renewcommand{\r}{\rangle}
\newcommand{\pe}{\preceq}

\newcommand{\pp}{\proba_{p_2^*}}
\newcommand{\tp}{\tau _{p_2^*}}
\newcommand{\W}{\mathcal{W}}
\newcommand{\good}{\Omega_{\rm good}}
\newcommand{\weight}[1]{\mathbf{wt(}#1\mathrm{)}}

\newcommand{\qiic}{\mathbb{Q}_{\mathrm{IIC}}}

\renewcommand{\mp}{{3/5}}

\DeclareMathOperator*{\limit}{\longrightarrow}

\DeclareMathOperator*{\Var}{\mathrm {Var}}

\DeclareMathOperator*{\lr}{\longleftrightarrow}
\DeclareMathOperator*{\nlr}{\centernot\longleftrightarrow}
\DeclareMathOperator*{\Tr}{Tr}

\newcommand{\without}{\hbox{ {\rm without} }}
\DeclareMathOperator*{\slr}{\leftrightarrow}

\newcommand{\andd}{\hbox{ {\rm and} }}

\newcommand{\cs}{t}

\newcommand{\cnst}{\cs^*}
\newcommand{\cmin}{\mathfrak{t}_{\mathrm{min}}}

\newcommand{\Cnst}{\mathbf{C}}

\newcommand{\ti}[1]{\widetilde{#1}}
\newcommand{\M}{M}

\counterwithin{figure}{section}

\title{Critical percolation on the discrete torus in high dimensions}
\author{Arthur Blanc-Renaudie \and Asaf Nachmias}

\numberwithin{equation}{section}

\begin{document}
\maketitle
\begin{abstract} We consider percolation on the discrete torus  $\Z_n^d$ at $p_c(\Z^d)$, the critical value for percolation on the corresponding infinite lattice $\Z^d$, and within the scaling window around it. We assume that $d$ is a large enough constant for the nearest neighbor model, or any fixed $d>6$ for spread-out models. We prove that there exist constants $\Cnst, \Cnst'$ depending only on the dimension and the spread-out parameter such that for any $\lambda \in \R$ if the edge probability is $p_c(\Z^d)+\Cnst \lambda n^{-d/3} + o(n^{-d/3})$, then the joint distribution of the largest clusters normalized by $\Cnst' n^{-2d/3}$ converges as $n\to \infty$  to the ordered lengths of excursions above past minimum of an inhomogeneous Brownian motion started at $0$ with drift $\lambda-t$ at time $t\in[0,\infty)$. This canonical limit was identified by Aldous in the context of critical Erd\H{o}s--R\'enyi graphs. %

\end{abstract}

\section{Introduction}

Mean-field approximation is a simple yet powerful heuristic employed since the beginning of the 20th century. 
In this heuristic, the geometry is ``trivialized'' by approximating the interaction between a single particle and all others by that between the particle and a field representing the average effect of the rest. This heuristic is believed to be precise at the critical temperature only when the dimension is larger than the \emph{upper critical dimension} $d_u$; for percolation $d_u=6$ \cite{TreeGraph, HS90}. It is thus natural to compare critical percolation on the discrete torus $\Z_n^d$ with $d>6$ and on the complete graph. The latter is known as the Erd\H{o}s--R\'enyi random graph $G(n,p)$. This approach was first taken by Borgs, Chayes, van der Hofstad, Slade and Spencer \cite{BCHSS1,BCHSS2}, who showed that component sizes are of the same \emph{order} as those in the critical $G(n,p)$. Here we show that their scaling limits coincide. 

We consider two edge sets on the vertices of $\Z_n^d$: (a) the nearest-neighbor model where edges connect any two vertices of $\ell_1$ distance precisely $1$; we require that $d$ is large enough but fixed, and, (b) the spread-out model where edges connect any two vertices of $\ell_\infty$ distance at most $L$; we require that $d>6$ is fixed and $L$ is large enough. We call each of these graphs the {\bf discrete torus in high dimensions} and denote by $V=n^d$ the number of vertices. Percolation on these graphs is obtained by independently retaining each edge with probability $p\in[0,1]$ and erasing it otherwise.  The corresponding infinite lattice is the graph on vertex set $\Z^d$, again with either nearest neighbor edges and $d$ a large enough constant, or spread-out edges and $d>6$. We write $p_c(\Z^d)$ for the critical value for percolation on $\Z^d$.

\begin{theorem} \label{thm:main_pc} 
Denote by $(\C_1, \C_2,\ldots)$ the connected component of percolation on the discrete torus $\Z_n^d$ in high dimensions with edge probability $p\in[0,1]$ sorted in non-increasing order. There exist constants $\Cnst=\Cnst(d,L)\in(0,\infty)$ and $\Cnst'=\Cnst'(d,L)\in(0,\infty)$ such that for any $\lambda\in \R$, if 
$$ p = p_c(\Z^d) + \Cnst \lambda V^{-1/3} + o(V^{-1/3}) \, ,$$
then 
\be\label{eq:ConvergencePC} \Cnst' V^{-2/3}(|\C_1|, |\C_2|,  \ldots) \underset{\ell_2}{\overset{(d)}{\longrightarrow}}  \ZZ_{\lambda} \qquad \textrm{as } n\to \infty \, ,\ee
where $\ZZ_{\lambda}$ is the list of lengths of excursions above past minimum of an inhomogeneous Brownian motion started at $0$ with drift $\lambda-t$ at time $t\in[0,\infty)$ sorted in descending order.   
\end{theorem}

\begin{remark}\label{rmk:constants}
The constants $\Cnst, \Cnst'$ are expressed as $\Cnst = \Cnst_1 \Cnst_2^{2/3}$ and $\Cnst'=\Cnst_2^{1/3}$ where $\Cnst_1, \Cnst_2$ are the constants from \cref{thm:sharpChiChi2Zd} below. 
\end{remark}

The random vector $\ZZ_\lambda$ was introduced in the seminal work of Aldous \cite{Aldous97}. There he shows that component sizes normalized by $n^{-2/3}$ and sorted in descending order of the Erd\H{o}s--R\'enyi random graph $G(n,{1+\lambda n^{-1/3} \over n})$, with $\lambda\in \R$ fixed, converge to $\ZZ_\lambda$ as $n\to \infty$ (before his work \L uczak, Pittel and Wierman \cite{LPW94} proved that the random vector $n^{-2/3} (|\C_1|,|\C_2|,\ldots)$ converges in distribution but to an unknown limit). Prior to \cite{Aldous97} it has already been established by Bollob\'as \cite{Bol84} followed by \L uczak \cite{Luczak90}, that for $G(n,{1\pm\eps(n) \over n})$ with $\eps(n) \gg n^{-1/3}$, the sizes of the largest components are concentrated so there are no non-trivial scaling limits in that regime. Thus, these works fully characterize the \emph{scaling window} of the Erd\H{o}s and R\'enyi random graph, that is, the range of parameters (depending on the system size) where critical behavior, such as non-trivial scaling limits and power law decay, occurs. In the Erd\H{o}s and R\'enyi model this window is $p={1+\lambda n^{-1/3} \over n}$ over $\lambda \in \R$. In the same vein, \cref{thm:main_pc} fully characterizes the scaling window on the discrete torus in high dimensions centered around the point $p_c(\Z^d)$ corresponding to $\lambda=0$. It follows that just above the scaling window, i.e., in what is known as the \emph{slightly supercritical phase}, the largest component is much larger than $V^{2/3}$, and the second largest much smaller.
\begin{corollary} \label{cor:supercritical} Denote by $(\C_1, \C_2,\ldots)$ the connected component sorted in non-increasing order of percolation on the discrete torus $\Z_n^d$ in high dimensions with edge probability $p = p_c(\Z^d) + \eps(n)$ where $\eps(n)V^{1/3} \to \infty$ \emph{slowly enough}. Then 
$$ |\C_1| \frac{\mathbf C_2 \Cnst_1}{2\eps_n V} \to 1 \, ,$$
and for every $j\geq 2$
$$|\C_j|  \frac{\eps_n^2}{2\mathbf C_1^2\mathbf C_2\log(\eps_n^3V)}\to 1 \, ,$$
where the convergence is in probability.
\end{corollary}

In an influential paper \cite{BCHSS1}, Borgs, Chayes, van der Hofstad, Slade, and Spencer, defined a useful and natural notion of the \emph{critical threshold} for percolation on a sequence of finite transitive graphs. Informally it is just the set numbers $p\in [0,1]$ in which the expected size of the component containing any fixed vertex is of order $V^{1/3}$, where $V$ denote the number of vertices in the $n$\textsuperscript{th} graph. They showed that if a sequence of transitive graphs satisfies a finite analogue of the so-called triangle condition, then there is a phase transition similar to the one in the Erd\H{o}s and R\'enyi, in particular, that component sizes at the critical threshold have order $V^{2/3}$. The techniques of the current paper allow us find the precise relation between the critical thresholds of \cite{BCHSS1} and the ones of \cref{thm:main_pc} in the case of the discrete torus in high dimensions, and to infer the scaling limit result at these probabilities. Let us state this precisely. 

It is well known (see \cite[Corollary 3.9]{HypercubePaper}) that there exists an increasing continuous bijection $\kappa:\R \to (0,\infty)$ such that for each $\lambda \in \R$ the Erd\H{o}s--R\'enyi random graph $G(n,{1+\lambda n^{-1/3} \over n})$ satisfies $n^{-1/3} \E|\C(v)| \to \kappa(\lambda)$ where $\C(v)$ is the connected component containing $v$; by transitivity the distribution of $|\C(v)|$ does not depend on $v$. Given $\lambda\in \R$ we define $p_c(\lambda)=p_c(\lambda,n)$ (we suppress the dependence on $d$ and $L$) as the unique number $p\in [0,1]$ such that in percolation on the discrete torus we have
\be\label{eq:defpc} \E _p|\C(v)| = \Cnst_2^{-2/3} \kappa(\lambda) V^{1/3} \, ,\ee
where $\Cnst_2=\Cnst_2(d,L)>0$ is a constant defined in \cref{thm:sharpChiChi2Zd}; we set $p_c(\lambda)=1$ if the right hand side of \eqref{eq:defpc} is greater than $V$
.%

It is not at all clear that $p_c(\Z^d)$ resides within the scaling window defined in \eqref{eq:defpc}; this was shown by Heydenreich and van der Hofstad in \cite{HeydenHofstad1,HeydenHofstad2} who proved that $|p_c(\Z^d) - p_c(\lambda)| \pe V^{-1/3}$ for any $\lambda \in \R$. Here we precisely pinpoint $p_c(\lambda)$ in terms of the scaling window of \cref{thm:main_pc} and infer the convergence of component sizes at $p_c(\lambda)$.

\begin{theorem} \label{thm:main} 
Consider percolation on the discrete torus in high dimensions and let $p_c(\lambda)$ for fixed $\lambda \in \R$ be defined as in \eqref{eq:defpc}. Then
$$ \Big | p_c(\lambda) - (p_c(\Z^d) + \Cnst \lambda V^{-1/3}) \Big | = o(V^{-1/3}) \, ,$$
where $\Cnst=\Cnst(d,L)$ is from \cref{thm:main_pc}. It thus follows that the convergence \eqref{eq:ConvergencePC} holds at $p_c(\lambda)$.
\end{theorem}

Lastly, the constants in \cref{thm:main_pc}, \cref{cor:supercritical} and \cref{thm:main} are obtained from a new sharp estimate on the susceptibility in percolation on the infinite lattice $\Z^d$ in high dimensions. This estimate improves a result of Aizenman and Newman \cite{TreeGraph} and Hara and Slade \cite{HS90} and is of independent interest. Recall that the susceptibility $\chi_{\Z}(p)$ is the expectation of the cluster at the origin $|\C^{\Z}(0)|$ in percolation at $p<p_c(\Z^d)$. The $\Z$ in the notation is introduced in order to emphasize we work here on the infinite lattice $\Z^d$, not the torus. Aizenman and Newman showed that $c(p_c-p)^{-1}\leq \chi_{\Z}(p)\leq C(p_c-p)^{-1}$ for universal positive constants $C,c$ as long as the so-called triangle condition holds; the latter was subsequently shown to hold for the infinite lattice $\Z^d$ in high dimensions in the seminal work of Hara and Slade \cite{HS90}. Here we provide an estimate on the susceptibility and on the expected cluster size squared that are sharp up to $1+o(1)$ as $p \uparrow p_c$. The constants obtained here determine the constants in \cref{thm:main_pc}, \cref{cor:supercritical} and \cref{thm:main}, as stated in \cref{rmk:constants}.
\begin{theorem}\label{thm:sharpChiChi2Zd}
Consider the infinite lattice $\Z^d$ in high dimensions. There exists positive constants $\Cnst_1,\Cnst_2 \in (0,\infty)$ depending only on $d$ and $L$ such that as $p \uparrow p_c$,
$$ \chi_{\Z}(p) = (\Cnst_1+o(1)) (p_c - p)^{-1} \, , \qquad \E_p \big [ |\C^\Z(v)|^2 \big ] = (\Cnst_2 + o(1)) \chi_\Z(p)^3.$$
\end{theorem}

\noindent A similar estimate for the FK-Ising model was very recently given in \cite{SharpChiFKIsing}.

\subsection{Remarks}

\begin{remark} \cref{thm:main} proves a conjecture posed in \cite[p401]{HeydenHofstad1}, see also \cite[Open Problem 13.5]{ProgressHDP}. Also, \cref{cor:supercritical}  confirms a prediction of \cite{BCHSS2} (see \cite[Theorem 1.7]{BCHSS2} and the paragraph below it).
\end{remark}
\begin{remark}
\cref{thm:main_pc} shows that $p_c(\Z^d)$ is in some sense the ``center'' of the scaling window as it corresponds to $p_c(\lambda)$ with $\lambda=0$. In other words, the largest components' sizes of $\Z_n^d$ at $p_c(\Z^d)$, when multiplied by $\Cnst'$, and those of the critical Erd\H{o}s and R\'enyi random graph $G(n,{1 \over n})$  have the same scaling limit $\ZZ_0$. We find this  surprising since both $p_c(0)$ and $p_c(\Z^d)$ are arbitrary points in the scaling window and we see no apparent reason why they coincide. 
\end{remark}

\begin{remark} The behavior at $p_c(\Z^d)$ in the box with free or bulk boundary conditions in high dimensions is different, as first predicted by Aizenman \cite{Aizenman97} (the torus is often referred to in the literature as the $d$-dimensional box with periodic boundary conditions). See the discussion in \cite{HeydenHofstad1, HeydenHofstad2} and the recent paper \cite{CarWer25}. We believe  $p_c(\Z^d)$ is subcritical for the box in the sense that at this edge probability we have $|\C_1|/|\C_j|\to 1$ in probability for any fixed $j \geq 2$. Also, it is likely that critical behavior as in \cref{thm:main} holds at some $p_c^{\text{box}}(\lambda,n)>p_c(\Z^d)$.

\end{remark}

\begin{remark} In a separate publication we will show the convergence in distribution, with respect to the Gromov-Hausdorff-Prokhorov (GHP) topology, of the large components of critical percolation on the discrete torus in high dimensions (viewed as metric measure spaces) to the canonical limit constructed by Addario-Berry, Broutin and Goldschmidt \cite{ABG12}. Since these results rely on a different set of techniques we chose to address them in a separate paper.
\end{remark}

\begin{remark} Together with Broutin, we have recently proved the analogue of \cref{thm:main} and GHP convergence of the components in the setting of critical percolation on the Hamming hypercube $\{0,1\}^n$, see \cite{HypercubePaper}. In \cref{sec:CompareHypercube} we examine the differences and identify the obstacles that hinder the application of the techniques from our previous paper \cite{HypercubePaper} to the discrete torus.
\end{remark}
\begin{remark} %
In a work in progress \cite{BRH25}, the first author and Tom Hutchcroft show that the cluster at the origin of percolation on $\Z^d$ at $p_c(\Z^d)$ conditioned to be large, converges after proper scaling towards the Brownian Continuum Random Tree / super Brownian motion. The convergence includes jointly the masses, distances, and spatial positions. 

Although the two models are closely related, there is a major distinction: in finite volume, %
 critical clusters can have macroscopic loops which leads to various difficulties. %
 Most notably, %
 the ``open'' triangle condition, which plays a central role in the infinite-lattice analysis, fails in the toroidal setting. This is the main reason the proofs in the two models are very different. %
\end{remark}

\begin{remark} In the Erd\H{o}s and R\'enyi random graph there is an important distinction between the two regimes $p={1- \eps(n) \over n}$ and $p={1+ \eps(n) \over n}$ where $\eps(n)\gg n^{-1/3}$, i.e., between the subcritical and supercritical phase. In the subcritical phase $|\C_1|$ and $|\C_2|$ are both concentrated, are of size $o(V^{2/3})$ and are asymptotically of the same size i.e., $|\C_1|/|\C_2| = 1+o(1)$ with probability $1-o(1)$ (in fact this holds with $|\C_2|$ replaced by $|\C_j|$ for any fixed $j \geq 1$). In the supercritical phase again both $|\C_1|$ and $|\C_2|$ are concentrated however $|\C_1| \gg n^{2/3}$ and $|\C_2| \ll n^{2/3}$ (the precise asymptotics was obtained by \cite{Bol84} and \cite{Luczak90}). 

In the case of the torus $\Z_n^d$ with sequences $\eps(n) \gg V^{-1/3}$ that do not grow too fast, this is precisely what is shown in \cref{cor:supercritical} with $p=p_c(\Z^d)(1+\eps(n))$. The proof of \cref{cor:supercritical} is a straightforward application of \cref{thm:main_pc} together with \cite{Bol84} and \cite{Luczak90}. Showing this for any sequence satisfying $\eps(n)=o(1)$ and $\eps(n) \gg V^{-1/3}$ remains an open problem. (The order of magnitude of component sizes in the regime $V^{-1/3}\ll \eps(n)\ll \log(V)^{-C}$ is determined in an upcoming work by the first author.)
\end{remark}

\subsection{Brief background} \label{sec:background}
Around the same time as the aforementioned discoveries concerning the critical Erd\H{o}s–R\'enyi random graph, Hara and Slade \cite{HS90} made a seemingly unrelated breakthrough. They used a perturbative technique known as the \emph{lace-expansion} to prove the finiteness of the so-called \emph{triangle diagram}
\be\label{eq:triangleDiagram} \sum_{x,y \in \Z^d} \proba_{p_c}(0 \lr x) \proba_{p_c}(x \lr y)  \proba_{p_c}(y \lr 0) < \infty \, ,\ee
whenever $d$ is sufficiently large or when $d>6$ and the lattice is sufficiently spread out. In the above $\proba_{p_c}$ the probability measure obtained by performing percolation on the edges of $\Z^d$ at $p_c=p_c(\Z^d)$, that is, the critical percolation probability of $\Z^d$, and $x \lr y$ is the event that $x$ is connected to $y$ in a path of retained edges. The finiteness of the triangle diagram in high dimensions was suggested first by Aizenman and Newman \cite{TreeGraph} and was shown to have a great number of implications. It easily implies that no infinite components appear with probability $1$ at $p_c $ (this assertion is still wide open in $3\leq d\leq 6$), but it also shows that $\proba_{p_c(\Z^d)}(|\C(0)|\geq n) \asymp n^{-1/2}$, as shown by Barsky and Aizenman \cite{BarskyAizenman91}. The exponent $1/2$ is one of the \emph{critical exponents} ($\delta=2$ in this case) which govern polynomial the decay/growth of various related percolation quantities. In dimensions above the upper critical dimensions we expect the critical exponents to attain their mean-field value; $1/2$ in the aforementioned case. We refer the reader to \cite{SladeSTFLOUR} and \cite{ProgressHDP} for comprehensive overviews. 

The critical phenomenon in percolation above the upper critical dimension in sequences of \emph{finite} transitive graphs was initiated in the influential papers \cite{BCHSS1,BCHSS2} mentioned earlier. There are two inherent obstacles to finding the right analogue of \eqref{eq:triangleDiagram}. First, unlike the infinite case, it is unclear how to define the critical percolation probability. Second, the sum in \eqref{eq:triangleDiagram} is always finite when the graph is finite. To overcome these obstacles, in \cite{BCHSS1} the authors define the critical percolation probability $p_c(\lambda)$ as in \eqref{eq:defpc} (their definition is slightly different, but it is immaterial for this discussion) and the finite analogue of the triangle condition \eqref{eq:triangleDiagram}. They prove that \emph{if} it holds, then the definition \eqref{eq:defpc} represents the critical window as $\kappa$ ranges in $(0,\infty)$ (or $\lambda$ in $(-\infty, \infty)$) in the sense that it guarantees that $|\C_1|V^{-2/3}$ and $|\C_1|^{-1}V^{2/3}$ are tight random variables (see \cite[Theorem 1.3]{BCHSS1}) for any edge probability within this scaling window and that various other quantities in the subcritical, critical and supercritical regimes behave as they do in the Erd\H{o}s and R\'enyi random graph, up to multiplicative constants. We use these consequences throughout the paper and do not directly use the triangle condition defined in \cite{BCHSS1}, so we omit its definition. The same authors proceeded in \cite{BCHSS2} to show that this finite analogue of the triangle condition holds for various graph sequences, in particular, to the torus $\Z_n^d$ in high dimensions and so the results of \cite{BCHSS1} apply to it.

There is no apparent relation between \eqref{eq:defpc} and $p_c(\Z^d)$, i.e., the critical threshold for percolation on the infinite lattice $\Z^d$. It is natural to guess, and was speculated by Aizenman \cite{Aizenman97}, that $p_c(\Z^d)$,  is inside the aforementioned critical window, i.e., that $\E_{p_c(\Z^d)}|\C(v)|$ is of order $V^{1/3}$, or equivalently, that $|p_c(\Z^d)-p_c(\lambda)|$ is of order at most $V^{-1/3}$ for any fixed $\lambda \in \R$. This statement was shown to hold by Heydenreich and van der Hofstad \cite{HeydenHofstad1,HeydenHofstad2}; an alternate proof of this is also given in \cite{HMS23}. It follows that the results of \cite{BCHSS1} hold at $p_c(\Z^d)$ and at a scaling window of width $V^{-1/3}$ around it,  namely, that $V^{-2/3}|\C_1|$ and $V^{2/3}|\C_1|^{-1}$ are tight random variables. 

Given this tightness it is natural to ask what is the scaling limit of component sizes of critical percolation on the torus throughout the scaling window. \cref{thm:main_pc} answers this question.

\subsection{Notations and terminology}\label{sec:notation}

\noindent {\bf Torus and lattice notation.} As mentioned earlier, we consider the discrete torus $\Z_n^d$ with either nearest-neighbor edges ($L=1$) and $d$ is a large fixed number, or spread out edges with $\ell_\infty$ distance at most $L$ and $d$ is any fixed number larger than $6$. We denote the number of vertices by $V=n^d$ and the degree by $m$; in the nearest-neighbors case $m=2d$ and in the spread out case $m=(2L+1)^d-1$. Even though they are sometimes required to be large enough, as just described, the numbers $d,L$ and $m$ are considered as fixed constants throughout this paper.

We also consider the infinite lattice $\Z^d$ with the same requirements over the edge sets, $d$ and $L$, and call it the {\bf infinite lattice $\Z^d$ in high dimensions}. For both infinite and finite graphs we write $0$ to denote the vertex corresponding to the all zeroes vector and treat it as a root. For $r >0$ let $B_r$ be the box $\{-r,\ldots,r\}^d$ of side length $2r+1$ around the origin. Lastly, to avoid division by zero, for a vertex $u$ of $\Z^d$ or $\Z_n^d$ we write $\l u \r:= \max(|u|,1)$. \\

\noindent {\bf Asymptotic notation.} For two numbers $a,b$, that may depend on several parameters depending on the context, we write $a \pe b$,  if there is a deterministic constant $C=C(d,L)\in(0,\infty)$ such that $|a| \leq C |b|$. This constant will always be implicit, and may vary from line to line. We will explicitly emphasize the rare cases where we use this notation and the constant $C$ depends on other parameters. We write $a\asymp b$ if $a\preceq b$ and $b\preceq a$. For every pair of real random variables $X,Y$, and event $E$ we say that $X\preceq Y$ on $E$ if there exists a deterministic constant $C=C(d,L)$ such that $|X|\leq C|Y|$ on $E$. For every two sequences $(X_n)_{n\in \N}$ and $(Y_n)_{n\in \N}$ of real random variables, we say that with probability $1-o(1)$ we have $X_n\preceq Y_n$ if there exists a deterministic constant $C=C(d,L)$ such that $\proba(|X_n|\leq C|Y_n|)\to 1$. Furthermore, we write $X_n = O_\proba(Y_n)$ when $X_n/Y_n$ is a tight sequence and $X_n = o_{\proba}(Y_n)$ when $X_n/Y_n \to 0$ in probability.

We also use the standard Landau notation. When $a_n$ and $b_n$ are two non-negative sequences we write $a_n=O(b_n)$ when there is a constant $C=C(d,L)\in(0,\infty)$ such that $a_n \leq Cb_n$ for all $n$, and write 
$a_n = o(b_n)$ or $a_n\ll b_n$ or $b_n\gg a_n$ if $a_n/b_n \to 0$. We write $a_n= \Theta(b_n)$ if $a_n = O(b_n)$ and $b_n =O(a_n)$.

Lastly, we write $o_m(1)$ for a fixed small number that tends to $0$ as $m\to \infty$; we emphasize that unlike the usual notation $o(1)$, when writing $o_m(1)$ in this paper we mean a fixed number that can be made as small as we want at the expense of increasing $m$ which will also remain fixed (i.e., it is not a sequence tending to $0$). 
\\

\noindent {\bf Matrix notation.} Given a matrix $M$ write either $M_{a,b}$ or $M(a,b)$ for the $a,b$ entry of $M$. For a symmetric matrix $M$ we denote its eigenvalues by $\{\lambda_i(M)\}_{i \geq 1}$ ordered in a descending order and by $\|M\|_F$ its 
Frobenius norm, i.e., 
$$ \|M\|_F^2 := \Tr(M^2) = \sum_{i,j} M_{i,j}^2 = \sum_{i \geq 1} \lambda_i(M)^2 \, .$$
We also write $\|M\|_\infty:= \max_{a.b} M_{a,b}$ for its supremum norm. \\

\noindent {\bf Percolation notation.} Given a graph $G$ and $p\in [0,1]$ we write $\proba_p$ and $\E_p$ for the probability measure and expectation operator obtained by retaining each edge independently with probability $p$ and erasing it otherwise. Edges retained are often called \emph{open} or \emph{$p$-open} and erased edges are called \emph{closed} or \emph{$p$-closed}. Throughout this paper we work with various different $p$'s and the standard \emph{simultaneous coupling} of the probability measures $\proba_p$ for all $p\in[0,1]$ is often used. In this coupling we independently draw for each edge $e$ a random variable $U_e$ drawn from the uniform distribution on $[0,1]$ and let $\omega_p=\{e : U_e \leq p\}$ for $p\in[0,1]$ denote coupled percolation configurations. We write $v \stackrel{\omega_p}{\lr} u$ if $v$ is connected to $u$ in a path of $p$-open edges, but mostly we write that event as $v \lr u$ dropping the $\omega_p$ superscript when it is clear what $p$ is from context. Similarly, given two disjoint vertex sets $A$ and $B$ we write $A \stackrel{\omega_p}{\lr} B$ or $A \lr B$ when there is an open path connecting some vertex of $A$ to some vertex of $B$. Given some vertex subset $A$ and an event $E$ we write $E \off A$ for the event that $E$ occurs after closing all the edges which have at least one endpoint in $A$. 

We write $\C_p(v)$ for the connected component containing $v$ in $\omega_p$ and frequently omit the subscript $p$ and just write $\C(v)$ when $p$ is clear from the context. The set of connected components of $\omega_p$ is denoted by $\comp_p$ and given $M\geq 1$ we denote $\comp_{p,M} \subset \comp_p$ the set of components of size at least $M$. We frequently use the fact that on any finite graph $\sum_v |\C_p(v)| = \sum_{\C \in \comp_p} |\C|^2$ since $|\C(v)|$ is counted $|\C(v)|$ times on the left. When $G$ is a transitive graph the distribution of $|\C(v)|$ is the same for all vertices $v$ and we define the \emph{susceptibility} as
\be\label{eq:defchi} \chi(p) := \E_p |\C(v)| = \sum_{u} \proba_p(v \lr u)\, .\ee
 Given two distinct components $a\neq b\in \comp_p$ we write $\Delta_{a,b}$ for number of edges having one endpoint at $a$ and the other at $b$; note that these edges must be $p$-closed by definition. We always set $\Delta_{a,a}=0$. We will often write $\Delta=\{\Delta_{a,b}\}_{a,b \in \comp_p}$ or $\Delta=\{\Delta_{a,b}\}_{a,b \in \comp_{p,M}}$ and view it as a symmetric matrix with vanishing diagonal.

We write $p_c=p_c(\Z^d)\in (0,1)$ for the critical percolation probability on $\Z^d$ and note that this number depends on $d$ and $L$ and we consider it to be a constant throughout this paper. Since we will occasionally compare between quantities, such as the susceptibility, on the infinite lattice and the discrete torus, we denote by $\C^\Z (0)$ the component of the origin in percolation on $\Z^d$ and by $\C^\T(0)$ the component of the origin in percolation on $\Z_n^d$ and write $\chi_\Z(p)=\E_p|\C^\Z(0)|$ and $\chi_\T(p)=\E_p|\C^\T(0)|$ for the corresponding susceptibilities. Since the torus is finite $\chi_\T(p)$ is an increasing polynomial in $p$ and we denote by $\chi'_\T(p)$ its derivative at $p\in[0,1]$. On $\Z^d$ it is a classical result \cite{Kesten81} that $\chi_\Z(p)$ is differentiable on $[0,p_c)$ and we denote by $\chi'_\Z(p)$ its derivative on that interval. Unless otherwise stated, the notation $\chi(p)$ means $\chi_\T(p)$. However, in sections of this paper that involves both, we will not use $\chi(p)$ rather $\chi_\T(p)$ and $\chi_\Z(p)$ to avoid confusion. In sections of this paper that \emph{only} deal with percolation on the infinite lattice, we drop the superscript/subscript $\Z$; in this case we will state this in the section's beginning. 

We ubiquitously use the van den Berg-Kesten inequality \cite{BK85} and its extension to non-monotone events by Reimer \cite{Reimer1}. To differentiate we call the first the BK inequality and the second the BKR inequality. It states that $\proba(A \circ B) \leq \proba(A)\proba(B)$ where $A \circ B$ is the disjoint occurrence of the events $A$ and $B$ and $\proba$ is any product measure on the edges of a graph. We refer the reader to the helpful survey \cite{Reimer2}. One of the useful applications of the BK inequality is known as Aizenman and Newmans' ``tree graph'' inequalities \cite{TreeGraph}, the first one stating that on any graph $G$ and vertices $x_1,x_2,x_3$
\be\label{eq:treegraph1}
\proba(x_1 \lr x_2 \lr x_3) \leq \sum_{v \in G} \proba(x_1 \lr v)\proba(x_2 \lr v) \proba(x_3 \lr v) \, ,
\ee
and the other inequalities bounding $\proba_p(x_1 \lr x_2 \cdots \lr x_n)$ for any integer $n\geq 3$, see also \cite[Section 6.3]{Grimm99}. A useful consequence is that for any transitive graph and any vertex $v$, if $\chi(p)<\infty$, then
\be\label{eq:treegraph2} \E_p |\C(v)|^k \leq C(k) \chi(p)^{2k-1} \, ,
\ee
see \cite[(6.94)]{Grimm99}. 
We frequently use in this paper the BK, BKR and \eqref{eq:treegraph1} in the setting of inhomogeneous percolation, that is, when the appearance of each edge in the graph is independent, but different edges may have different probabilities.

\subsection{About the proof} \label{sec:discussproof}
With the simultaneous coupling described above we may view the component sizes $(|a|)_{a\in \comp_p}$ with $p$ varying from $0$ to $1$ as a coalescent process. In this dynamics components $a$ and $b$ merge to a single component of size $|a|+|b|$ with rate $\Delta_{a,b}$. Indeed, when sprinkling additional $\eps>0$ edges components $a$ and $b$ coalesce with probability $1-(1-\eps)^{\Delta_{a,b}}$. When there is no geometry, i.e., when $G$ is the complete graph, then $\Delta_{a,b}=|a||b|$ and the corresponding process is known as the ``multiplicative coalescent'' introduced by Aldous \cite{Aldous97}. The mean-field heuristic therefore suggests the same behavior for the torus in high dimensions as long as $p$ is around the critical value. In other words, we expect that $\Delta_{a,b}$ is proportional to $|a||b|$. If one further imagines that connected components conditioned on their size behave like uniform random sets, then a natural guess is that $\Delta_{a,b}$ is of order $m|a||b|V^{-1}$. 

If $\Delta_{a,b}$ is precisely proportional to $|a||b|$ at $\omega_{p}$ for some $p$, then the results of Aldous  \cite{Aldous97} can be directly used to find where the scaling window is and to show convergence. The multiplicative coalescent described above is realized in \cite{Aldous97} by considering a inhomogeneous percolation model with vertex set $\{a\}_{a \in \comp_p}$ so that an edge $(a,b)$ appear independently with probability $1-e^{-q|a||b|}$ for some $q>0$. If one has good control over the second and third moments $\sum_a |a|^2$ and $\sum_a |a|^3$, then one can adjust the parameter $q>0$ so that the largest components in this inhomogeneous model converge to $\ZZ_\lambda$ for any desired $\lambda \in \R$. We call this model the {\em multiplicative component graph}, see \cref{sec:mcg}. 

However, this is of course too simplistic. To obtain \cref{thm:main_pc} we perform a coupling between the multiplicative component graph and the actual coalescence of components from the subcritical phase to the critical one. Several realistic considerations become viable at this point.  Firstly, since our goal is obtaining scaling limits, we must also find the right constant $\mathfrak{c}>0$ so that $\Delta_{a,b} = (\mathfrak{c} + o(1))m|a||b|V^{-1}$. Secondly, we must discard small components since we cannot hope that our guess on $\Delta_{a,b}$ will be true for all $a,b$. It turns out that when $p$ is slightly subcritical then we may discard components of order smaller than $\chi(p)^4/V^{2/3}$ since the mass contributed from them to the largest components at $p_c$ is $o(V^{2/3})$ and even more importantly, they cannot ``break'' the largest $p_c$-components into macroscopic pieces  (i.e., pieces of size $\Theta(V^{2/3})$). Lastly, we are unfortunately unable to control $\Delta_{a,b}-\mathfrak{c}m|a||b|V^{-1}$ uniformly over (large) $a,b$, only in a weaker $\ell_2$ sense. All this leads us to the main quantity we bound in this paper:
\be\label{eq:defW1} \W(p ; M) := \inf _{\cs>0} \sum_{A\neq B \in \comp_{p,M}} \left ( \Delta_{A,B} - {\cs m |A||B| \over V} \right )^2 \, ,\ee
where $p$ is in the slightly subcritical regime, but very close to critical, and $M$ is a sequence that is $o(\chi(p)^4/V^{2/3})$. It will become evident (see \cref{sec:newsubcritical}) that the summation in \eqref{eq:defW1} over each of the three terms obtained by opening the parenthesis is of order $\chi(p)^2$. Our main effort is to show that cancellations occur so that in fact $\W(p; M)=o(\chi(p)^2)$ and to obtain a good quantitative bound. 

\begin{theorem} \label{thm:main_goal} Consider the discrete torus $\Z_n^d$ in high dimensions. There exists $\eta_0>0$ such that for all $\eta\in (0,\eta_0)$ setting $p = p_c(\Z^d) - V^{-1/3+\eta}$ gives that with probability $1-o(1)$ as $n\to \infty$ the percolation configuration $\omega_{p}$ satisfies
\be\label{eq:maingoal} \W(p ; V^\mp) \leq \X(p)^2 (\X(p)/V^{1/3})^{{16}}  \, .\ee
\end{theorem}
\begin{remark}
The constants $16$ and $\mp$ are uninteresting and are chosen for convenience.
The constant $16$ can be replaced by any other constant at the expense of taking smaller $\eta_0$ (see \cref{sec:Contraction}) and $\mp$ is chosen so that it is in $(1/2,2/3)$ but for technical reasons not any number in that interval can be chosen.
\end{remark}

Given \cref{thm:main_goal} and the estimates controlling $\sum_a |a|^2$ and $\sum_a |a|^3$ given in \cref{sec:standard}, we perform the coupling between the two inhomogeneous percolation models (the multiplicative and the ``real'' model) in \cref{sec:theCoupling}. It is similar to the one performed in our recent work \cite{HypercubePaper}.

Therefore, the most novel part of the paper and the bulk of the work is the proof of \cref{thm:main_goal} which is presented in \cref{sec:Contraction} (and heavily relies on the estimates of \cref{sec:subcriticalEstimates}). We combine spectral methods together with high-dimensional percolation techniques to show that the symmetric matrix $\Delta$ approaches a matrix of rank $1$ as $p$ gets closer to the critical window from below. In particular, that its top eigenvalue gets very large and all other eigenvalues are very small compared to the top one. We perform this in several rounds of exposure in the subcritical phase, at locations that get closer and closer to the critical window. In each round, between say $p_1$ and $p_2>p_1$, we improve our estimate on $\W$. Our main new idea is that the spectrum of the symmetric matrix $\Delta$  seen at $\omega_{p_1}$ can be used to give a delicate choice of $p_2>p_1$ so that the distance of $\Delta$ seen at $\omega_{p_2}$ from a rank $1$ matrix contracts significantly compared to at $\omega_{p_1}$. Furthermore, when can use the spectrum to extract an implicit choice of $\cs>0$, that is used in \eqref{eq:defW1}, in which a crucial cancellation occurs allowing us to upper bound $\W(p_2)$. This approach is explained in detail in \cref{sec:Contraction}.

\subsubsection{The $k$-arm $r$-IIC}\label{sec:iicintro}

As mentioned earlier, sufficient control over $\sum_a |a|^2$ and $\sum_a |a|^3$ in the subcritical phase is required already at the level of the aforementioned multiplicative component graph. I.e., it is independent of the coupling and \cref{thm:main_goal}. A precise estimate on $\sum_a |a|^2$ can be obtained by standard methods, see the first assertion of \cref{lem:ConcentrationFixedp2} and its proof. However, a precise estimate of $\sum_a |a|^3$, as in the second assertion of \cref{lem:ConcentrationFixedp2}, is considerably more difficult since it involves the unknown constant $\Cnst_2$. For this and other reasons we had to construct a novel extension to an object known as the  \emph{incipient infinite cluster} (IIC) for critical percolation in high dimensions. This object, constructed by van der Hofstad and J\'arai \cite{HofstadJaraiIIC} is a measure $\qiic$ on connected subgraphs of $\Z^d$ containing the origin satisfying 
\be\label{eq:iic} \qiic(F) = \lim _{p \uparrow p_c} {1 \over \chi(p)} \sum_{x\in \Z^d} \proba_p(F, 0 \lr x) \, ,\ee
for any cylinder event $F$. We remark that in \cite{HofstadJaraiIIC} the authors obtain other equivalent constructions which are relevant for the percolation at $p_c$ (rather than at $p<p_c$ as in \eqref{eq:iic}) but the proof strategy in this paper relies heavily on the slightly subcritical regime so \eqref{eq:iic} is most fitting for us. 

In \cref{thm:kArm-rIIC} in \cref{sec:IICs} we extend the construction of \eqref{eq:iic} which allows us to obtain the constants $\Cnst_1,\Cnst_2$ in \cref{thm:sharpChiChi2Zd} which in turn determine the constants $\Cnst,\Cnst'$ in \cref{thm:main_pc}. Our extension involves three additional properties of the IIC:
\begin{enumerate}
    \item Allow for $k$ disjoint arms (instead of just one). 
    \item For an arbitrary integer $r\geq 0$, require that the arms emanating from $x_1, \ldots, x_k$ are connected to the box $B_r=\{-r,\ldots, r\}^d$ around $0$, instead of being connected to $0$.
    \item Require that the local configuration $F$ ``remembers'' which part of it is connected to each of the $k$ arms represented by $x_1,\ldots,x_k$.
\end{enumerate}
We remark that in a recent paper \cite{BiIIC} the authors construct a bi-infinite extension of the IIC, similar to the case $k=2$ and $r=0$ of the above, only  corresponding to percolation at $p_c$ (rather than $p<p_c$ as in \cref{thm:kArm-rIIC} here). See also \cref{rmk:prevIIC} for further discussion. Our method is completely different and avoids the lace expansion.  This is further developed in work in preparation by the first author and Hutchcroft.

\subsubsection{Comparison with the recent work \cite{HypercubePaper}.} \label{sec:CompareHypercube}
In \cite[Theorem 1.1]{HypercubePaper} the current authors and Broutin prove the analogous statement of \cref{thm:main} for critical percolation on the hypercube $\{0,1\}^m$ as $m \to \infty$ as well as other ``high-dimensional'' graphs, see \cite[Theorem 1.3]{HypercubePaper}. In fact, the high level heuristic described in the paragraphs above \cref{thm:main_goal} applies to the hypercube case. Indeed, in \cite[Proposition 4.7]{HypercubePaper} the authors show that $\E \W(p)=o(\chi(p)^2)$ in the hypercube for some subcritical $p$ very close to the scaling window. However, this estimate crucially hinges on the fact that $\cs=1$ in \eqref{eq:defW1} yields $o(\chi(p)^2)$ in expectation and the minimizer of \eqref{eq:defW1} is $1+o(1)$ with high probability.  This in turn relies on two additional facts: (a) the finite triangle diagram is $o(1)$ at criticality (see \cite[(15)]{HypercubePaper}) and (b) the critical percolation probability $p_c$ is $(1+o(1))m^{-1}$ as the degree of the hypercube $m$ tends to $\infty$. Due to these one can open the parenthesis in \eqref{eq:defW1} with $t=1$ and estimate rather precisely each of the three terms and observe the cancellation leading to the $o(\chi(p)^2)$ bound (see Lemmas 4.8, 4.9 and 4.10 in \cite{HypercubePaper}).  

Both (a) and (b) above fail for the torus and $t=1$ is an uninteresting point of \eqref{eq:defW1}. The degree $m$ is a constant and the critical probability is uniformly bounded away from $0$. Furthermore, the triangle diagram is small as either $d\to \infty$ or $L \to \infty$, but since the values of $d$ and $L$ are fixed, the diagram is not $o(1)$. The approach described below \cref{thm:main_goal} is used to overcome this difficulty. Also, the construction of the IICs in \cref{sec:IICs} is necessary since, unlike the torus, the corresponding constants $\Cnst,\Cnst'$ of \cref{thm:main_pc} in the hypercube are simply $1$. Thus, the only similarities between the proofs of \cite{HypercubePaper} and the current paper are contained in  \cref{sec:theCoupling}.

\subsection{Subcritical building blocks} \label{sec:buildingblocks}

The bulk of the proofs require sharp estimates of slightly subcritical regime. We collect them here.

\subsubsection{Past results}\label{sec:pastresults}

Two important results of \cite{BCHSS1,BCHSS2, HeydenHofstad1,HeydenHofstad2,HMS23} are instrumental in this paper. The first is the following.

\begin{theorem} (\cite[Theorem 1.2]{BCHSS1} combined with \cite[Theorem 2.1]{HeydenHofstad2}) \label{thm:subcriticalBCHSS} Consider percolation on the discrete torus $\Z_n^d$ in high dimensions and suppose that $p=p_c(\Z^d)(1-\eps(n))$ where $\eps(n)$ is a nonnegative sequence with $\eps(n)=o(1)$ and $\eps(n) V^{1/3} \to \infty$. Then 
$$ \chi_\T(p) \asymp \eps(n)^{-1} \, ,$$
and with probability $1-o(1)$
$$ \max_{a \in \comp_{p}} |a| \leq 2\chi_\T(p)^2 \log(V/\chi(p)^3) \, .$$
\end{theorem}

In fact, in \cref{lem:sharpChiChi2Torus} below we will improve the first assertion of this theorem and obtain a $1+o(1)$ estimate.
The second estimate we abundantly use is a recent result by Hutchcroft, Michta and Slade \cite{HMS23} bounding the two point function on the torus.%
\begin{theorem}\cite[Theorem 1.3]{HMS23}\label{thm:plateau} Consider percolation on the discrete torus $\Z_n^d$ in high dimensions at $p \leq p_c(\Z^d)$. For any two vertices $x,y\in \Z_n^d$ we have 
$$\proba_p(x \lr y) \pe \langle x-y \rangle^{2-d} + { \chi_\T(p) \over V} \, .$$
\end{theorem}
We remark that several other results from \cite{BCHSS1,BCHSS2} will be used in various locations in the paper. We will cite them as needed.

\subsubsection{Sharp asymptotics of the susceptibility and expected cluster size squared} \label{sec:sharpChiChi2}

The constants $\Cnst, \Cnst'$ appearing in \cref{thm:main_pc} are determined by sharp estimates on the susceptibility and on the expected cluster size square on the torus, similar to those of \cref{thm:sharpChiChi2Zd}. Furthermore, we will also need careful estimates comparing the derivative of the susceptibilities on the torus $\Z_n^d$ and the infinite lattice $\Z^d$.

\begin{corollary}\label{lem:sharpChiChi2Torus} Suppose that $p=p_c(\Z^d) - \eps(n)$ where $\eps(n)$ is a positive sequence with $\eps(n)=o(1)$ and $\eps(n)V^{1/3} \to \infty$. Then 
\be\label{eq:sharpChi} \chi_\T(p) = \Big (1-O(V^{-1}\eps(n)^{-3}) \Big )\chi_\Z(p)= (\Cnst_1+o(1)) \eps(n)^{-1} \, ,\ee
and
\be\label{eq:sharpChi2}  \E_p|\C^\T(v)|^2 = (\Cnst_2 + o(1)) \chi_\T(p)^3 \, ,\ee
and
\be\label{eq:sharpChi'Chi2} \chi_\T'(p)/\chi_\T(p)^2 = (\Cnst_1+o(1))^{-1}\, ,\ee
and
\be\label{eq:Chi'TvsZ} \chi_\T'(p) = \Big (1-O(V^{-{1/2}}\eps(n)^{-3/2}) \Big )\chi_\Z'(p) \ee
where $\Cnst_1,\Cnst_2$ are the constants from \cref{thm:sharpChiChi2Zd}.    
\end{corollary}
\begin{lemma}\label{lem:chi''} Suppose that $p=p_c(\Z^d) - \eps(n)$ where $\eps(n)$ is a positive sequence with $\eps(n)=o(1)$ and $\eps(n) \gg V^{-1/3}$. Then
$$ |\chi_\T''(p)| \pe \eps(n)^{-3} \, , \qquad |\chi_\Z''(p)| \pe \eps(n)^{-3} \, .$$
\end{lemma}

\subsubsection{Second and third moments of subcritical component sizes} \label{sec:standard}

As discussed in \cref{sec:discussproof}, the ``up to $1+o(1)$'' estimates on the sum of squares and cubes of components sizes in the subcritical regime have an important role in the proof. Since $\sum_{A \in \comp_{p}}|A|^k = \sum_{v} |\C(v)|^{k-1}$ (as each component is counted $|\C(v)|$ times on the right hand side), the desired estimate follow from \cref{lem:sharpChiChi2Torus} by standard second moment arguments. We also show that discarding the small components does not affect these sums. These estimates are recorded here.

\begin{lemma}\label{lem:ConcentrationFixedp2} Suppose that $p=p_c(\Z^d)(1-\eps(n))$ where $\eps(n)$ is a nonnegative sequence with $\eps(n)=o(1)$ and $\eps(n) V^{-1/3} \to \infty$ and that $M=M_n=o(\chi_\T(p)^2)$. Then
$$ \sum_{A \in \comp_{p, M}} |A|^2 = V\chi(p) + \Op\Big (V^{1/2}\chi(p)^{5/2} + V\sqrt{M} \Big ) \, ,$$
and with probability $1-o(1)$
$$ \sum_{A \in \comp_{p, M}} |A|^3 = (\Cnst_2+o(1))  V \chi(p)^3 \, ,$$
where $\Cnst_2$ is the constant from \cref{thm:sharpChiChi2Zd}.
\end{lemma}

Our last technical estimate here will also be useful for discarding the small components. Given two edge probabilities $p_2>p_1$ and a number $M$, let $N(p_1,p_2; M)$ denote the set of pairs of vertices $\{u,v\}$ such that $u$ is connected to $v$ in $\omega_{p_2}$ but every path connecting $u$ to $v$ in $\omega_{p_2}$ visits a component of $\omega_{p_1}$ that has size smaller than $M$. 
\begin{lemma} \label{lem:Np1p2Bound} Assume that $p_1< p_2$ and $p_1\leq p_c(0)$. Then 
$$ \E |N(p_1,p_2; M)| \pe \Big (1+\chi(p_2)(p_2-p_1)\Big )^2 V \sqrt{M} \, .$$
\end{lemma}

\subsubsection{Subcritical estimates involving $\Delta$}\label{sec:newsubcritical}

In all of the following theorems we consider percolation on the discrete torus $\Z_n^d$ in high dimensions. Recall that $\Delta_{a,b}$ for distinct components $a,b\in \comp_p$ is the number of edges with one endpoint in $a$ and the other in $b$ (by definition these edges are $p$-closed). We begin with two large deviation estimates of the maximum entry of $\Delta$ and the maximal row sum. 

\begin{theorem} \label{thm:MaxDelta} Suppose that $p\leq p_c(\Z^d)$ and that $\chi(p) \geq V^{0.3}$. Then with probability at least $1-o(V^{-100})$ we have
\[ \|\Delta\|_\infty = \max_{a,b \in \comp_p} \Delta_{a,b} \leq {\log(V)^{3} \chi(p)^4 \over V} \, . \]
\end{theorem}

\begin{theorem} \label{thm:MaxSumDelta} Suppose that $p\leq p_c(\Z^d)$, that $\chi(p) \geq V^{0.3}$ and $M=V^{\mp}$. Then with probability at least $1-o(V^{-100})$ we have
\[ \max_{a \in \comp_p} \sum_{b\in \comp_{p,M}} \Delta_{a,b}^2 \leq {\log(V)^{10}\chi(p)^5 \over V} \, . \]
\end{theorem}

\begin{remark} We will use the last two estimates only in the subcritical regime close to the critical window; since we are using it to prove \cref{thm:main_goal}. But we remark that, with minor changes, the proof works at the entire critical window, so the two theorems hold also when $p=p_c(\Z^d)+O(V^{-1/3})$. The assumption that $\chi(p)\geq V^{0.3}$, meaning that $p$ cannot be too subcritical, can likely be dropped but we keep it, as it relaxes some of the calculations.
\end{remark}

For the following four assertions we really do need to be in the subcritical regime but close enough to the critical window (even closer than previous theorems). 

\begin{lemma}\label{lem:cmin} Suppose that $p=p_c(\Z^d)(1-\eps(n))$ where $\eps(n)$ is a nonnegative sequence with $\eps(n)=o(1)$ and $\eps(n) V^{-1/3} \to \infty$. Suppose also that $M=\chi(p)^5/V$ and denote by $\cmin$ the minimizer of
$$ \W(p; M) = \inf_{\cs>0} \sum_{A\neq B \in \comp_{p,M}} \left ( \Delta_{A,B} - {\cs m |A||B| \over V} \right )^2 $$
defined in \eqref{eq:defW1}. Then with probability $1-o(1)$ we have
$$ \cmin = (1+ o(\X(p)V^{-1/3})) {(1-p)\chi'(p) \over m \chi(p)^2} \, .$$
\end{lemma}

\begin{theorem}\label{thm:abDeltaAlternative} Suppose that $p=p_c(\Z^d)-\eps(n)$ where $\eps(n)$ is a nonnegative sequence with $\eps(n)=o(1)$ and $\eps(n) V^{1/3} \log(V)^{-10} \to \infty$. Suppose also that $M=V^{\mp}$ and that $\chi(p) \gg V^{3/10}$. Then we have 
\[ (1-o_m(1))m\X(p)^2 \leq \X'(p)\leq m\X(p)^2 \, ,\] and also,
$$
\sum_{\substack{a \neq b \in \comp_{p,M}}}|a||b|\Delta_{a,b}= (1-p)\X'(p)V+o_{\proba}(V^{2/3}\X(p)^3) \, .
$$
\end{theorem}

\begin{theorem}\label{thm:Delta2} Suppose that $p=p_c(\Z^d)-\eps(n)$ where $\eps(n)$ is a nonnegative sequence with $\eps(n)=o(1)$ and $\eps(n) V^{1/3} \log(V)^{-10} \to \infty$. Suppose also that $M=V^{\mp}$ and that $\chi(p)\gg V^{0.33}$. Then for some constant $C=C(d,L)\in (0,\infty)$ we have that with probability $1-o(1)$
$$ %
\|\Delta\|_F^2 = \sum_{a,b \in \comp_{p,M}} \Delta_{a,b}^2 \leq C(d,L) \chi(p)^2 \, .
$$    
\end{theorem} %

\begin{theorem} \label{thm:Delta4} Suppose that $p=p_c(\Z^d)-\eps(n)$ where $\eps(n)$ is a nonnegative sequence with $\eps(n)=o(1)$ and $\eps(n) V^{1/3} \log(V)^{-10} \to \infty$. Suppose also that $M=V^{\mp}$ and that $\chi(p)\gg V^{3/10}$. Then with probability $1-o(1)$ we have 
$$\Tr(\Delta^4) \leq (1 + o_m(1)) m^4\X(p_1)^4 \, .$$
\end{theorem}

\subsection{Organization} We present the arguments leading to the proofs of \cref{thm:main_pc,thm:main} and \cref{cor:supercritical} in \cref{sec:proofOfMainThm}. These arguments assume \cref{thm:main_goal} and some of the estimates given in \cref{sec:buildingblocks}. We then proceed in \cref{sec:Contraction} to present the proof of \cref{thm:main_goal} given the estimates of \cref{sec:buildingblocks}. As mentioned earlier, this is the most novel part of the paper so we opted to present it first even though the results of \cref{sec:buildingblocks} logically precede it. In \cref{sec:IICs} we construct the new IICs described in \cref{sec:iicintro} and use them to prove \cref{thm:sharpChiChi2Zd} and \cref{lem:sharpChiChi2Torus}; the proofs in that section do not rely on \cref{sec:buildingblocks} and can be read independently from the rest of the paper. Finally, in \cref{sec:subcriticalEstimates} we provide the proofs of the new estimates described in \cref{sec:standard,sec:newsubcritical}.

\section{Proof of the main theorems given \cref{thm:main_goal} and \cref{sec:buildingblocks}} \label{sec:proofOfMainThm}

As described earlier, we perform our first round of exposure at the slightly subcritical phase and next at a precise location in the critical phase determined by $\lambda\in \R$ as in \cref{thm:main_pc}. It is not clear how to pinpoint the precise location of $p_c(\Z^d)$ or $p_c(\lambda)$ (defined in \eqref{eq:defpc}) in the scaling window since they are defined indirectly. So the proof will go by setting a ``direct'' edge probability, showing its convergence to $\ZZ_\lambda$ and only then relating it to $p_c(\Z^d)$ and $p_c(\lambda)$. In this section we use $\chi(p)$ to denote the torus susceptibility $\chi_\T(p)$; the notation $\chi_\Z(p)$ will be used as usual for the susceptibility on $\Z^d$.

Let $\eta_0>0$ be the constant from \cref{thm:main_goal}, let $\eta \in (0,\eta_0)$ and $\{\beta_n\}_{n\geq 1}$ be any sequence with $\beta_n\to 0$ and lastly fix $\lambda\in \R$. Set
\be\label{def:ps} p_s = p_s(\eta,n) := p_c(\Z^d)-V^{-1/3+\eta} \, ,\ee
and define $p_c'(\lambda)=p_c'(\lambda, \eta, n, \beta_n)>p_s$ by
\be\label{eq:Newpc'def}  p_c'(\lambda) := p_s + {\X(p_s) \over \X'(p_s)} + {\Cnst_1 \Cnst_2^{2/3}} \lambda V^{-1/3} + \beta_n V^{-1/3} \, ,\ee
where $\Cnst_1,\Cnst_2$ are from \cref{thm:sharpChiChi2Zd}.

The next theorem, which is the core of the argument of this section, shows that at $p_c'(\lambda)$ we have convergence to $\ZZ_\lambda$. Afterwards, in \cref{sec:main_pc}, we will prove that 
$$ p_c'(\lambda) - p_c(\Z^d) - {\Cnst_1 \Cnst_2^{2/3}}  \lambda V^{-1/3}= o(V^{-1/3})\, ,$$ 
and complete the proof of \cref{thm:main_pc}. It is interesting to note here, that convergence to $\ZZ_\lambda$ at $p_c'(\lambda)$ implies that taking different value of $\eta\in(0,\eta_0)$ only changes the value of $p'_c(\lambda)$ by $o(V^{-1/3})$. 

\begin{theorem}\label{thm:convergencepc'} For any $\lambda \in \R$ any $\eta\in(0,\eta_0)$ and any sequence $\beta_n\to 0$ denote by $(\C_1,\C_2,\ldots)$ the connected components, sorted by their size in a descending order, of percolation on the discrete torus $\Z_n^d$ in high dimensions at $p=p'_c(\lambda)=p'_c(\lambda,\eta,n,\beta_n)$ defined in \eqref{eq:Newpc'def}. Then 
\be\label{eq:convergencepc'} \Cnst_2^{1/3} \cdot V^{-2/3}(|\C_1|, |\C_2|,  \ldots) \underset{\ell_2}{\overset{(d)}{\longrightarrow}}  \ZZ_{\lambda} \, ,\ee
and 
\be\label{eq:ConvergenceChipc'} \X(p_c'(\lambda)) = (1+o(1)) \Cnst_2^{-2/3}\kappa(\lambda) V^{1/3} \, .\ee
\end{theorem}

\subsection{Proof of \cref{thm:convergencepc'}}\label{sec:theCoupling}

We begin by invoking \cref{thm:main_goal} and obtain that with probability $1-o(1)$ the configuration $\omega_{p_s}$ (where $p_s$ is defined in \eqref{def:ps}) satisfies
$$ \W(p_s; V^\mp) \leq  \X(p_s)^2 (\X(p_s)/V^{1/3})^{{16}}  \, ,$$
(recall that $\W(p_s; M)$ was defined in \eqref{eq:defW1}). Next we set 
\be\label{eq:ms} M_s = \X(p_s)^5/V \, ,\ee
and note that since $\X(p_s) \asymp V^{1/3-\eta}$ by \cref{thm:subcriticalBCHSS} and since $\eta_0$ is small enough we have that $M_s \geq V^{\mp}$. Since $M \mapsto \W(p; M)$ is a non-increasing map we obtain that 
\be \W(p_s; M_s) \leq  \X(p_s)^2 (\X(p_s)/V^{1/3})^{{16}}  \, . \label{eq:main1} \ee

Next we invoke some of the building blocks of \cref{sec:buildingblocks}. We define the following $\omega_{p_s}$-measurable events:
\begin{align}
    \sum_{A \in \comp_{p_s,M_s}} |A|^2 & = V \X(p_s) + o(V^{2/3} \X(p_s)^2) \, , \label{eq:SumSquares} \\
     \sum_{A \in \comp_{p_s,M_s}} |A|^3 &= (\Cnst_2+o(1))  V \X(p_s)^3 \, , \label{eq:SumCubes} \\ 
     \max _{A \in \comp_{p_s}} |A| & \leq 2  \X(p_s)^2 \log V \, ,\label{eq:MaxComp} \\
     |\comp_{p_s, M_s}| &\leq \log(V) V^{5/2}/\X(p)^{15/2}  \, , \label{eq:CompNum}  \\
    \cmin &= (1+ o(\X(p_s)V^{-1/3})) {(1-p_s)\chi'(p_s) \over m\chi(p_s)^2} \, . \label{eq:cnsttConcentrated} 
\end{align}
where $\Cnst_2=\Cnst_2(d,L)\in (0,\infty)$ is the constant from \cref{lem:sharpChiChi2Torus} and $\cmin=\cmin(\omega_{p_s})$ is the minimizer of $\W(p_s; M_s)$. To see that \eqref{eq:SumSquares}---\eqref{eq:cnsttConcentrated} hold with probability $1-o(1)$ first observe that  \eqref{eq:SumSquares} and \eqref{eq:SumCubes} follow from \cref{lem:ConcentrationFixedp2}, noting that both errors there concerning the sum of squares, i.e., $V^{1/2}\chi(p)^{5/2}$ and $V\sqrt{M_s}$ are $o(V^{2/3}\chi(p)^2)$ since $\X(p_s)=o(V^{1/3})$ %
and our choice of $M_s$. Secondly, \cref{thm:subcriticalBCHSS} immediately handles \eqref{eq:MaxComp} and \cref{lem:cmin} handles  \eqref{eq:cnsttConcentrated}. Lastly,  we note that 
\[ \E[|\comp_{p_s, M_s}|] \leq \frac{1}{M_s} \sum_{u}\proba(|\C_u|\geq M_s) \leq \frac{V}{M_s^{3/2}} = \frac{V^{5/2}}{\X(p)^{15/2}} \, , \] 
from which \eqref{eq:CompNum} holds with probability $1-o(1)$ using Markov's inequality. Thus, the event that $\omega_{p_s}$ satisfies \eqref{eq:main1}---\eqref{eq:cnsttConcentrated} above is $1-o(1)$. In what follows we denote the  connected components $\comp_{p_s,M_s}$ by $\{A,B,C,\ldots\}$. The proof will proceed by considering \emph{two} random graphs with vertex set $\comp_{p_s,M_s}$, the \emph{multiplicative component graph} and the \emph{sprinkled component graph}.

\subsubsection{Convergence of the random multiplicative component graph}\label{sec:mcg} We assign to each component $A\in \comp_{p_s,M_s}$ a weight
$$ w_A := {|A| \Cnst_2^{1/3} \over V^{2/3}} \, ,$$
and set a random parameter $q_\lambda=q_\lambda(\omega_{p_s})$ defined by 
\be\label{eq:defQlambda}  q_\lambda ={m\X(p_s)^2 \cmin \over (1-p_s)\X'(p_s)} \Big( \Cnst_2^{-2/3} V^{1/3}/\X(p_s)
+ \lambda + o(1) \Big ) \, ,\ee
where the term $o(1)$ is an arbitrary deterministic sequence (i.e., it does not depend on the configuration $\omega_{p_s}$) 
which tends to $0$ as $n\to \infty$; it will be carefully chosen in the next section, see \cref{clm:pcAsympt}.

 \begin{definition}\label{def:GTimes} The \emph{multiplicative component graph} $G^\times$ is a random undirected graph with vertex set $\comp_{p_s,M_s}$ so that, each edge $(A,B)$ appears independently with probability 
\be\label{eq:defqab} q_{A,B} := 1-e^{-q_\lambda w_A w_B} \, .\ee
The {\bf weight} of a component $\C^\times$ in $G^\times$ is the sum of weights 
$$ \weight{\C^\times} = \sum_{A \in \C^\times} w_A = {\Cnst_2^{1/3} \over V^{2/3}}\sum_{A \in \C^\times} |A| \, .$$
Denote by $\{\C_j^\times\}_{j=1}^\infty$ the connected components of $G^\times$ sorted according to their weights in a descending order.
\end{definition}

We further define $\sigma_j=\sum_{A \in \comp_{p_s, M_s}} w_A^2$ for $j=2,3$. If the event \eqref{eq:SumSquares} holds, then 
$$ \sigma_2 = \Cnst_2^{2/3} V^{-4/3} \sum_{A\in \comp_{p_s, M_s}} |A|^2 = \big ( 1 + o(\chi(p_s)/V^{1/3}) \big ) \Cnst_2^{2/3} V^{-1/3} \X(p_s) \, ,$$
so that $\sigma_2^{-1} = \Cnst_2^{-2/3}V^{1/3}/\X(p_s) + o(1)$. Hence, if in addition \eqref{eq:cnsttConcentrated} holds, by plugging in the value of $q_\lambda$ in \eqref{eq:defQlambda} we get that $q_\lambda - \sigma_2^{-1} \to \lambda$.
Next, if \eqref{eq:SumCubes} holds, then
$$ \sigma_3  = \Cnst_2 V^{-2} \sum_{A\in \comp_{p_s, M_s}} |A|^3 = (1+o(1)) \Cnst_2^2 V^{-1} \X(p_s)^3 \, .$$
We obtain that \eqref{eq:SumSquares}, \eqref{eq:SumCubes},  \eqref{eq:MaxComp} and \eqref{eq:cnsttConcentrated} imply three consequences:  
\[ \sigma_3/(\sigma_2)^3 \to 1 \quad \quad ; \quad \quad q_\lambda - \sigma_2^{-1} \to \lambda, \quad\quad ; \quad\quad  \max_{A} \sigma_{2}^{-1} w_A \to 0 \, , \] 
where the third is due to \eqref{eq:MaxComp} and since $\chi(p_s) \asymp V^{1/3-\eta} = o(V^{1/3}/\log V)$ by \cref{thm:subcriticalBCHSS}. We now appeal to a celebrated result of Aldous \cite[Proposition 4]{Aldous97} which states that these three  consequences imply that 
\be\label{eq:GTimesConverges}
(\weight{\C_1^\times}, \weight{\C_2^\times}, \ldots) \underset{\ell_2}{\overset{(d)}{\longrightarrow}}  \ZZ_{\lambda} \qquad \textrm{as } n\to \infty \, ,\ee
where $\ZZ_{\lambda}$ is the list of lengths of excursions above past minimum of an inhomogeneous Brownian motion started at $0$ with drift $\lambda-t$ at time $t\in[0,\infty)$ sorted in descending order. Since \eqref{eq:SumSquares}, \eqref{eq:SumCubes},  \eqref{eq:MaxComp} and \eqref{eq:cnsttConcentrated} occur with probability $1-o(1)$ we deduce that \eqref{eq:GTimesConverges} holds.

\subsubsection{Convergence of the sprinkled component graph}
The convergence \eqref{eq:GTimesConverges} implies nothing about percolation at $p_c'(\lambda)$; note that information about the edges between two components $A$ and $B$ has not been used yet, that is, \eqref{eq:main1} was not invoked.%

\begin{claim}\label{clm:pcAsympt} One can choose the $o(1)$ term in the definition of $q_\lambda$ in \eqref{eq:defQlambda} such that 
\be\label{def:pcPrime}\log \Big ( {1-p_c'(\lambda) \over 1-p_s} \Big ) = -{q_\lambda \Cnst_2^{2/3} \over \cmin mV^{1/3}} \, ,\ee
\end{claim}

\begin{remark} Note that $\cmin$ in the denominator cancels on the right hand side with the one in the numerator of $q_\lambda$, defined in \eqref{eq:defQlambda}, hence, both sides of the equation above are real numbers and do not depend on $\omega_{p_s}$ (unlike the random variable $q_\lambda(\omega_{p_s})$ which depends on $\omega_{p_s}$ via $\cmin$).    
\end{remark}

\begin{proof} We have that  \eqref{def:pcPrime} is equivalent to 
$$ p_c'(\lambda) = 1-(1-p_s)\exp \Big (-{q_\lambda \Cnst_2^{2/3} \over \cmin mV^{1/3}} \Big ) \, .$$
By continuity, it suffices to show that for any fixed $\eps>0$ there exists $n_0$ such that for all $n\geq n_0$
\be\label{eq:pcPrimeEquiv} p_c'(\lambda) \leq  1-(1-p_s)\exp \Big (-{q_{\lambda+\eps} \Cnst_2^{2/3} \over \cmin mV^{1/3}} \Big ) \, ,\ee
and 
$$ p_c'(\lambda) \geq  1-(1-p_s)\exp \Big (-{q_{\lambda-\eps} \Cnst_2^{2/3} \over \cmin mV^{1/3}} \Big ) \, .$$
Let us prove just the first inequality, the second is similar. By \eqref{eq:defQlambda} we have
$$ {q_{\lambda+\eps} (1-p_s) \over m \cmin} = {\Cnst_2^{-2/3} V^{1/3} \chi(p_s) \over \chi'(p_s)} +  { \chi(p_s)^2 \over \chi'(p_s)}(\lambda + \eps) + o(1) {\chi(p_s)^2 \over \chi'(p_s)} \, .$$
by \eqref{eq:sharpChi'Chi2} of \cref{lem:sharpChiChi2Torus}, we get 
$$ {q_{\lambda+\eps} (1-p_s) \over m \cmin} = {\Cnst_2^{-2/3} V^{1/3} \chi(p_s) \over \chi'(p_s)} +  \Cnst_1(\lambda+\eps+o(1))\, ,$$
By rearranging this, we get that \eqref{eq:pcPrimeEquiv} is equivalent to
\begin{eqnarray*} p_c'(\lambda) \leq &p_s& + {\chi(p_s) \over \chi'(p_s)} + \Cnst_1 \Cnst_2^{2/3} (\lambda+\eps+o(1)) V^{-1/3}  \\ &-& (1-p_s)  \Big [\exp \Big (-{q_{\lambda+\eps} \Cnst_2^{2/3} \over \cmin mV^{1/3}} \Big ) - \Big (1 -  {q_{\lambda+\eps} \Cnst_2^{2/3} \over \cmin mV^{1/3}}  \Big )\Big ] \, . \end{eqnarray*}

We observe that the last term above is $o(V^{-1/3})$. Indeed, the term in the exponent is of order $1/\chi(p_s)$ and by $e^{-x} = 1-x+O(x^2)$ as $x\to 0$, the last term is of order $\X(p_s)^{-2}\asymp V^{-2/3+2\eta}=o(V^{-1/3})$ as long as $\eta\in (0,1/6)$.
We also note that all the terms are deterministic and do not depend on $\omega_{p_s}$ since $\cmin$ in the denominator is canceled with the $\cmin$ in the numerator at \eqref{eq:defQlambda}. Therefore, since $\beta_n=o(1)$ in the definition \eqref{eq:Newpc'def} of $p_c'(\lambda)$, we get \eqref{eq:pcPrimeEquiv} holds for large enough $n$.
\end{proof}

Given \cref{clm:pcAsympt} from now on we assume that we have set the $o(1)$ term in $q_\lambda$ so that \eqref{eq:defQlambda} holds.

\begin{definition}\label{def:GComp} Given $\omega_{p_s}$ satisfying \eqref{eq:main1}---\eqref{eq:cnsttConcentrated}, the \emph{sprinkled component graph} $G^\comp$ is a random undirected graph with vertex set $\comp_{p_s, M_s}$ so that each edge $(A,B)$ appears independently with probability
$$ p_{A,B} := 1 - \Big ( {1-p_c'(\lambda) \over 1-p_s} \Big )^{\Delta_{A,B}} = 1- \exp\Big ({-q_\lambda \Cnst_2^{2/3} \Delta_{A,B} \over \cmin mV^{1/3}} \Big ) \, .$$
\end{definition}

Since $x\mapsto e^{-x}$ is $1$-Lipschitz on $[0,\infty)$, using  \eqref{eq:defqab} we have that
\begin{eqnarray*} \sum_{A,B \in \comp_{p_s, M_s}} \big ( p_{A,B} - q_{A,B} \big )^2 &\leq& {q_\lambda^2 \Cnst_2^{4/3} \over \cmin^2 m^2 V^{2/3}}\sum_{A,B \in \comp_{p_s, M_s}}\Big ( \Delta_{A,B} - {\cmin m |A||B| \over V} \Big )^2 \\ &=& {q_\lambda^2 \Cnst_2^{4/3} \over \cmin^2 m^2 V^{2/3}} \W(p_s; M_s)  \, ,\end{eqnarray*}
by definition of $\cmin$. \cref{thm:abDeltaAlternative} implies that $q_\lambda/\cmin \pe V^{1/3}/\X(p_s)$. Hence, if \eqref{eq:main1} holds, then 
$$ \sum_{A,B \in \comp_{p_s, M_s}} \big ( p_{A,B} - q_{A,B} \big )^2 \pe \Big (\X(p_s)/V^{1/3} \Big )^{16} \, .$$
Furthermore, if \eqref{eq:CompNum} holds, the Cauchy-Schwartz inequality implies that
$$ \sum_{A,B \in \comp_{p_s, M_s}}  | p_{A,B} - q_{A,B}| \pe |\comp_{p_s, M_s}| \Big (\X(p_s)/V^{1/3}\Big )^{8} = o(1) \, .$$

Thus, by naive coupling, we may couple so that the random edges of $G^\times$ and $G^\comp$ are \emph{equal} with probability $1-o(1)$. To each connected component $\C^\comp$ of $G^\comp$ we denote its size by $|\C^\comp|=\sum_{A \in \C^\comp} |A|$ 
and by $\C_1^\comp, \C_2^\comp,\ldots$ the connected components of $G^\comp$ sorted in a non-increasing order. 

Since \eqref{eq:main1} and \eqref{eq:CompNum} occur with probability $1-o(1)$, by \eqref{eq:GTimesConverges} we deduce that
\be\label{eq:GCompConverges}
\Cnst_2^{1/3} V^{-2/3} (|\C_1^\comp|, |\C_2^\comp|,\ldots ) \underset{\ell_2}{\overset{(d)}{\longrightarrow}}  \ZZ_{\lambda} \qquad \textrm{as } n\to \infty \, .
\ee

\subsubsection{Positioning  $p_c'(\lambda)$ in the scaling window and proof of \eqref{eq:ConvergenceChipc'}}

The convergence \eqref{eq:GCompConverges} is close to the desired result (\cref{thm:convergencepc'}), but the actual percolation components in $\omega_{p_c'(\lambda)}$ may be larger since we ignored the $p_s$-percolation components of size at most $M_s$. To overcome this obstacle we recall the definition \eqref{eq:defpc} and begin by choosing $\Lambda_-$ to be a negative number of large absolute value so that $\chi(p_c(\Lambda_-)) = \eps V^{1/3}$ with $\eps>0$ small enough so that  \cref{lem:ConcentrationFixedp2} implies that
$$ \proba_{p_c(\Lambda_-)} \Big ( \sum_{i\geq 1} |\C_i|^2 \geq {3 \over 4} \eps V^{4/3} \Big) \geq 5/6 \, .$$
Furthermore, \cref{lem:Np1p2Bound} and our choice of $M_s$ in \eqref{eq:ms} implies that $\E N(p_s, p_c(\Lambda_-); M_s) = o(V^{4/3})$ (since $(p_c-p_s)^{-1} \asymp \chi(p_s)$ by \cref{thm:subcriticalBCHSS}). By Markov's inequality and the last inequality
$$ \proba_{p_c(\Lambda_-)} \Big ( \sum_{i\geq 1} |\C_i|^2 - |N(p_s, p_c(\Lambda_-); M)| \geq {\eps \over 2} V^{4/3} \Big) \geq 2/3 \, .$$

Next we choose $\lambda_-<\Lambda_-$ to be negative with very large absolute value depending on $\eps$ so that by \eqref{eq:GCompConverges} we have that
$$ \proba_{p_c'(\lambda_-)} \Big ( \sum_{i\geq 1}|\C_i^\comp|^2 \geq {\eps \over 2} V^{4/3} \Big ) \leq 1/3 \, .$$
We note that for any $\lambda \in \R$
$$ \sum_{i\geq 1} |\C_i^\comp|^2 + |N(p_s,p_c'(\lambda));M_s)| = \sum_{\C \in \comp_{p_c'(\lambda)}} |\C|^2 \, $$
since both sides count the number of pairs of vertices which are connected in $\omega_{p_c'(\lambda_-)}$. Hence, at $\lambda_-$ we have
$$ \proba_{p_c'(\lambda_-)} \Big ( \sum_{i\geq 1}|\C_i|^2 - |N(p_s,p_c(\lambda_-);M)| \geq {\eps \over 2} V^{4/3} \Big ) \leq 1/3 \, .$$
Since the map $p \mapsto \sum_{\C \in \comp_{p}}|\C|^2 - |N(p_s,p;M)|$ is monotone increasing in $(p_s,1)$, we deduce that $p_c'(\lambda_-) \leq p_c(\Lambda_-)$. This shows that if $\lambda_-$ is negative with large absolute value, then $p'_c(\lambda_-)$ is not above the scaling window, that is, $\chi_\T(p_c'(\lambda_-)) \pe V^{1/3}$. We now return to our original $\lambda\in \R$ and assume that $\lambda>\lambda_-$, so by \eqref{eq:Newpc'def} we have 
$$ p_c'(\lambda) - p_c'(\lambda_-) = O((\lambda-\lambda_-)V^{-1/3}) \, .$$
We appeal to \cite[Theorem 1.4, (1.25)]{BCHSS1} and get that 
\be\label{eq:c'inwindow} \X_{\T}(p_c'(\lambda)) \pe V^{1/3} \, ,\ee
where the implicit constant may depend on $\lambda$ and $\lambda_-$. In particular $p'_c(\lambda)$ is also not above the scaling window for any fixed $\lambda\in \R$. Next we find the $(1+o(1))$ asymptotics of $\X(p_c'(\lambda))$ proving \eqref{eq:ConvergenceChipc'}. We are close since by \eqref{eq:GCompConverges} we have that
$$ \Cnst_2^{2/3} V^{-4/3} \sum_{i \geq 1} |\C_i^\comp|^2 \stackrel{(d)}{\longrightarrow} \kappa(\lambda) \, ,$$
in $\omega_{p_c'(\lambda)}$. To show that this convergence takes places in expectation we need to show that $\sum_{i \geq 1} |\C_i^\comp|^2$ are uniformly integrable. To this end, it suffices to show that $\sum_{i \geq 1} |\C_i|^2$, are uniformly integrable.
Indeed, $$\E_{p_c'(\lambda)} \sum_{i \geq 1} |\C_i|^2 = V\chi(p_c'(\lambda)) \pe V^{4/3} \, ,$$
using \eqref{eq:c'inwindow}. %
Also,  we bound 
\begin{eqnarray*} \E_{p_c'(\lambda)} \Big [(\sum_{i \geq 1} |\C_i|^2)^2 \Big ] &=& \E_{p_c'(\lambda)} \Big [ \sum_{i\geq 1} |\C_i|^4 \Big ] + \E_{p_c'(\lambda)} \Big [ \sum_{i\neq j} |\C_i|^2|\C_j|^2\Big ] \\ &\leq& \E_{p_c'(\lambda)} \Big [ \sum_{i\geq 1}|\C_i|^4 \Big ] + \Big ( \E\big [ \sum_{i\geq 1}|\C_i|^2 \big ] \Big )^2 \, ,
\end{eqnarray*}
using the BK inequality. %
By tree graph inequalities \eqref{eq:treegraph2}, 
$$\E_{p_c'(\lambda)} \sum_{i \geq 1}|\C_i|^4 = V \E_{p_c'(\lambda)} |\C(v)|^3 \pe V \X(p_c'(\lambda))^5 \pe V^{8/3}  \, ,$$
using \eqref{eq:c'inwindow}. %
Therefore,
the random variables $V^{-4/3} \sum_{i\geq 1} |\C_i|^2$ are bounded in $L^2$ so they are uniformly integrable. So, the smaller sums $V^{-4/3} \sum_{i\geq 1} |\C_i^\comp|^2$ are also uniformly integrable. We deduce that  
$$ 
 \Cnst_2^{2/3} V^{-4/3} \E_{p_c'(\lambda)} \Big [\sum_{i \geq 1} |\C_i^\comp|^2 \Big ] \longrightarrow \kappa(\lambda) \, .
$$
Again we have that 
$$ \E_{p_c'(\lambda)} \sum_{i \geq 1} |\C_i|^2 = \E_{p_c'(\lambda)} \sum_{i\geq 1} |\C_i^\comp|^2 + \E |N(p_s,p_c'(\lambda); M)| \, .$$
Since the second expectation on the right hand side is $o(V^{4/3})$ by \cref{lem:Np1p2Bound}, and since  $V\chi(p)=\E_p \sum_{i\geq 1} |\C_i|^2$ we get
$$ \X(p_c'(\lambda)) = (1+o(1)) \Cnst_2^{-2/3}\kappa(\lambda) V^{1/3} \, .$$
This proves \eqref{eq:ConvergenceChipc'} and shows that $p_c'(\lambda)$ is in the desired location in the scaling. To complete the proof of \cref{thm:convergencepc'} it thus remains to show \eqref{eq:convergencepc'}.

\subsubsection{Convergence of components sizes at $p_c'(\lambda)$}
We are only left to show that the sizes of the actual components $\{\C_i\}_{i \geq 1}$ in $\comp_{p_c'(\lambda)}$, rather than $\{\C_i^\comp\}_{i\geq 1}$, converge to $\ZZ_\lambda$. Fix an integer $k\geq 1$ and small $\eps>0$. We first claim that if
\be \label{eq:separated} \Big \{ |\C_i^\comp| \geq |\C_{i+1}^\comp| + \eps V^{2/3} : i=1,\ldots, k \Big \}\ee
and 
\be\label{eq:Nsmall} N:=|N(p_s,p_c'(\lambda); M_s)| < \eps^2 V^{4/3} \ee
hold, then 
\be\label{eq:done} |\C_i - \C_i^\comp| \leq \eps V^{2/3} \qquad i=1,\ldots, k \, .\ee
Indeed, since components of $G^\comp$ are subsets of components in $\comp_{p_c'(\lambda)}$, for each $j \geq 1$ there exists a unique $I_j$ such that $\C_j^\comp \subset \C_{I_j}$. We cannot have two distinct $j_1,j_2 \in \{1,\ldots, k\}$ such that $I_{j_1}=I_{j_2}=i$ for some $i \geq 1$ since any $p_c'(\lambda)$-open path between a vertex of $\C_{j_1}^\comp$ and a vertex of $\C_{j_2}^\comp$ must visit a vertex belonging a component of $\comp_{p_s}$ of size smaller than $M_s$; this implies that $N \geq \eps^2 V^{4/3}$ contradicting \eqref{eq:Nsmall}. Next, if $\C_1^\comp \not \subset \C_1$, then by \eqref{eq:separated} we have
$$ |\C_1| - \max_j |\C_j^\comp| \geq \eps V^{2/3} \, ,$$
which in turn again implies that $N \geq \eps^2 V^{4/3}$ contradicting \eqref{eq:Nsmall}. Continuing iteratively this argument gives that $\C_i^\comp \subset \C_i$ for each $i=1,\ldots, k$. Lastly, if $|\C_i| - |\C_i^\comp|\geq \eps V^{2/3}$ for some $i=1,\ldots, k$ we again get that $N \geq \eps^2 V^{4/3}$ contradicting $\eqref{eq:Nsmall}$. We conclude that \eqref{eq:separated} and \eqref{eq:Nsmall} imply \eqref{eq:done}. 

Now, let $\ell_1,\ell_2,\ell_3,\dots$ be the  coordinates of $\ZZ_\lambda$ and denote 
$$ \Omega_{\eps,k}=\{ \ell_i \geq \ell_{i+1} + \eps : i=1,\ldots,k\} \, .$$
It is a standard fact that $(\ell_1,\ldots,\ell_{k+1})$ is absolutely continuous with respect to Lebesgue measure in $[0,\infty)^{k+1}$ so $\proba(\Omega_{\eps,k}) \to 1$ as $\eps \to 0$. By \eqref{eq:GCompConverges} we deduce that the probability of the event \eqref{eq:separated} tends to $1$ uniformly in $n$ as $\eps \to 0$. Furthermore,  \cref{lem:Np1p2Bound} we have that $\E |N(p_1,p_c'(\lambda); M_s)| = o(V^{4/3})$ so for each fixed $\eps>0$ the event \eqref{eq:Nsmall} occurs with probability $1-o(1)$. We deduce that there exists a positive function $f(\eps)\to 0$ as $\eps \to 0$ such that for every $n$ large enough with probability $1-f(\eps)$ the event $\eqref{eq:done}$ holds. Together with \eqref{eq:GCompConverges} we obtain that 
$$ \Cnst_2^{1/3} V^{-2/3} (|\C_1|, \ldots |\C_k|) \stackrel{(d)}{\longrightarrow} (\ell_1,\ldots, \ell_{k}) \, .$$
This shows \eqref{eq:convergencepc'} and completes the proof of \cref{thm:convergencepc'}.  \qed

\subsection{Proof of \cref{thm:main_pc} and \cref{thm:main}}\label{sec:main_pc}

To prove \cref{thm:main_pc} and \cref{thm:main} it remains to relate $p_c(\Z^d)$ to $p_c'(\lambda)$ and to position it in the scaling window. We will use the following.

\begin{lemma} \label{lem:PcMinusPs} There exists $\eta>0$ such that for any $p<p_c(\Z^d)$ we have that
$$ p_c(\Z^d) - p - {\chi_\Z(p) \over \chi'_\Z(p)} \pe (p_c(\Z^d)-p)^{{1 \over 1-3\eta}} \, . $$
\end{lemma}
\begin{proof}For any $\lambda\in \R$, let $p_c'(\lambda,\eta,n)$ be as in the definition of $p_c'(\lambda)$ in \eqref{eq:Newpc'def} with $\beta_n\equiv 0$ and $\eta\in(0,\eta_0)$ is a small fixed constant. That is,
$$p_c'(\lambda,\eta,n) = p_s + {\X_\T(p_s) \over \X_\T'(p_s)} + {\Cnst_1} \Cnst_2^{2/3} \lambda  V^{-1/3} \, ,$$
where $p_s=p_s(\eta, n)=p_c - V^{-1/3+\eta}$ and $V=n^d$.

Next, for each $n\in \N$ we define $\lambda_n \in \R$ as the unique number such that 
$$ p_c(\Z^d) = p_c'(\lambda_n,\eta,n) \, .$$
By \cite[Corollary 2.2]{HeydenHofstad2} we have $\chi_\T(p_c(\Z^d))= \Theta(V^{1/3})$, so \eqref{eq:ConvergenceChipc'} of \cref{thm:convergencepc'} implies that $\{\lambda_n\}$ is a bounded sequence (recall that $\kappa(\lambda)$ is a continuous increasing bijection mapping $\R$ onto $(0,\infty)$).
Furthermore, applying \eqref{eq:sharpChi} and \eqref{eq:Chi'TvsZ} of \cref{lem:sharpChiChi2Torus} gives that
\be \label{eq:TvsZ} {\chi_\T(p_s) \over \chi'_\T(p_s)} = (1+O(V^{-1/2} \chi_\T(p_s)^{3/2})){\chi_\Z(p_s) \over \chi'_\Z(p_s)} = {\chi_\Z(p_s) \over \chi'_\Z(p_s)} + o(V^{-1/3}) \, ,\ee
since $\chi_\T(p_s) \asymp \chi_\Z(p_s) \asymp (p_c(\Z^d)-p_s)^{-1}$ and $\chi'_\Z(p_s)\asymp (p_c(\Z^d)-p_s)^{-2}$ and since $(p_c(\Z^d)-p_s)\gg V^{-1/3}$. Therefore,
$$ p_c(\Z^d)- p_s - {\chi_\Z(p_s) \over \chi'_\Z(p_s)} \pe V^{-1/3} \, .$$
Since $p_c-p_s = n^{-d(1/3-\eta)}$ we may rewrite the last inequality as 
\be p_c(\Z^d)- p - {\chi_\Z(p) \over \chi'_\Z(p)} \pe (p_c-p)^{{1 \over {3(1/3-\eta)}}} \label{eq:PcMinusPsRestricted} \ee
for any $p \in S:=\{ p_c(\Z^d) - n^{-d(1/3-\eta)} : n \in \N \}$. We would like to prove the last inequality for any $p<p_c$ so it suffices to estimate how the left hand side of \eqref{eq:PcMinusPsRestricted} changes between two consecutive values of $S$. To this end, note that the derivative of the left-hand side of \eqref{eq:PcMinusPsRestricted} is given by 
\[ -2+\frac{\X_\Z(p)\X_\Z''(p)}{\X_\Z'(p)^2} \, ,\]
which is bounded in absolute value 
using \cref{thm:sharpChiChi2Zd}, \cref{lem:sharpChiChi2Torus}, and \cref{lem:chi''}. Next, for any $p< p_c$ let $n(p)$ be the largest $n\in \N$ such that $p'=p_c-n(p)^{-d(1/3-\eta)}\leq p$. We apply the inequality \eqref{eq:PcMinusPsRestricted} to $p'\in S$ and since the aforementioned derivative is bounded we get that
\[ p_c(\Z^d)- p - {\chi_\Z(p) \over \chi'_\Z(p)} \pe n(p)^{-d/3} + (p-p') \, . \]
Using the definition of $n(p)$ and using that $n^{-\alpha} - (n+1)^{-\alpha} \pe n^{-\alpha-1}$ for any fixed $\alpha>0$ we get that
$$ n(p)= \Big \lfloor (p_c-p)^{-{1 \over d(1/3-\eta)}} \Big \rfloor \qquad \text{hence} \qquad p-p' \pe n(p)^{-d(1/3-\eta)-1} \, .$$
So as long as $\eta<1/d$ we get that $p-p' \pe n(p)^{-d/3}$. Putting these together we conclude that if $\eta<1/d$, then
\[ p_c(\Z^d)- p - {\chi_\Z(p) \over \chi'_\Z(p)} \pe n(p)^{-d/3} \pe (p_c-p)^{1 \over 3(1/3-\eta)} \, ,\]
using $p_c-p =\Theta (n(p)^{-d(1/3-\eta)})$ for the last inequality. This concludes the proof.
\end{proof}

 \begin{remark} \cref{lem:PcMinusPs} implies that $\X_\Z(p)-\Cnst_1(p_c-p)^{-1}$ is negligible compared to $(p_c-p)^{-1+\eps}$ for some $\eps>0$. We believe the next term to be of order the square diagram 
 \[ \sum_{x,y,z} \proba_p(0\slr x)\proba_p(x\slr y)\proba_p(y\slr z)\proba_p(z\slr 0). \]
 \end{remark}

\begin{proof}[Proof of \cref{thm:main_pc}] Let $\eta>0$ be a small fixed constant so that the assertion of \cref{lem:PcMinusPs} holds. We proceed similarly to the proof of \cref{lem:PcMinusPs} and set
\be\label{eq:pc'zerobetaagain}p_c'(\lambda,{\eta \over 2},n) = p_s + {\X_\T(p_s) \over \X_\T'(p_s)} + {\Cnst_1 \Cnst_2^{2/3} \lambda } V^{-1/3} \, ,\ee
where $p_s=p_s({\eta \over 2},n)=p_c - V^{-1/3+\eta/2}$. Again for each $n$ set $\lambda_n\in \R$ to be the unique number such that $p_c(\Z^d) = p_c'(\lambda_n,{\eta \over 2},n)$, so by \eqref{eq:TvsZ} we get that

\be\label{eq:almostdone} p_c(\Z^d) - p_s - {\X_\Z(p_s) \over \X'_\Z(p_s)} = {\Cnst_1 \Cnst_2^{2/3} \lambda_n } V^{-1/3} + o(V^{-1/3}) \, .\ee
However, by \cref{lem:PcMinusPs} the left hand side of the above is at most $(p_c(\Z^d)-p_s)^{1 \over 1 - 3\eta}$. Since this time $p_c(\Z^d)-p_s =V^{-1/3+\eta/2}$, and ${-1/3+\eta/2 \over 1-3\eta} < -1/3$ when $0<\eta<1/3$, we deduce  that the left hand side of \eqref{eq:almostdone} is in fact $o(V^{-1/3})$. Multiplying by $V^{1/3}$, we get that $\lambda_n\to 0$ as $n\to \infty$. We deduce by this together with \eqref{eq:pc'zerobetaagain}, \eqref{eq:almostdone} and \eqref{eq:TvsZ} that

that for any $\lambda \in \R$ 
\be\label{eq:pcZdpc'}\big  | p_c(\Z^d) + \Cnst_1 \Cnst_2^{2/3}\lambda V^{-1/3} - p_c'(\lambda,{\eta \over 2},n) \big | = o(V^{-1/3}) \, .\ee
Denote 
$$ \beta_n = V^{1/3} \Big (p_c(\Z^d) + \Cnst_1 \Cnst_2^{2/3}\lambda V^{-1/3} - p_c'(\lambda,{\eta \over 2},n) \Big ) \, ,$$
so that $\beta_n \to 0$ as $n\to \infty$. Invoking \cref{thm:convergencepc'} with $p_c'(\lambda,\eta,n,\beta_n)$ (which equals $p_c'(\lambda,\eta,n)+\beta_n V^{-1/3}=p_c(\Z^d) + \Cnst_1 \Cnst_2^{2/3}\lambda V^{-1/3}$) yields the convergence of $(\C_1,\C_2,\ldots)$ to $\ZZ_\lambda$ when $p=p_c(\Z^d) + \Cnst_1 \Cnst_2^{2/3}\lambda V^{-1/3}$ and concludes the proof. 
\end{proof}

\begin{proof}[Proof of \cref{thm:main}] 
Let $p_c(\lambda)=p_c(\lambda,n)$ be the sequence defined in \eqref{eq:defpc}.
By \eqref{eq:ConvergenceChipc'} of \cref{thm:convergencepc'} and the definition of $p_c(\lambda)$ we
get that for any $\eps>0$, as long as $n$ is large enough, we have
$$ p_c'(\lambda-\eps) \leq p_c(\lambda) \leq p_c'(\lambda+\eps) \, .$$
Thus, by definition \eqref{eq:Newpc'def} we get that for every $\lambda\in \R$ we have
$$ |p_c(\lambda) - p_c'(\lambda)| = o(V^{-1/3}) \, .$$    
Hence \eqref{eq:pcZdpc'} implies the desired asymptotics for $p_c(\lambda)$ and so \cref{thm:main_pc} implies the desired convergence.  
\end{proof}

\subsection{Proof of \cref{cor:supercritical}}
Let $d_P(\cdot,\cdot)$ be the Prohorov distance between probability measures on $\R$, see \cite[p 72]{Billing99}. Abusing notation we write $d_P(X,Y)$ for real random variables $X,Y$ to denote the Prohorov distance of their laws. 
For $j\geq 1$ and $\lambda\in \R$, write $\C^\T_{j,\lambda,n}$ for the $j$-th connected component at $p_c(\Z^d)+\Cnst \lambda V^{-1/3}$ in the discrete torus $(\Z/n\Z)^d$ in high dimension. Also, let $ \C^{\text{ER}}_{j,\lambda}$ be the $j$-th connected component of an Erd\H{o}s-R\'enyi random graph $G(n,1/n+\lambda n^{-4/3})$. Also, let $\mathbf{L}_{j,\lambda}$ be the length of the $j^{th}$ entry of $\ZZ_{\lambda}$ (that is, the length of the $j$th longest excursion). By \cref{thm:main_pc} for every converging sequence $\lambda_n\to \lambda$  we have
\[ d_{P}(\Cnst'V^{-2/3}|\C^\T_{j,\lambda_n,n}|,\mathbf L_{j,\lambda})\limit_{n\to \infty} 0.\]
It follows by a compactness argument that for any $a<b$ we have
\[ \sup_{\lambda \in [a,b]} d_{P}(\Cnst'V^{-2/3}|\C^\T_{j,\lambda,n}|,\mathbf L_{j,\lambda})\limit_{n\to \infty} 0.\]
Since the last expression in increasing in $b$, we may pick an increasing diverging sequence $(b_n)_{n\in \N}$ such that as $n\to \infty$,
\[ \sup_{\lambda \in [0,b_n]} d_{P}(\Cnst'V^{-2/3}|\C^\T_{j,\lambda,n}|,\mathbf L_{j,\lambda})\limit_{n\to \infty} 0.\]
In other words, if $\lambda_n=\eps(n)V^{1/3}\to \infty$ slowly enough, then we have
\[d_{P}(\Cnst'V^{-2/3}|\C^\T_{j,\lambda_n,n}|,\mathbf L_{j,\lambda_n})=o(\lambda_n^{-3}) \, ,\]
as $n\to \infty$. Similarly using Aldous \cite{Aldous97}, we have 
\[d_{P}(n^{-2/3}|\C^{\text{ER}}_{j,\lambda_n,n}|,\mathbf L_{j,\lambda_n})=o(\lambda_n^{-3}).\]
By the triangle inequality, for every $j\geq 1$ as $n\to \infty$ we have
\begin{equation} d_{P}(\Cnst'V^{-2/3}|\C^\T_{j,\lambda_n,n}|,n^{-2/3}|\C^{\text{ER}}_{j,\lambda_n,n}|)=o(\lambda_n^{-3})\label{eq:WassersteinTorusER} \end{equation}

By Bollob\'as \cite{Bol84} and \L uczak \cite{Luczak90} we have in probability as $n\to \infty$ that 
\[ \frac{|\C^{\text{ER}}_{1,\lambda_n,n}|}{2n^{2/3}\lambda_n}\to 1\quad \text{and for every $j\geq 2$} \quad \frac{|\C^{\text{ER}}_{j,\lambda_n,n}|}{2n^{2/3}\log(\lambda_n^3)/\lambda_n^2}\to 1 . \]
Since convergence in probability implies the weak convergence, this last convergence can be expressed using the distance $d_P$ between the above random variables and  the (random) variable $1$. By using the triangle inequality, \eqref{eq:WassersteinTorusER}, and the fact that for every real random variables $X,Y$ and $r\in \R$ we have $d_{P}(rX,rY)\leq \max(|r|,1)d_{P}(X,Y)$, we deduce that,
\[ d_{P}\Big (\frac{\mathbf C'|\C^\T_{1,\lambda_n,n}|}{2V^{2/3}\lambda_n}, 1\Big)\to 0\quad \text{and for every $j\geq 2$} \quad d_{P}\Big (\frac{\Cnst'|\C^\T_{j,\lambda_n,n}|}{2V^{2/3}\log(\lambda_n^3)/\lambda_n^2}, 1\Big )\to 0.\]
In other words we have the convergence in probability,
\[ \frac{\mathbf C'|\C^\T_{1,\lambda_n,n}|}{2V^{2/3}\lambda_n} \to  1\quad \text{and for every $j\geq 2$} \quad \frac{\Cnst'|\C^\T_{j,\lambda_n,n}|}{2V^{2/3}\log(\lambda_n^3)/\lambda_n^2}\to 1.\]
The corollary is obtained by the change of variable $\lambda_n=\eps(n)V^{1/3}/\Cnst$, and by replacing $\Cnst=\Cnst_1\Cnst_2^{2/3}$ and $\Cnst'=\Cnst_2^{1/3}$.
\qed 

\section{Contractive bound for $\W(p)$}\label{sec:Contraction}
Our goal is to prove \cref{thm:main_goal}. That is, we wish to bound from above the random variable
\be\label{eq:defW} \W(p; M) := \inf _{\cs>0} \sum_{A\neq B \in \comp_{p,M}} \left ( \Delta_{A,B} - {\cs m |A||B| \over V} \right )^2 \, ,\ee
defined in \eqref{eq:defW1} for a subcritical $p$ and suitable $M=V^{\mp}$. For convenience we will suppress $M$ from the notation and we will denote this random variable by $\W(p)$. We will achieve \cref{thm:main_goal} by six rounds of exposure $p_1 < p_2 < \ldots < p_7$ within the slightly subcritical regime. At each round $i >1$, with high probability $\W(p_i)/\chi(p_i)^2$ will significantly contract compared to $\W(p_{i-1})/\chi(p_{i-1})^2$.

\begin{theorem} \label{thm:contraction} There exists $\delta_0\in(0,1/30)$ such that for any $\delta_1\in (0,\delta_0)$ and $\delta_2\in(0,{\delta_1 \over 4})$, upon setting 
\be\label{eq:epsDef} \eps_1 = V^{\delta_1-1/3}\, , \quad \eps_2= V^{\delta_2 - 1/3} \qquad \textrm{and} \qquad p_1 = p_c - \eps_1 \, ,\quad p_2 = p_c - \eps_2 \, ,\ee
we have that
$$ \W(p_2) = \Op\Big ( \W(p_1) (\eps_1/\eps_2)^{2} {\log(V)^{10} \over \eps_2^3 V} + V^{3/10} \eps_1^2 / \eps_2^3 \Big )  \, .$$
\end{theorem}

We begin with some preparations and notations. As mentioned earlier, throughout the proof we denote $M=V^{\mp}$. We will denote the components of $\comp_{p_1,M}$ by $\{a,b,c,\ldots\}$ and that of $\comp_{p_2,M}$ by $\{A,B,C,\ldots\}$. We will frequently use the following properties of $\eps_1$ and $\eps_2$ 
\be\label{eq:eps2Fixed} V^{-1/3} \log (V)^{20}\leq \eps_2 \leq \eps_1 /\log(V) \quad \hbox{and} \quad {1 \over \eps_2^4} \geq {V\over \eps_1} \, ,\ee 
which follows for $n$ large enough since $\delta_2\in(0,\delta_1/4)$. Lastly, we will not try to find the optimal value of $\delta_0$, rather take it to be small enough whenever it is needed throughout the proofs in this section.  
As usual we write $\Delta_{a,b}$ for the number of $p_1$-closed edges between distinct component $a$ and $b$ of $\comp_{p_1,M}$ and similarly $\Delta_{A,B}$ for the number of $p_2$-closed edges between distinct components $A$ and $B$ of $\comp_{p_2,M}$; in both we set $\Delta$ to vanish on the diagonal. We treat $\Delta$ as a symmetric matrix indexed by $\comp_{p_1, M}$ or $\comp_{p_2,M}$ and use the usual matrix notation on it.  

Define the $\omega_{p_1}$-measurable event $\good$ as the intersection of the following: 
\begin{eqnarray} 
\max_{a} |a| &\leq& 2\chi(p_1)^2 \log(V/\chi(p_{1})^3) \, , \label{eq:maxa}\\
\sum_{a} |a|^2 &\leq& (1+1/m) V \chi(p_1) \, , \label{eq:a2b2} \\
\sum_{a,b} \Delta_{a,b}^2 &\leq& C(d,L) \chi(p_1)^2 \, , \label{eq:Delta2} \\
\sum_{a \neq b} |a||b|\Delta_{a,b} &\geq& (1-o_{m}(1))V m \chi(p_1)^2  \, , \label{eq:abDelta} \\ 
\Tr(\Delta^4) &\leq& (1 + o_{m}(1)) m^4\X(p_1)^4 \, , \label{eq:TrDelta4} \\
\max_{a,b} \Delta_{a,b} &\leq& \log(V)^{5} \X(p_1)^4/V \, ,  \label{eq:maxDelta} \\ 
\max_{a} \sum_{b} \Delta_{a,b}^2 &\leq& \log(V)^{10}\X(p_1)^5/V \, , \label{eq:maxSumrowDelta} 
\end{eqnarray}
where $C(d,L)$ is the constant of \cref{thm:Delta2}. 
Our subcritical building blocks described in \cref{sec:buildingblocks}, in particular, 
\cref{thm:subcriticalBCHSS},
\cref{lem:ConcentrationFixedp2}, \cref{thm:Delta2}, \cref{thm:abDeltaAlternative}, \cref{thm:Delta4}, \cref{thm:MaxDelta}, and \cref{thm:MaxSumDelta}, in that order, show that \eqref{eq:maxa}---\eqref{eq:maxSumrowDelta} occur with probability $1-o(1)$. We conclude that 
\be\label{eq:WHPgood} \proba_{p_1}(\good) = 1-o(1) \, .\ee

We proceed to state the main step of the proof which involves conditioning on any $\omega_{p_1} \in \good$. From now until \cref{sec:P2positioning} we will slightly abuse notation and assume that $\eps_2=\eps_2(\omega_{p_1})$ is a $\omega_{p_1}$-measurable random variable satisfying \eqref{eq:eps2Fixed} (unlike the setup of \cref{thm:contraction}, where $\eps_2$ is fixed). Given $\omega_{p_1}\in \good$ and any $\eps_2=\eps_2(\omega_{p_1})>0$ satisfying \eqref{eq:eps2Fixed} we set $p_2^*>p_1$ by 
\be\label{def:p2} p_2^* := p_1 + {(1-p_1)(1- \eps_2/\eps_1) \over \lambda_1(\Delta)} \, ,\ee
where $\lambda_1(\Delta)$ is the largest eigenvalue of the symmetric matrix $\Delta$ which is indexed by $\comp_{p_1,M}$. Thus $p_2^*$ is a random variable measurable with respect to $\omega_{p_1}$. The relation between the random variable $p_2^*$ and the number $p_2$ will become clearer in \cref{sec:P2positioning}.

We condition on any configuration $\omega_{p_1}\in \good$ but instead of considering $\omega_{p_2^*}$, we consider percolation on a restricted set of edges, namely only on edges connecting two distinct components of $\comp_{p_1,M}$. By definition these edges are $p_1$-closed and we check whether they are they are $p_2^*$-open. If there exists such a $p_2^*$-open edge, then we declare that an ``edge'' between the corresponding components $a$ and $b$ is open. The conditional probability of this event is precisely 
\be\label{defQ} Q_{a,b} := 1-\Big ( 1- {p_2^*-p_1 \over 1-p_1} \Big )^{\Delta_{a,b}} 
\, .
\ee
We have thus defined an inhomogeneous percolation model with vertex set $\comp_{p_1,M}$ and edges are \emph{open}  independently with probability $Q_{a,b}$ and \emph{closed} otherwise. We write $a \lr b$ for the event that there is an open path connecting $a$ to $b$ in this percolation model. We also frequently write $a \lr b \nlr c \lr d$ for the event that $a\lr b$ and $c \lr d$ but $a,b$ and $c,d$ are in distinct components. We denote the resulting components of this model by $\comp(p_2^*, p_1, M)$. The size $|A|$ of a component $A \in \comp(p_2^*, p_1, M)$ is the sum of sizes of the corresponding components of $\comp_{p_1,M}$ which comprise $A$, i.e., $|A|=\sum_{a \subset A} |a|$. By definition $|A|\geq M$ for all $A\in \comp(p_2^*, p_1, M)$. Note that up to the choice of parameters, this is exactly the definition of the sprinkled component graph in \cref{def:GComp}; we opted to use different notation to avoid confusion. In the rest of this section we denote by $\proba_{p_2^*}$ this conditional probability space (given $\omega_{p_1}$) of this inhomogeneous percolation process and by $\E_{p_2^*}$ the corresponding expectation.

\begin{remark} The components $\comp(p_2^*, p_1, M)$ are necessarily subsets of components of $\comp_{p_2^*,M}$ (that is, components of size at least $M$ in $\omega_{p_2^*}$, however, there may be components in $\comp_{p_2^*,M}$ that have no subsets in  $\comp(p_2^*, p_1, M)$ because they are comprised of small $\omega_{p_1}$ components. Furthermore, there may be components in $\comp_{p_2^*,M}$ which contain two distinct subsets in $\comp(p_2^*, p_1, M)$ because there is a small $\omega_{p_1}$ component ``breaking'' the large $\comp_{p_2^*,M}$ component. Since \cref{thm:contraction} is about $\comp_{p_2,M}$ (rather than $\comp(p_2^*, p_1, M)$) we will need to deal with this pesky nuisance. This is done in 
\cref{sec:ProofContraction} and in the meantime we proceed with the main thrust of the proof. 
\end{remark}

Both $\Delta$ and $Q$ are random symmetric matrices measurable with respect to $\omega_{p_1}$ and are indexed by $\comp_{p_1,M}$. We also set $Q$ to vanish on the diagonal. We denote by $\{\lambda_i(Q)\}_{i \geq 1}$ the eigenvalues of $Q$ ordered in descending order. It is not hard to show that $\lambda_2(Q) < \lambda_1(Q)<1$ (see \cref{lem:SpectrumQ}). Hence $\lambda_1(Q)$ is a simple eigenvalue and we may take the unique unit eigenvector $v$ of $Q$ corresponding to $\lambda_1(Q)$. By Perron-Frobenius the entries of $v$ are non-negative. Since $\lambda_1(Q)<1$ the matrix $I-Q$ is invertible and we note that $v$ is also an eigenvector of the matrix $Q(I-Q)^{-1}$ with corresponding eigenvalue $\lambda_1(Q)(1-\lambda_1(Q))^{-1}$. The matrix $Q(I-Q)^{-1}$ is of central importance since its $(a,b)$ entry approximates very well $\proba_{p_2^*}(a \lr b)$, as shown in \cref{sec:mainest}. We define $\ti{Q}$ by
\be\label{def:TildeP2} \ti{Q}_{a,b} := {\lambda_1(Q) \over 1-\lambda_1(Q)}  v_a v_b \, ,\ee
in other words, $\ti{Q}$ is a rank $1$ matrix approximation of $Q(I-Q)^{-1}$. Next, for $\cs\in \R$ we define the following $\omega_{{p_1}}$-measurable random matrix indexed by $\comp_{p_1,M}$, 
\be\label{eq:defWMatrix} W(p_1,\cs)_{a,b} :=  {\bf 1}_{a \neq b} \Big( \Delta_{a,b} - {\cs m \over V}|a||b| \Big ) \, ,\ee
Note that $\W(p_1; M)$ from \eqref{eq:defW} is just $\inf_{\cs \in \R} \|W(p,\cs)\|_F^2$. 
We now define $\cnst=\cnst(\omega_{p_1}, \eps_2(\omega_{p_1}))>0$ as the unique number so that
\begin{equation} \label{def:Const} \Tr(\ti{Q} W(p_1,\cnst)) = 0 \, ,\end{equation}
which is always possible since $\ti{Q}$ has non-negative entries so both  $\Tr(\ti{Q}\Delta)$ and $\Tr(\ti{Q}({\bf 1}_{a \neq b}|a||b|)_{a,b})$ are non-negative. In other words, $\cnst>0$ is defined by 
\be\label{eq:defCnst} \cnst := \frac{V\Tr(\ti{Q} \Delta)}{m \Tr(\ti{Q} ({\bf 1}_{a \neq b}|a||b|)_{a,b})} \, .
\ee

We cannot expect that $\cnst$ minimizes $\|W(p_1,t)\|_F$ since $\cnst$ depends on $\eps_2$ as well as $\omega_{p_1}$ and the minimizer and the minimum depends only $\omega_{p_1}$. However, we will later see that $\|W(p_1,\cnst)\|_F = O(\W(p_1))$ whenever $\omega_{p_1}\in \good$, see \cref{lem:WstarCloseToMinimal}.  We now state the main step in the proof of \cref{thm:contraction}. 

\begin{theorem} \label{thm:QuenchedContraction} Let $p_1,\eps_1$ be as in \cref{thm:contraction} and let $\omega_{p_1}$ be any configuration in $\good$. Let $\eps_2=\eps_2(\omega_{p_1})$ satisfy \eqref{eq:eps2Fixed} and set $p_2^*=p_2^*(\omega_{p_1},\eps_2)$ and $\cnst=\cnst(\omega_{p_1},\eps_2)>0$ as in \eqref{def:p2} and \eqref{def:Const} . Then 
$$ \sum_{a,b,c, d\in \comp_{p_1,M}} W(p_1,\cnst)_{a,c} W(p_1,\cnst)_{b,d} \proba_{p_2^*}({a \lr b \nlr c \lr d})   \pe \|W(p_1,\cnst)\|_F^2 (\eps_1/\eps_2)^2 \cdot  \frac{\log(V)^{10}}{\eps_2^3 V}  \, . $$
\end{theorem}

The rest of this section is organized as follows. In \cref{sec:SpectrumQ} we provide some basic estimates on the spectrum of $Q$; most importantly  that $Q$ exhibits a spectral gap, namely, $\lambda_1(Q)=1-(1+o_{m}(1))\eps_2/\eps_1$ and $\lambda_2(Q)< 1$ and bounded away from $1$ uniformly. This allows to significantly amplify the spectral gap for the matrix $Q(I-Q)^{-1}=\sum_{i\geq 1}Q^i$: it has first eigenvalue $(1+o(1))\eps_1/\eps_2$ and all others are $O(1)$; the eigenvectors are the same as $Q$'s, so it follows that $Q(I-Q)^{-1}$ is very close to its rank one approximation $\tilde{Q}$. The role of $Q(I-Q)^{-1}$ is rather intuitive as we have the trivial upper bound $\pp(a \lr b) \leq Q(I-Q)^{-1}_{a,b}$, see \eqref{eq:tpUpperBound}. In \cref{sec:mainest} we show a converse bound to \eqref{eq:tpUpperBound} which holds only in the $\ell_2$ sense, not necessarily pointwise. In \cref{sec:BoundingError1} we prove a diagrammatic inequality which will be used in the proof of \cref{thm:QuenchedContraction} which is in turn presented in \cref{sec:ProofQuenchedContraction}. To deduce \cref{thm:contraction} from \cref{thm:QuenchedContraction} we have to show that we can invoke \cref{thm:QuenchedContraction} with $p_2$ instead of $p_2^*$; this is done in \cref{sec:P2positioning}. 
The proof of \cref{thm:contraction} is then presented in \cref{sec:ProofContraction}. Lastly, the six iterations of \cref{thm:contraction} needed to deduce \cref{thm:main_goal} are performed in \cref{sec:proofMainGoal}.

\subsection{Linear algebra}\label{sec:SpectrumQ} 
In this section we estimate the spectrum of $Q$ and related matrices. We will use the min-max principle stating that for any symmetric $k\times k$ matrix $A$ we have 
\be\label{eq:minmax} \lambda_1(A) = \max_{y \in \R^k} { \l Ay, y\r \over ||y||_2^2 } \quad \textrm{and}\quad \lambda_2(A) = \min_{\substack{\mathcal{M} \subset \R^k \\ \dim(\mathcal{M})=k-1}} \max_{y \in \mathcal{M}} { \l Ay, y\r \over ||y||_2^2 } \, .\ee

The results of this section are encapsulated in the following lemma.

\begin{lemma} \label{lem:SpectrumQ} Assume the setting of \cref{thm:QuenchedContraction} and suppose that $\omega_{p_1}\in \good$, and that $n$ is large enough. Then
\begin{eqnarray}
\lambda_1(Q) &\leq& 1 - {\eps_2 \over \eps_1} \, , \label{eq:lambda1Upper} \\
\lambda_1(Q) &\geq& 1- (1+o_{m}(1)) {\eps_2 \over \eps_1} \, , \label{eq:lambda1Lower} \\ 
\lambda_2(Q) &\leq& o_{m}(1) \, , \label{eq:lambda2QUpper} \\
\|Q\|_F &\pe& 1 \, , \label{eq:QNorm} \\ 
\| Q(I-Q)^{-\ell} \|_F &\leq& (1+o_{m}(1))\left ({\eps_1 \over \eps_2} \right )^{\ell} \quad \forall \ell\geq 1 \, ,\label{eq:QIQnorm}\\ 
 \|Q(I-Q)^{-1} - \ti{Q}\|_F &\pe& 1 \label{eq:QIQRank1} \, . 
\end{eqnarray}
Furthermore, if $v$ is the unit eigenvector of $Q$ corresponding to $\lambda_1(Q)$ and $y$ is the vector $y_a:=|a|$, then
\be\label{eq:xv} \langle y ,v \rangle \geq (1-o_{m}(1)) V \chi(p_1) \, .\ee
\end{lemma}

\begin{proof} Write $N = Q -  {p_2^*-p_1 \over 1-p_1}\Delta$ so that for any vector $v$ indexed by $\comp_{p_1,M}$ we have
$$ \langle Qv,v\rangle = {p_2^*-p_1 \over 1-p_1} \langle \Delta v , v \rangle + \langle N v, v\rangle \, .$$
Note that $N$ has non-positive entries by Bernoulli's inequality applied in \eqref{defQ}. Hence, by taking $v$ to be the unit eigenvector of $Q$ corresponding to $\lambda_1(Q)$ and noting that the Perron-Frobenius Theorem implies $v$ has non-negative entries, we obtain that 
$$ \lambda_1(Q) \leq {p_2^*-p_1 \over 1-p_1} \langle \Delta v , v \rangle \leq {p_2^*-p_1 \over 1-p_1}  \lambda_1(\Delta) = 1-{\eps_2 \over \eps_1} \, ,$$
where the second inequality is due to the min-max principle and the last equality is by definition of $p_2^*$ in \eqref{def:p2}. This proves \eqref{eq:lambda1Upper}. 

Next, take the vector $y_a=|a|$ and apply the min-max principle with \eqref{eq:a2b2} and \eqref{eq:abDelta} to obtain 
\be\label{eq:lambda1DeltaLower} \lambda_1(\Delta) \geq \frac{\sum_{a\neq b} |a||b| \Delta_{a,b}}{\sum_{a} |a|^2} \geq (1-o_{m}(1))m \chi(p_1) \, .
\ee
Putting this in the definition of $p_2^*$ in \eqref{def:p2} and using \cref{thm:subcriticalBCHSS} gives that $p_2^*-p_1 \pe \eps_1$. Also, since $(1-x)^k \leq 1-kx+k^2x^2$ for any $x\in (0,1)$ we get using \eqref{defQ} that $|N_{a,b}|\leq (p_2^*-p_1)^2\Delta_{a,b}^2$. Hence 
$|N_{a,b}| \pe \eps_1^2 \Delta_{a,b}^2$. Then \eqref{eq:maxDelta} implies that $|N_{a,b}| \pe \log(V)^{10} \eps_1^{-6}/V^2$. This quantity is $o(\eps_2/\eps_1)$ by \eqref{eq:eps2Fixed}. Therefore, for any unit vector $v$ we have 
\be\label{eq:boundNvv} |\langle N v, v\rangle| \leq  \|N\|_{\infty}\langle v,v \rangle = o(\eps_2/\eps_1) \, ,\ee
whence 
\be\label{eq:QvsDelta} \langle Qv,v\rangle = {p_2^*-p_1 \over 1-p_1} \langle \Delta v , v \rangle + o(\eps_2/\eps_1) \, .\ee
Applying this with $v$ being the unit eigenvector of $\Delta$ corresponding to $\lambda_1(\Delta)$ and using the min-max principle \eqref{eq:minmax}, we get that $\lambda_1(Q) \geq (1-\eps_2/\eps_1) + o(\eps_2/\eps_1)$, proving \eqref{eq:lambda1Lower}. 

Next by \eqref{eq:TrDelta4} we have
$$ \lambda_1(\Delta)^4 + \lambda_2(\Delta)^4 \leq \Tr(\Delta^4) \leq (1+o_{m}(1)) m^4 \chi(p_1)^4 \, ,$$
so together with \eqref{eq:lambda1DeltaLower} we get that
\be\label{eq:lambda1DeltaUpper}  (1-o_{m}(1)) m \chi(p_1) \leq \lambda_1(\Delta) \leq (1+o_{m}(1)) m \chi(p_1) \, ,\ee
and that \be\label{eq:lambda2Upper} \lambda_2(\Delta) = o_{m}(1)m \chi(p_1) \, .\ee

To bound $\lambda_2(Q)$ we again write $v$ for the unit eigenvector of $\Delta$ corresponding to eigenvalue $\lambda_1(\Delta)$ so by the min-max principle we have that $\lambda_2(Q) \leq \max_{y \perp v, \|y\|=1} \langle Qy,y \rangle$. For any unit vector $y$ with $y \perp v$ we have $\langle \Delta y, y \rangle \leq \lambda_2(\Delta)$ and $|\langle N y, y \rangle| = o(1)$ by \eqref{eq:boundNvv}. Hence
$$\langle Qy,y \rangle = {p_2^*-p_1 \over 1-p_1} { \langle \Delta y, y \rangle } + { \langle N y, y \rangle } \leq {\lambda_2(\Delta) \over \lambda_1(\Delta)} + o(1) = o_{m}(1) \, ,$$
where we used \eqref{def:p2},  \eqref{eq:lambda1DeltaLower} and \eqref{eq:lambda2Upper}. This proves \eqref{eq:lambda2QUpper}.

Next, by the triangle inequality 
$$ \|Q\|_F \leq {p_2^*-p_1 \over 1-p_1} \|\Delta\|_F + \|N\|_F \, .$$
As before we have that $p_2^*-p_1\pe \eps_1$ and $\|\Delta\|_F \pe \eps_1^{-1}$  by \eqref{eq:Delta2}; this shows that the first term above is $\pe 1$. To bound $\|N\|_F$ we recall that $|N_{a,b}|=O(\eps_1^2 \Delta_{a,b}^2)$ so using \eqref{eq:maxDelta} we get
$$ \|N\|_F^2 \pe V^{-2} \eps_1^4 \chi(p_1)^8\log(V)^{10}\sum_{a,b} \Delta_{a,b}^2 = o(1) \, ,$$
by \eqref{eq:Delta2}, \cref{thm:subcriticalBCHSS} and since $\chi(p_1)V^{-1/3} \ll \log(V)^{-10}$ by \eqref{eq:eps2Fixed}. This proves \eqref{eq:QNorm}.

Next, we note that the eigenvalues of $Q(I-Q)^{-\ell}$ are precisely $\{ \lambda_i(Q)/(1-\lambda_i(Q))^\ell \}_{i \geq 1}$. By \eqref{eq:lambda2QUpper} we have $\lambda_k(Q) \leq 1/2$ for $k \geq 2$, which together with \eqref{eq:lambda1Upper} and \eqref{eq:lambda1Lower} gives
$$ \|Q(I-Q)^\ell\|_F^2 \leq (1-\lambda_1(Q))^{-2\ell} + 4^\ell \Tr(Q^2) \leq (\eps_2/\eps_1)^{-2\ell} + O(1) \, ,$$
using \eqref{eq:QNorm} in the last inequality. This proves \eqref{eq:QIQnorm}. Next, the eigenvalues of $Q(I-Q)^{-1} - \ti{Q}$ are $0$ and $\lambda_i(Q)/(1-\lambda_i(Q))$ for $i\geq 2$. Hence
$$ \big \| Q(I-Q)^{-1} - \ti{Q} \big \|_F^2 \leq  2\Tr(Q^2) \pe 1 \, ,$$
using \eqref{eq:QNorm} and proving \eqref{eq:QIQRank1}.

To prove \eqref{eq:xv}, we decompose $y= \langle y,  v\rangle v + u$ with $u \perp v$ where $v$ is the unit eigenvector of $Q$ corresponding to $\lambda_1(Q)$. On the one hand we have
$$ \langle Qy,y \rangle = \langle y,v\rangle^2 \lambda_1(Q) + \langle Qu,u \rangle \leq \langle y,v\rangle^2 + \lambda_2(Q)\|u\|^2 \leq \langle y,v\rangle^2 + \lambda_2(Q)\|y\|^2 \, ,$$
and on the other hand
$$ \langle Qy,y \rangle = {p_2^\star-p_1 \over 1-p_1}\langle \Delta y,y\rangle + \langle Ny,y \rangle \geq {(1-\eps_1/\eps_2)\langle \Delta y,y\rangle \over \lambda_1(\Delta)} - \|N\|_{\red{\infty}}\|y\|^2 $$
by \eqref{def:p2} and Cauchy-Schwartz. To continue bounding we have that $\|y\|^2 \pe V\chi(p_1)$ by \eqref{eq:a2b2} and $\|N\|_{\red{\infty}} = o(1)$ as we have shown earlier in this proof. Furthermore, \eqref{eq:abDelta} is precisely a lower bound on $\langle \Delta y,y\rangle$, and \eqref{eq:lambda1DeltaUpper} provides a upper bound to $\lambda_1(\Delta)$. These yield that 
$$(1-o_{m}(1))V \chi(p_1) \leq \langle Qy,y \rangle \leq \langle y,v\rangle^2 + \lambda_2(Q)\|y\|^2 \, ,$$
and since $\|y\|^2 = (1+o(1))V\chi(p_1)$ and $\lambda_2(Q)=o_{m}(1)$ the proof of \eqref{eq:xv} and the lemma is complete.
\end{proof}

\subsection{The two-point function $\tau_{p_2^*}$ and $Q(I-Q)^{-1}$} \label{sec:mainest}

Conditioned on $\omega_{p_1}$ we wish to investigate the two-point function $\tp(a,b)$ which we treat as a non-negative symmetric matrix indexed by $\comp_{p_1,M}$ 
$$ \tp(a,b) = \pp(a \lr b) {\bf 1}_{a \neq b} \, ,$$
where $\pp$ is the inhomogeneous percolation model defined below \eqref{defQ}.

If $a \lr b$ in $\omega_{p_2^*}$ and $a\neq b$, then there exists some integer $k\geq 1$ and distinct components $a_1,\ldots, a_k \in \comp_{p_1,M}$ such that $a=a_1$ and $b=a_k$ and there exists a $p_2^*$-open edge between $a_i$ to $a_{i+1}$ for $i=1,\ldots,k-1$. Hence, the  upper bound
\be\label{eq:tpUpperBound} \tp(a,b) \leq \sum_{k\geq 1} Q^k(a,b) = Q(I-Q)^{-1}_{a,b} \, ,\ee
follows immediately. The goal of this section is to prove the converse in an $\ell_2$ averaged sense. 
\begin{lemma} \label{lem:mainEstimate} Conditioned on $\omega_{p_1}\in \good$ we have
$$ \big \| \tp - Q(I-Q)^{-1} \big \|_F \pe \frac{\eps_1\log(V)^{10} }{\eps_2^4 V} \, .$$
\end{lemma}
Note that by \eqref{eq:eps2Fixed} this upper bound is smaller than $\eps_1/\eps_2$ which is the order of $\|Q(I-Q)^{-1}\|_F$ by \eqref{eq:QIQnorm} of \cref{lem:SpectrumQ}. So \cref{lem:mainEstimate} shows that some cancellation is taking place. In the proof we will use some of $Q$'s spectral estimates shown in \cref{lem:SpectrumQ}. We will also need bounds on the maximum entry and maximum row sum of $Q$; indeed, if $\omega_{p_1} \in \good$, then
\be\label{eq:maxQ} \|Q\|_\infty \pe \frac{\log(V)^5}{\eps_1 ^3 V} \, ,\ee
and
\be\label{eq:maxQRowsum}
\max_a \sum_b Q_{a,b}^2 \pe \frac{\log(V)^{10}}{\eps_1 ^3 V}  \, ,
\ee
which follows immediately from \eqref{eq:maxDelta} and \eqref{eq:maxSumrowDelta} together with Bernoulli inequality in \eqref{defQ}. Additionally will use a bound on the sup-norm $\|Q(I-Q)^{-1}\|_\infty$ which is better than the trivial $\|A\|_\infty \leq \|A\|_F$ valid for any matrix $A$. 

\begin{lemma} \label{lem:SupNormBoundOnSumQk} We have
$$ \|Q (I-Q)^{-1}\|_\infty \leq \|Q\|_\infty + \big (\max_a \sum_{b} Q_{a,b}^2 \big) \big ( 1 + \|Q(I-Q)^{-1} \|_F \big ) \, .$$
\end{lemma}
\begin{proof}
Since $Q(I-Q)^{-1} = \sum_{k \geq 1} Q^k$ we may use the triangle inequality to bound
$$ \|Q (I-Q)^{-1}\|_\infty \leq \|Q\|_\infty + \|Q^2\|_\infty + \|Q (\sum_{k\geq 1} Q^k) Q  \|_\infty\, .$$
We bound the second term using Cauchy-Schwartz 
\be\label{eq:Q2-CS} (Q^2)_{a,b} = \sum_c Q_{a,c} Q_{c,b} \leq  \sqrt{\sum_{c} Q_{a,c}^2} \sqrt{\sum_c Q_{b,c}^2} \leq \max_a \sum_{b} Q_{a,b}^2  \, ,\ee
hence $\|Q^2\|_\infty \leq \max_a \sum_{b} Q_{a,b}^2$. The third term we bound again using Cauchy-Schwartz 
\begin{eqnarray*} (Q (\sum_{k\geq 1} Q^k) Q)_{a,b} &=& \sum_{c,d} Q_{a,c} (Q(I-Q)^{-1})_{c,d} Q_{d,b} \\ &\leq& \big \|Q(I-Q)^{-1} \big \|_F \sqrt{\sum_{c,d} Q_{a,c}^2 Q_{b,d}^2}  \leq \big (\max_a \sum_{c} Q_{a,c}^2 \big)\|Q(I-Q)^{-1} \|_F \, ,\end{eqnarray*}
concluding the proof.
\end{proof}

\begin{corollary}\label{cor:TauSupNormBound} We have
$$ \|Q (I-Q)^{-1}\|_\infty \pe  \frac{\log(V)^{10}}{\eps_2 \eps_1^2 V} \, .$$
\end{corollary}
\begin{proof}
Plug into \cref{lem:SupNormBoundOnSumQk} the bounds \eqref{eq:QIQnorm} of \cref{lem:SpectrumQ} together with \eqref{eq:maxQ} and \eqref{eq:maxQRowsum}.
\end{proof}

\begin{proof}[Proof of \cref{lem:mainEstimate}]
The upper bound on $\tp$ is immediate and is given by \eqref{eq:tpUpperBound}. For the corresponding lower bound we note that $\{a \lr b\}$ is the union of the event that the edge $(a,b)$ is open with the events 
$$ \bigcup_{c \not \in \{a,b\}} \{ (a,c) \hbox{ open} \} \cap \{ c \lr b \off a\} \, .$$ 
Hence, by Bonferroni's inequality
\begin{eqnarray*} \tp(a, b) &\geq& Q_{a,b} + \sum_{c\not \in \{a,b\}} Q_{a,c} \proba_{p_2^*}(c \lr b \off a) \\ &-& Q_{a,b} \sum_{c}  Q_{a,c} \tp(c,b) - \sum_{c,d} Q_{a,c} Q_{a,d} \proba_{p_2^*}(c \lr b \andd d \lr b) \, ,\end{eqnarray*}
where in the third and fourth terms in the right hand side above we have dropped the ``off $a$'' and $c,d\not\in\{a,b\}$ requirements, making the right hand side even smaller. For the second term above we bound
$$ \proba_{p_2^*}(c \lr b \off a) \geq \tp(c, b) - \tp(c, a)\tp(a,b) \, ,$$
using the BK inequality. For the fourth term, we again use tree graph inequalities \eqref{eq:treegraph1} to bound
\begin{eqnarray*} \proba_{p_2^*}(c \lr b \andd d \lr b) &\leq&   \sum_e (I + \tp)(b,e) (I + \tp)(e,c) (I + \tp)(e,d) \\ &\leq& \sum_e (I -Q)^{-1}(b,e) (I -Q)^{-1}(e,c)(I -Q)^{-1}(e,d) \, ,
\end{eqnarray*}
where in the first inequality the identity matrix $I$ was added to account for the terms $e=b$, $e=c$ and $e=d$ and in the second inequality we used \eqref{eq:tpUpperBound} to bound $I+\tp \leq \sum_{k\geq 0} Q^k = (I -Q)^{-1}$.
We put these together and rearrange, recalling that both $Q$ and $\tp$ have $0$ on the diagonal, to obtain
\be\label{eq:LowerBoundPp2} ((I-Q)\tp)(a,b) = \tp(a,b) - \sum_{c} Q_{a,c} \tp(c, b) \geq Q_{a,b} - S_1 - S_2 - S_3 \, ,  
\ee
where we define and bound $S_1,S_2,S_3$ below. First, 
\begin{eqnarray*} S_1 := \sum_{c} Q_{a,c} \tp(c, a)\tp(a, b) &=& \tp(a, b) (Q \tp)_{a,a} \leq \tp(a, b) \big (\sum_{k\geq 2} Q^k\big )_{a,a} \\ &\leq& \|Q (I-Q)^{-1}\|_\infty  \tp(a, b)\, ,
\end{eqnarray*}
where we used \eqref{eq:tpUpperBound} for the first inequality, and summed over $k\geq 1$ (instead of $k\geq 2$) in the second inequality. Secondly,
$$ S_2 := Q_{a,b} \sum_{c}  Q_{a,c} \tp(c,b) \leq \| Q (I-Q)^{-1}\|_\infty  Q_{a,b} \, ,
$$
in the same way we bounded $S_1$. Lastly,
\begin{eqnarray*} S_3 &:=& \sum_{e}  (I-Q)^{-1}(b, e) \sum_c Q_{a,c}  (I-Q)^{-1}(e, c) \sum_d Q_{a,d}  (I-Q)^{-1}(e, d) \\ &=& \sum_e  (I-Q)^{-1}(b,e) (Q (I-Q)^{-1})_{a,e}(Q (I-Q)^{-1})_{a,e} \\ &\leq&  \|Q (I-Q)^{-1}\|_\infty  (Q (I-Q)^{-2})_{a,b}\, , \end{eqnarray*}
where we used that $Q$ and $ (I-Q)^{-1}$ commute. We put this back into \eqref{eq:LowerBoundPp2} and get %
$$ ((I-Q) \tp - Q)_{a,b} \geq - \|Q(I-Q)^{-1}\|_\infty ( Q+ \tp + Q (I-Q)^{-2})_{a,b} \, ,$$
for any $a \neq b$. Since $(I-Q)^{-1}=\sum_{k\geq 0}Q^k$ has non-negative entries, we may multiply by it on both sides to obtain
$$ (\tp - Q(I-Q)^{-1})_{a,b} \geq -\|Q(I-Q)^{-1}\|_\infty ((I-Q)^{-1}( Q+ \tp+Q (I-Q)^{-2}) )_{a,b} \, .$$
Together with \eqref{eq:tpUpperBound} we get that 
$$ \big \| \tp - Q(I-Q)^{-1} \big \|_F \leq \|Q(I-Q)^{-1}\|_\infty \| Q(I-Q)^{-1}+ (I-Q)^{-1}\tp+Q (I-Q)^{-3}) \|_F \, .$$
Since all the entries of the matrices on the right hand side are non-negative, we may use \eqref{eq:tpUpperBound} again and the triangle inequality to bound the right hand side of the last inequality by
$$ \|Q(I-Q)^{-1}\|_\infty \Big [ \| Q(I-Q)^{-1} \|_F + \| Q(I-Q)^{-2} \|_F + \| Q (I-Q)^{-3} \|_F \Big ] \, .$$
By \eqref{eq:QIQnorm} of \cref{lem:SpectrumQ} we see that the fourth term dominates the terms in the parenthesis, and is $(1+o(1))(\eps_1/\eps_2)^3$. The sup-norm is bounded using \cref{cor:TauSupNormBound}. We conclude that 
$$ \big \| \tp - Q(I-Q)^{-1} \big \|_F \pe \frac{\log(V)^{10}}{\eps_2 \eps_1^2 V } (\eps_1/\eps_2)^3 =  \frac{\eps_1 \log(V)^{10} }{\eps_2^4 V} \, .$$
\end{proof}

\subsection{A useful inequality}\label{sec:BoundingError1}

\begin{lemma} \label{lem:Error1} Suppose $\omega_{p_1}\in\good$. Then for any symmetric matrix $W$ indexed by $\comp_{p_1,M}$ with zeroes on the diagonal we have
$$\Big | \sum_{a\neq c,b\neq d} W_{a,c} W_{b,d}\Big [ \proba_{p_2^*}(a \lr b)\proba_{p_2^*}(c \lr d) - \proba_{p_2^*}({a \lr b \nlr c \lr d}) \Big ] \Big | \pe   \frac{\|W\|_F^2 \eps_1^2 \log(V)^{10}}{\eps_2^5 V}  \, .$$

\end{lemma}
We first prove two auxiliary lemmas.

\begin{lemma}\label{lem:sup2Star} Suppose $\omega_{p_1}\in\good$. Then for any $e,f\in \comp_{p_1,M}$ we have
$$ \sum_{b} \proba_{p_2^*}( b \lr e) \proba_{p_2^*} (b \lr f) - {\bf 1}_{e=f} \pe  \frac{\log(V)^{10}}{\eps_2^2 \eps_1 V}\, .$$
\end{lemma}
\begin{proof} The term ${\bf 1}_{e=f}$ on the left hand side cancels the term $b=e=f$ in the sum when $e=f$. Furthermore, if $e\neq f$ and $b\in \{e,f\}$ we simply bound $\proba_{p_2^*}( b \lr e)\proba_{p_2^*} (b \lr f) \leq \|\tp\|_\infty$ which, by \cref{cor:TauSupNormBound} together with \eqref{eq:tpUpperBound}, is $o\left ( \frac{\log(V)^{10}}{\eps_2^2 \eps_1 V}\right )$ since $\eps_2 \ll \eps_1$.

Hence we may sum over $b\not \in \{e,f\}$. We proceed as usual using \eqref{eq:tpUpperBound} and the union bound
\begin{eqnarray*} \sum_{b \not \in \{e,f\}} \proba_{p_2^*}( b \lr e) \proba_{p_2^*} (b \lr f) &\leq& 
\sum_{b} \sum_{k_1, k_2 = 1}^\infty Q^{k_1}(e,b)Q^{k_2}(b,f) = \sum_{k_1, k_2 = 1}^\infty Q^{k_1+k_2}(e,f) \\ &=& Q^2(e,f) + \sum_{k=3}^\infty (k-1)Q^k(e,f) \\ &=& Q^2(e,f) + \left [Q (\sum_{k=3}^\infty (k-1)Q^{k-2}) Q \right ] (e,f) \, . 
\end{eqnarray*} 
We bound $Q^2(e,f) \leq \max_{a} \sum_b Q^2_{a,b}$ (as done in \eqref{eq:Q2-CS}) and the second term using Cauchy-Schwartz by
$$\left (\sum_{c,d} Q_{e,c}^2 Q_{d,f}^2\right )^{1/2} \left \|\sum_{k=3}^\infty (k-1)Q^{k-2} \right \|_F \, .$$
The $i$th eigenvalue of $\sum_{k=3}^\infty (k-1)Q^{k-2}$ equals 
$$\sum_{k=3}^\infty (k-1) \lambda_i(Q)^{k-2} = (2-\lambda_i(Q))\lambda_i(Q)(1-\lambda_i(Q))^{-2}\, .$$
Thus \cref{lem:SpectrumQ} implies that $ \|\sum_{k=3}^\infty (k-1)Q^{k-2}\|_F \pe (1-\lambda_1)^{-2}$. Thus 
$$ \left (\sum_{c,d} Q_{a,c}^2 Q_{d,b}^2 \right )^{1/2} \left \|\sum_{k=3}^\infty (k-1)Q^{k-2} \right \|_F \pe (1-\lambda_1)^{-2} \max_a \sum_b Q_{a,b}^2  \, .$$
By \eqref{eq:lambda1Upper} of \cref{lem:SpectrumQ} and \eqref{eq:maxQRowsum} this is at most of order $\log(V)^{10}/(\eps_2^2 \eps_1 V)$ as required.
\end{proof}

\begin{lemma}\label{lem:Diagram8W} 
Suppose $\omega_{p_1}\in\good$ and that $W$ is a nonnegative symmetric matrix indexed by $\comp_{p_1,M}$ with zeroes on the diagonal. Then 
\begin{align*}\sum_{a\neq c, b\neq d, e ,f} & W_{a,c} W_{b,d}  \proba_{p_2^*}(c \lr e)\proba_{p_2^*}(e \lr d)  \proba_{p_2^*}(a \lr f)  \proba_{p_2^*}(f \lr b)  \proba_{p_2^*}(f \lr e) \\ &\pe \frac{\|W\|_F^2 \eps_1^2\log(V)^{10}}{\eps_2^5 V}  \, . \end{align*}
\end{lemma}
\begin{proof}
If $e \neq f$ we bound $\proba_{p_2^*}(f \lr e)\leq \|\tp\|_\infty$ and use that $W$ has non-negative entries and zeroes on the diagonal to upper bound the sum by $\|\tp\|_\infty \Tr(W(I + \tp)^2W(I+\tp)^2)$. We open the parenthesis in the trace and use $\Tr(AB)\leq \|A\|_F \|B\|_F$ (which follows by Cauchy-Schwartz) to obtain
\begin{eqnarray*}
\Tr(W(I + \tp)^2W(I+\tp)^2) \pe \|W\|_F^2(1 + \|\tp\|_F + \|\tp\|_F^2 + \|\tp\|_F^3+ \|\tp\|_F^4)  \, .
\end{eqnarray*} 
By \eqref{eq:QIQnorm} of \cref{lem:SpectrumQ} the term $\|\tp\|_F^4$ dominates the parenthesis and is of order $(\eps_1/\eps_2)^4$. Hence, together with \cref{cor:TauSupNormBound} we bound
\begin{eqnarray*} \|\tp\|_\infty \Tr(W(I + \tp)^2W(I+\tp)^2) \pe    \frac{\|W\|_F^2 \eps_1^2 \log(V)^{10}}{\eps_2^5 V} \, .\end{eqnarray*}
Therefore it remains to handle the case $e=f$. It suffices to bound from above the sum
$$ S := \sum_{a\neq c,b\neq d, e} |W_{a,c}| |W_{b,d}|  \proba_{p_2^*}(c \lr e)\proba_{p_2^*}(e \lr d)  \proba_{p_2^*}(a \lr e)  \proba_{p_2^*}(e \lr b)  \, .$$
We proceed using Cauchy-Schwartz
\begin{eqnarray*} S^2 &\leq& \sum_{a\neq c,b\neq d} |W_{a,c}|^2||W_{b,d}|^2 \sum_{a\neq c,b\neq d} \Big ( \sum_{e} \proba_{p_2^*}(c \lr e)\proba_{p_2^*}(e \lr d)  \proba_{p_2^*}(a \lr e)  \proba_{p_2^*}(e \lr b)\Big )^2  \\ &=& \|W\|_F^4 \sum_{\substack{a\neq c,b\neq d, e,f}} \proba_{p_2^*}(a \lr e)\proba_{p_2^*}(a \lr f)\proba_{p_2^*}(b \lr e)\proba_{p_2^*}(b \lr f) \\ &\ & \qquad\qquad\qquad\quad \cdot \proba_{p_2^*}(c \lr e)\proba_{p_2^*}(c \lr f)\proba_{p_2^*}(d \lr e)\proba_{p_2^*}(d \lr f) \, . \end{eqnarray*}
We bound the sum $\sum _{b \neq d} \proba_{p_2^*}(b \lr e)\proba_{p_2^*}(b \lr f)\proba_{p_2^*}(d \lr e)\proba_{p_2^*}(d \lr f)$
by 
$$ \sum_{b,d} \proba_{p_2^*}(b \lr e)\proba_{p_2^*}(b \lr f)\proba_{p_2^*}(d \lr e)\proba_{p_2^*}(d \lr f) - {\bf 1}_{e=f} $$
simply by considering the term $b=d=e=f$, non-existing in the first sum, and giving $1$ in the second sum. We apply \cref{lem:sup2Star} and bound this from above by
$$ \Big [ {\bf 1}_{e=f} + O \left ( \frac{\log(V)^{10}}{\eps_2^{2} \eps_1 V}\right ) \Big ]^2 - {\bf 1}_{e=f} \pe  \frac{\log(V)^{10}}{\eps_2^{2} \eps_1 V} \, .$$
And what is left we bound by
\begin{eqnarray*} \sum_{a\neq c,e,f} &\,& \proba_{p_2^*}(a \lr e)\proba_{p_2^*}(a \lr f)\proba_{p_2^*}(c \lr e)\proba_{p_2^*}(c \lr f) = \sum_{a\neq c} [(I + \tp)^2(a,c)]^2\\ &\leq& \Tr [(I+\tp)^4] - \Tr[(I+\tp)^2] = \Tr(2\tp+5\tp^2+4\tp^3+\tp^4) \pe (\eps_1/\eps_2)^4 \, , \end{eqnarray*}
where for the first inequality we used the fact that $[(I+\tp)^2]_{a,a}^2 \geq [(I+\tp)^2]_{a,a}$ (since both sides are at least $1$) and the last inequality is due to \eqref{eq:tpUpperBound} and \eqref{eq:QIQnorm} of \cref{lem:SpectrumQ}. We put all these together and obtain that
$$ S^2 \pe \|W\|_F^4 (\eps_1/\eps_2)^4  \frac{\log(V)^{10}}{\eps_2^{2} \eps_1 V} \, ,$$
that is, 
$$ S \pe \|W\|_F^2 \eps_1^2  \frac{\log(V)^{10}}{\eps_2^{3}\eps_1^{1/2} V^{1/2}} \pe \|W\|_F^2 \eps_1^2  \frac{\log(V)^{10}}{\eps_2^5 V} \, ,$$
where the last inequality is since $\eps_1, \eps_2$ satisfy \eqref{eq:eps2Fixed}.
\end{proof}

\begin{proof}[Proof of \cref{lem:Error1}] %
We first use Aizenman's ``off'' method. We have
$$ \proba_{p_2^*}({a \lr b \nlr c \lr d}) = \sum_{A : a,b\in A, c,d\not\in A} \proba_{p_2^*}(\C(a)=A) \proba_{p_2^*}(c \lr d \mid \C(a)=A) \, .$$
Furthermore, $\proba_{p_2^*}(c \lr d \mid \C(a)=A) = \proba_{p_2^*}(c \lr d \off A)$ as long as $c,d \not \in A$. Hence 
$$ \proba_{p_2^*}({a \lr b \nlr c \lr d}) = \sum_{A : a,b\in A} \proba_{p_2^*}(\C(a)=A) \proba_{p_2^*}(c \lr d \off A) \, .$$
Using the BK inequality we bound
$$ \proba_{p_2^*}(c \lr d \off A) \geq \proba_{p_2^*}(c \lr d) - \sum_{e\in A} \proba_{p_2^*}(c \lr e)\proba_{p_2^*}(e \lr d) \, .$$
By changing order of summation we obtain
\begin{eqnarray*} \sum_{A : a,b\in A} \hspace{-.4cm} &\ & \proba_{p_2^*}(\C(a)=A) \sum_{e\in A} \proba_{p_2^*}(c \lr e)\proba_{p_2^*}(e \lr d) \\ &=& \sum_{e} \proba_{p_2^*}(c \lr e)\proba_{p_2^*}(e \lr d) \sum_{A : a,b,e \in A} \proba_{p_2^*}(\C(a)=A) \\ &=& \sum_{e} \proba_{p_2^*}(c \lr e)\proba_{p_2^*}(e \lr d) \proba_{p_2^*}(a \lr b, a \lr e) \, .\end{eqnarray*}
We put all these together and use the tree graph inequality \eqref{eq:treegraph1} to obtain
\begin{eqnarray*} \proba_{p_2^*}({a \lr b \nlr c \lr d}) &\geq& \proba_{p_2^*}(a \lr b)\proba_{p_2^*}(c \lr d) \\ &-& \sum_{e ,f} \proba_{p_2^*}(c \lr e)\proba_{p_2^*}(e \lr d)  \proba_{p_2^*}(a \lr f)  \proba_{p_2^*}(f \lr b)  \proba_{p_2^*}(f \lr e) \, .\end{eqnarray*}
We use this to bound the desired sum 
$$\sum_{a\neq c,b\neq d} W_{a,c} W_{b,d}\Big [ \proba_{p_2^*}(a \lr b)\proba_{p_2^*}(c \lr d) - \proba_{p_2^*}({a \lr b \nlr c \lr d}) \Big ] $$
by
$$ \sum_{a\neq c, b\neq d, e ,f} |W_{a,c}| |W_{b,d}|  \proba_{p_2^*}(c \lr e)\proba_{p_2^*}(e \lr d)  \proba_{p_2^*}(a \lr f)  \proba_{p_2^*}(f \lr b)  \proba_{p_2^*}(f \lr e) \, ,$$
and invoking \cref{lem:Diagram8W} finishes the proof.
\end{proof}

\subsection{Proof of \cref{thm:QuenchedContraction}}\label{sec:ProofQuenchedContraction}

We denote $W:=W(p_1,\cnst)$ as defined in \eqref{eq:defWMatrix} and \eqref{eq:defCnst}. Out goal is to bound from above
$$ \sum_{a,b,c, d} W_{a,c} W_{b,d} \proba_{p_2^*}({a \lr b \nlr c \lr d}) \pe \frac{\|W\|_F^2 \eps_1^2\log(V)^{10}}{\eps_2^5 V}  \, .$$
Since $W$ has null diagonal  \cref{lem:Error1} implies that it suffices to prove that 
\be\label{eq:defConvolutionSum} \sum_{a,b,c,d} W_{a,c} W_{b,d} \proba_{p_2^*}(a \lr b)\proba_{p_2^*}(c \lr d) \pe   \frac{\|W\|_F^2 \eps_1^2\log(V)^{10}}{\eps_2^5 V} \, .\ee
The left hand side of the above equals 
$$ \Tr((I+\tp) W (I+\tp) W) =  \Tr(\tp W \tp W) + \Tr(\tp W^2) + \Tr(W\tp W) + \Tr(W^2) \, .$$
The first term on the right hand side satisfies
\begin{eqnarray*} \Tr(\tp W \tp W) &=& \Tr((\tp - \ti{Q})W \tp W) + \Tr(\ti{Q}W \tp W) \\ &=& \Tr((\tp - \ti{Q})W \tp W) + \Tr(\ti{Q}W(\tp - \ti{Q})W) + \Tr(\ti{Q} W \ti{Q} W) \, .
\end{eqnarray*}
It is straightforward that for any symmetric rank $1$ matrix and another symmetric matrix, the trace of the square of their product equals the square of the trace of their product. Hence our choice of $\cnst$ in \eqref{def:Const} guarantees that the last term vanishes, i.e., $\Tr(\ti{Q} W \ti{Q} W)=0$. Since $\Tr(AB)\leq \|A\|_F \|B\|_F$ by Cauchy-Schwartz, we obtain from all of the above the bound 
\begin{eqnarray*}
\sum_{a,b,c,d} W_{a,c} W_{b,d} \proba_{p_2^*}(a \lr b)\proba_{p_2^*}(c \lr d) \leq \|W\|_F^2 \Big ( \|\tp - \ti{Q}\|_F (\|\tp\|_F + \|\ti{Q}\|_F) + 2 \|\tp\|_F + 1\Big ) \, . \end{eqnarray*}
We bound each term on the right hand side separately. By \eqref{eq:tpUpperBound} and \eqref{eq:QIQnorm} of \cref{lem:SpectrumQ} we bound $\|\tp\|_F  \pe \eps_1/\eps_2$ and the same bound for $\|\ti{Q}\|_F$.
Using the triangle inequality, \cref{lem:mainEstimate} and \eqref{eq:QIQRank1} of \cref{lem:SpectrumQ} we bound
\be\label{eq:mainestCor} \|\tp - \ti{Q}\|_F \leq \|\tp - Q(I-Q)^{-1}\|_F + \| Q(I-Q)^{-1} - \ti{Q} \|_F \pe  \frac{\eps_1\log(V)^{10}}{\eps_2^4 V} \, ,\ee
where we also used \eqref{eq:eps2Fixed}. We obtain that  
$$ \sum_{a,b,c,d} W_{a,c} W_{b,d} \proba_{p_2^*}(a \lr b)\proba_{p_2^*}(c \lr d) \pe \|W\|_F^2 \eps_1^2 \cdot  \frac{\log(V)^{10}}{\eps_2^5 V} \, ,$$
proving \eqref{eq:defConvolutionSum} and concluding the proof. \qed

\subsection{Positioning $p_2^*$ in the subcritical regime} \label{sec:P2positioning}
\cref{thm:QuenchedContraction} is an upper bound on a conditional probability (given $\omega_{p_1}\in \good$)  valid whenever $\eps_2=\eps_2(\omega_{p_1})$ (which in turn defines $p_2^*$) satisfies \eqref{eq:eps2Fixed}. On the other hand, \cref{thm:contraction} is an unconditional statement about a fixed probability $p_2=p_c-V^{\delta_2-1/3}$, so we need to show that if $\omega_{p_1}\in\good$, then $p_2$ satisfies all the necessary assumptions to apply \cref{thm:QuenchedContraction} and that $p_2^*$ and $p_2$ correspond to the same position in the subcritical regime so that the estimate in \cref{thm:QuenchedContraction} will be of the same order as the desired estimate of \cref{thm:contraction}. These two issues are very much related and we deal with them now.

Assume that $p_2>p_1$ fixed numbers that are defined as in \cref{thm:contraction}, i.e., $p_1 = p_c - V^{\delta_1-1/3}$ and $p_2=p_c-V^{\delta_2-1/3}$ where $\delta_2\in(0,\delta_1/4)$. Let $\eps_2^* = \eps_2^*(\omega_{p_1})$ be defined by
\be\label{eq:eps2QuenchedDef} p_2 = p_1 + {(1-p_1)(1- \eps_2^*/\eps_1) \over \lambda_1(\Delta)} \, .\ee
This seems very similar to \eqref{def:p2} and it follows immediately that $\eps_2^*<\eps_1$ since $p_2 > p_1$; however $\eps_2^*$ does not necessarily satisfy \eqref{eq:eps2Fixed} and is not even necessarily positive. Our goal is to show that \eqref{eq:eps2Fixed} occur for $\eps_2^\star$ whenever $\omega_{p_1}$ is in a set of probability $1-o(1)$. Let $\good'$ be the event that $\omega_{p_1}$ satisfies
\be \label{eq:fixedp2} \proba_{p_2} \Big ( V\chi(p_2)/2 \leq \sum_{A \in \comp_{p_2, M}} |A|^2 &-& |N(p_1,p_2; M)| \leq 2V\chi(p_2)  \mid \omega_{p_1} \Big ) \geq 9/10 \, . \ee
\cref{lem:ConcentrationFixedp2} together with \cref{lem:Np1p2Bound}, and \cref{thm:subcriticalBCHSS} show that the unconditional probability of the event on the left hand side of \eqref{eq:fixedp2} is $1-o(1)$, and so the probability that $\omega_{p_1}$ satisfies \eqref{eq:fixedp2} is also $1-o(1)$. Hence
\be \label{eq:WHPgood2}
\proba_{p_1}(\good') = 1-o(1) \, .
\ee

\begin{theorem}\label{thm:eps2starPosition} If $\omega_{p_1}\in \good \cap \good'$, then
$${1 \over \chi(p_2)} \pe \eps_2^* \pe {1 \over \chi(p_2)} \, .$$    
\end{theorem}

The proof will be obtained by comparing the conditional distributions of $\sum_{A \in \comp(p,p_1,M)} |A|^2$ for $p=p_2^*$ and $p=p_2$. The following lemma handles the $p=p_2^*$ part.

\begin{lemma} \label{lem:ConcentrationQuenchedp2} %
Assume that $m$ and $n$ are large enough.
Condition on $\omega_{p_1}\in\good$ and let $\eps_2=\eps_2(\omega_{p_1})$ satisfy  \eqref{eq:eps2Fixed} and set $p_2^*=p_2^*(\omega_{p_1}, \eps_2)$ by \eqref{def:p2}.
Then conditioned on $\omega_{p_1}\in\good$ we have 
$$  \proba \left ( \frac{V}{2\eps_2}\leq \sum_{A \in \comp(p_2^*, p_1, M)} |A|^2 \leq \frac{2V}{\eps_2}\right ) \geq 9/10 .$$ %
\end{lemma}
\begin{proof} The conditional first moment of $\sum_{A \in \comp(p_2^*, p_1, M)} |A|^2$ is 
\be\label{eq:quenchedSuc} \sum_{a,b \in \comp_{p_1,M}} |a||b|\proba_{p_2^*}(a \lr b) \geq \sum_{a,b} |a||b| \ti{Q}(a,b) + \sum_{a,b}|a||b|(\proba_{p_2^*}(a \lr b) - \ti{Q}(a,b)) \, . \ee
Let $v$ be the unit eigenvector of $Q$ corresponding to $\lambda_1(Q)$ and $x$ the vector $x_a=|a|$, so that the first term in the above equals 
\[ {\lambda_1(Q) \over 1 - \lambda_1(Q)} \sum_{a,b} |a||b| v_a v_b = {\lambda_1(Q) \over 1 - \lambda_1(Q)} \langle x,v \rangle^2 = (1+o_{m}(1)) V/\eps_2 \, ,\]
by \eqref{eq:xv} of \cref{lem:SpectrumQ}. To handle the second term on the right hand side of \eqref{eq:quenchedSuc} we have to be careful with the diagonal. We bound it from above by 
$$ \sum_{a\neq b} |a||b|( \tp(a,b) - \ti{Q}(a,b))^2 + \sum_a |a|^2 ( 1 - \ti{Q}(a,a)) \leq  \|\tp - \ti{Q}\|_F\sum_{a} |a|^2  + \sum_{a}|a|^2 \, ,$$
where we used Cauchy-Schwartz and the fact that the entries of $\ti{Q}$ are non-negative. Since $\omega_{p_1} \in \good$, by \eqref{eq:a2b2} we have that $\sum_a |a|^2 \pe V\chi(p_1)$, and using \eqref{eq:mainestCor} we get that 
$$ \sum_{a,b}|a||b|(\proba_{p_2^*}(a \lr b) - \ti{Q}(a,b)) \pe {V \over \eps_1} \Big ( 1 + \frac{\eps_1\log(V)^{10}}{\eps_2^4 V} \Big ) \, ,$$
and since $\eps_2$ satisfies \eqref{eq:eps2Fixed} this upper bound is $o(V/\eps_2)$. We obtain that the conditional first moment is $(1+o_{m}(1)) V/\eps_2$. 

The conditional second moment of $\sum_{A \in \comp(p_2^*, p_1, M)} |A|^2$ is 
$$ \E_{p_2^*}\sum_{A,B} |A|^2 |B|^2 = \E_{p_2^*}\sum_{A} |A|^4 + \E_{p_2^*}\sum_{A \neq B} |A|^2 |B|^2 \, .$$
The second term on the right hand side is bounded using the BK inequality 
$$ \E_{p_2^*}\sum_{A \neq B} |A|^2 |B|^2 = \sum_{a,b,c,d } |a||b||c||d|\proba_{p_2^*}(a \lr b \nlr c \lr d) \leq \Big  (\sum_{a,b} |a||b|\proba_{p_2^*}(a \lr b) \Big)^2 \, ,$$
which is just the conditional first moment squared. Thus it suffices to show that 
$$ \E_{p_2^*}\sum_{A} |A|^4 = o(V^2/\eps_2^2) \ .$$
Indeed,
$$ \E_{p_2^*}\sum_{A} |A|^4 = \sum_{a,b,c,d} |a||b||c||d|\proba_{p_2^*}(a \lr b \lr c \lr d) \, .$$
It is more convenient to sum the cases that either $a=c$ or $b=d$ separately; these two are identical by symmetry, so it suffices to bound the case $a=c$, that is,
\be\label{eq:AequalC} \sum_{a,b,d}|a|^2 |b||d| \proba_{p_2^*}(a \lr b  \lr d) \leq \sum_{a,b,d,e}|a|^2 |b||d|\proba_{p_2^*}(a \lr e) \proba_{p_2^*}(e \lr d)\proba_{p_2^*}(e \lr b) \, ,\ee
where the last inequality is the tree graph inequalities \eqref{eq:treegraph1}. If $a\neq e$ we bound using \cref{cor:TauSupNormBound}
$$ \proba_{p_2^*}(a \lr e) \leq \|\tp\|_\infty \pe  \frac{\log (V)^{10}}{\eps_2 \eps_1^2 V}\, ,$$
so if we write $N$ for the matrix $N_{b,d}=|b||d|$ we may first sum over $b,d$ to bound
\begin{eqnarray*} \sum_{a,b,d}|a|^2 |b||d| \proba_{p_2^*}(a \lr b  \lr d) &\pe& \left ( \frac{\log (\eps_1^3 V)}{\eps_2 \eps_1^2 V}\right) \Tr((I+\tp)N(I+\tp)) \sum_a |a|^2  \\
&\pe& \left ( \frac{\log (\eps_1^3 V)}{\eps_2 \eps_1^2 V}\right) V\eps_1^{-1} \Tr(N + \tp N + N\tp + \tp N \tp) \, ,\end{eqnarray*}
where we used \eqref{eq:a2b2}. Using \eqref{eq:a2b2} again gives that $\Tr(N) \pe V\eps_1^{-1}$ and since $\Tr(AB)\leq \|A\|_F\|B\|_F)$ we bound the last sum by 
$$ \left ( \frac{\log (V)^{10}}{\eps_2 \eps_1^2 V}\right) V\eps_1^{-1} \Big [ V\eps_1^{-1} + 2 V\eps_1^{-1} \|\tp\|_F + \|N\|_F \|\tp\|_F^2 \Big ] \, .$$
Again \eqref{eq:a2b2} shows that $\|N\|_F \pe V\eps_1^{-1}$ and $\|\tp\|_F $ is bounded using \eqref{eq:QIQnorm} of \cref{lem:SpectrumQ} and \eqref{eq:tpUpperBound} by $\eps_1/\eps_2$; so the last term in the parenthesis dominates and is of order $V\eps_1/\eps_2^2$. All this give an upper bound of order  
$$ {V \log (V)^{10} \over \eps_2^3 \eps_1^2} \ll V^2 / \eps_2^2 \, ,$$
by \eqref{eq:eps2Fixed}. If $a=e$, then the right hand side of \eqref{eq:AequalC} becomes 
\begin{eqnarray*} \sum_{a,b,d} |a|^2 |b||d|\proba_{p_2^*}(a \lr d)\proba_{p_2^*}(a \lr b) &\leq& \max_a |a|^2 \Tr(N(I+\tp)^2)  \\ &\pe& \eps_1^{-4} \log(V)^2 \Tr(N+2N\tp + N\tp^2) \, ,\end{eqnarray*}
where we used that $\max_a |a| \pe  \eps_1^{-2} \log(V)$ by \eqref{eq:maxa}. We use the previous bounds so that the last term in the trace dominates and is bounded by $\|N\|_F \|\tp\|_F^2 \pe V\eps_1/\eps_2^2$. All this gives a bound of 
$$ \eps_1^{-4} \log(V)^2 V\eps_1/\eps_2^2 \, ,$$
which is $o(V^2/\eps_2^2)$ since $\eps_1\gg V^{-1/3}\log(V)^{10}$. This concludes our treatment of the case $a=c$ or $b=d$. Therefore it remains to bound 
$$\sum_{a\neq c, b \neq d} |a||b||c||d|\proba_{p_2^*}(a \lr b \lr c \lr d)  $$
which we may bound by the tree graph inequalities \eqref{eq:treegraph1} (more precisely the variation of \eqref{eq:treegraph1} to four vertices) by the sum of 
\begin{eqnarray*}  \sum_{a\neq c, b\neq d,e,f} |a||b||c||d| \proba_{p_2^*}(a \lr e)\proba_{p_2^*}(e \lr b) \proba_{p_2^*}(e\lr f)   \proba_{p_2^*}(f \lr c) \proba_{p_2^*}(f \lr d) 
\end{eqnarray*}    
and 
\begin{eqnarray*}  \sum_{a\neq c, b\neq d,e,f} |a||b||c||d| \proba_{p_2^*}(a \lr f)\proba_{p_2^*}(f \lr e) \proba_{p_2^*}(e \lr b) \proba_{p_2^*}(e \lr c) \proba_{p_2^*}(f \lr d) \, .
\end{eqnarray*}    
Both sums are handled similarly using \cref{lem:Diagram8W} with the matrix $W$ just being $(|a||b|{\bf 1}_{a\neq b})_{a,b}$ so that $\|W\|_F=O(V\eps_1^{-1})$. This gives a bound of 
$$ V^2\eps_1^{-2} \eps_1^2 \cdot O \left ( \frac{\log(V)^{10} }{\eps_2^5 V} \right  ) = o(V^2/\eps_2^2) \, ,$$
by \eqref{eq:eps2Fixed}, which concludes our proof. 
\end{proof}

\begin{proof}[Proof of \cref{thm:eps2starPosition}] 

We assume that $\omega_{p_1}\in \good$ is fixed and set 
$$ \eps_2^{(a)} := c/\chi(p_2) \qquad \eps_2^{(b)} := C/\chi(p_2) \, ,$$
where $c>0$ is a small constant and $C>0$ is a large constant both will be chosen later. We define $p_{2,a}^{*}$ and $p_{2,b}^{*}$ by
$$ p_{2,(a)}^{*} := p_1 + {(1-p_1)(1- \eps_2^{(a)}/\eps_1) \over \lambda_1(\Delta)} \qquad p_{2,(b)}^{*} := p_1 + {(1-p_1)(1- \eps_2^{(b)}/\eps_1) \over \lambda_1(\Delta)} \, .$$ 
We apply \cref{lem:ConcentrationQuenchedp2} and choose the constants $c$ and $C$ small enough and large enough respectively so that the events 
$$ \sum_{A \in \comp(p_{2,(a)}^{*}, p_1, M)} |A|^2 \geq 3 V\chi(p_2) \, ,$$
and
$$ \sum_{A \in \comp(p_{2,(b)}^{*}, p_1, M)} |A|^2  \leq  V\chi(p_2)/3 \, ,$$ 
occur with probability at least $3/4$ conditioned on $\omega_{p_1}\in \good$, so both occur with conditional probability at least $1/2$. On the other hand, we have that 
$$ \sum_{A \in \comp_{p_2,M}} |A|^2 - |N(p_1,p_2;M)| = \sum_{A \in \comp(p_2,p_1, M)} |A|^2 \, ,$$
since both sides count pairs of vertices that belong to components of $\comp_{p_1, M}$ so that there exists a $p_2$-open path between them such that every vertex on the path belongs to a component of $\comp_{p_1, M}$. For this reason, the map $p\mapsto \sum_{A \in \comp(p, p_1, M)}|A|^2$ is increasing on $(p_1,1]$. We deduce from all these and \eqref{eq:fixedp2} that 
$$ p_{2,(b)}^{*} \leq p_2 \leq p_{2,(a)}^{*} \, ,$$
as long as $\omega_{p_1}\in \good$. In other words, $\eps_2^{(a)} \leq \eps_2^* \leq \eps_2^{(b)}$, concluding the proof.
\end{proof}

\subsection{Proof of \cref{thm:contraction}}\label{sec:ProofContraction}
Since $\W(p_2)$ is a minimizer, it suffices to bound it from above by $\|W(p_2,\cs)\|_F^2$ for some $\cs>0$. Recall that
\be\label{eq:defWAgain} \|W(p_2,\cs)\|_F^2 = \sum_{A \neq C \in \comp_{p_2, M}} \Big ( \Delta_{A,C} -{\cs m |A||C| \over V} \Big )^2 \, .\ee
Note that connected components in $\comp_{p_1}$ of size smaller than $M$ may contribute to \eqref{eq:defWAgain}. Indeed, this happens when they are contained in a cluster belonging to $\comp_{p_2,M}$. On the other hand, they do not contribute to the left hand-side of \cref{thm:QuenchedContraction}.
Therefore, our first order of business to bound their contribution. 
\begin{lemma}\label{pro:SecondMomentSmallMass} Assume the setting of \cref{thm:contraction}. For any $\cs>0$ we have
\[ \|W(p_2,\cs)\|_F^2 -  \sum_{\substack{ a,b,c,d \in \comp_{p_1,M} \\ a \lr b \nlr c\lr d }} W(p_1,\cs)_{a,c} W(p_1,\cs)_{b,d} \leq O_{\proba}\big ((1+\cs^2)\sqrt{M}\chi(p_2)^3/\chi(p_1)^2 \big ) \, .\]
\end{lemma}
\begin{remark} In the above lemma the notation $\leq O_{\proba}(\cdot)$ means that we upper-bound the positive part of the left-hand side in probability. 
\end{remark}

\begin{proof}

First note that we can rewrite the left hand side of the lemma as 
$$ \|W(p_2,\cs)\|_F^2 - \sum_{A, C \in \comp_{p_2,M}} \sum_{\substack{ a,b,c,d \in \comp_{p_1,M} \\ a \lr b \nlr c\lr d \\ a,b\subset A, c,d\subset C}} W(p_1,\cs)_{a,c} W(p_1,\cs)_{b,d} \, ,$$
since the case $A=C$ counts for the possibility that $\{a,b,c,d\}$ are in the same component of $\comp_{p_2,M}$ but there is component of $\comp_{p_1}$ of size smaller than $M$ ``breaking'' it so that $\{a,b\} \nlr \{c,d\}$. We then note that the contribution of the case $A=C$ in the above display is negative. Indeed, 
$$ \sum_{\substack{ a,b,c,d \in \comp_{p_1,M} \\ a \lr b \nlr c\lr d \\ a,b\subset A, c,d\subset A}} W(p_1,\cs)_{a,c} W(p_1,\cs)_{b,d}= \sum_{\substack{A'\neq C'\in \comp(p_2^\star,p_1,M)\\A',C'\subset A }}\left  (\sum_{a\in A',c\in C'} W(p_1,t)_{a,c} \right )^2\geq 0. $$

Thus, it suffices to prove that 
$$\|W(p_2,\cs)\|_F^2 - \sum_{A\neq C \in \comp_{p_2,M}} \sum_{\substack{ a,b,c,d \in \comp_{p_1,M} \\ a \lr b , c\lr d \\ a,b\subset A, c,d\subset C}} W(p_1,\cs)_{a,c} W(p_1,\cs)_{b,d} \leq O_{\proba}((1+\cs^2)\sqrt{M}\chi(p_2)^3/\chi(p_1)^2 ) \, ,$$
which we now prove. We decompose $\|W(p_2,c)\|_F^2$ as $(cm/V)^2S_0-(2cm/V)S_1+S_2$ where 
\[ S_0=\sum_{u,v,x,y \in V}\1(u\lr^{\omega_{p_2}} v\nlr^{\omega_{p_2}} x\lr^{\omega_{p_2}} y, |\C_{p_2}(u)|\geq M,|\C_{p_2}(y)|\geq M), \]
and
\[ S_1=\sum_{u,y \in V,  (v,v')\in E}\1(u\lr^{\omega_{p_2}} v\nlr^{\omega_{p_2}} v'\lr^{\omega_{p_2}} y,|\C_{p_2}(u)|\geq M,|\C_{p_2}(y)|\geq M), \]
and
\[ S_2=\sum_{(u,u'),(v,v')\in E}\1(u\lr^{\omega_{p_2}} v\nlr^{\omega_{p_2}} v'\lr^{\omega_{p_2}} u',|\C_{p_2}(u)|\geq M,|\C_{p_2}(u')|\geq M) \, . \]
In the above we write $u\stackrel{\omega_{p_2}}{\lr} v$ and $v\stackrel{\omega_{p_2}}{\nlr} x$ to indicate that the connections or lack them occurs in $\omega_{p_2}$.  

We write $\omega(p_1,p_2;M)$ for the set of edges in $\omega_{p_2}$ with both endpoints in components of $\comp_{p_1,M}$ (that is, both endpoints are in $\cup_{a\in \comp_{p_1,M}}a$). We decompose the negative term in a similar fashion. Since $W(p,\cs)_{a,c}=\Delta_{a,c}-(\cs m/V)|a||c|$ we may decompose the negative term as $(\cs m/V)^2S'_0-(\cs m/V)S'_1+S'_2$ where
\[ S'_0=\sum_{u,v,x,y \in V}\1(u\lr^{\omega_{p_1,p_2;M}} v\nlr^{\omega_{p_2}} x\lr^{\omega(p_1,p_2;M)} y), \]
and
\[ S'_1=\sum_{u,y \in V, (v,v')\in E}\1(u\lr^{\omega(p_1,p_2;M)} v\nlr^{\omega_{p_2}} v'\lr^{\omega(p_1,p_2;M)} y), \]
and
\[ S'_2=\sum_{(u,u'),(v,v')\in E}\1(u\lr^{\omega(p_1,p_2;M)} v\nlr^{\omega_{p_2}} v'\lr^{\omega(p_1,p_2;M)} u'). \]
In words, the difference between $S_0,S_1,S_2$ and $S_0',S_1',S_2'$ is that in the latter we additionally require that the connections occur in $\omega(p_1,p_2;M)\subset \omega_{p_2}$. Hence $S'_0\leq S_0$, and $S'_1\leq S_1$, and $S'_2\leq S_2$. Next, recall that $N(p_1,p_2;M)$ is the set of pairs of vertices $\{u,v\}$ such that $u,v$ have a $p_2$-open path connecting them, but any $p_2$-open path connecting them visits a vertex belonging to a component of $\comp_{p_1}$ of size smaller than $M$. If the event
$$ \Big \{ u\lr^{\omega_{p_2}} v\nlr^{\omega_{p_2}} x\lr^{\omega_{p_2}} y, |\C_{p_2}(u)|\geq M,|\C_{p_2}(y)|\geq M \Big \} \setminus \Big \{u\lr^{\omega_{p_1,p_2;M}} v\nlr^{\omega_{p_2}} x\lr^{\omega(p_1,p_2;M)} y \Big \} $$
occurs, then either $u \lr v \nlr x \lr y$ in $\omega_{p_2}$ and $\{u,v\} \in N(p_1,p_2;M)$, or the same event with $\{x,y\} \in N(p_1,p_2;M)$. We deduce that
 \[ \E[S_0-S'_0]\leq 2 \sum_{u,v,x,y} \proba_{p_2}(u \lr v\nlr x\lr y, \{u,v\}\in N(p_1,p_2,M)) \, .\]
The event $\{u,v\}\in N(p_1,p_2;M)$ depends only on the status of the edges touching $\C_{p_2}(u)$, and if in addition $x\lr y$ and $\C_{p_2}(x) \neq \C_{p_2}(u)$, then the path witnessing $x \lr y$ in $\omega_{p_2}$ does not touch $\C_{p_2}(u)$. Hence the BKR inequality implies that
\[ \E[S_0-S'_0]\leq 2\sum_{u,v,x,y} \proba_{p_2}(x\lr y)  \proba(\{u,v\}\in N(p_1,p_2;M))=2V\X(p_2) \E[|N(p_1,p_2;M)|] \, .\]
Thus 
\cref{lem:Np1p2Bound} and \cref{thm:subcriticalBCHSS} imply that the last term is bounded by order $V^2  \X(p_2)^3 \sqrt{M}/ \X(p_1)^2$, so that $(\cs m/V)^2 \E[S_0-S_0'] \pe \cs^2  \X(p_2)^3 \sqrt{M}/ \X(p_1)^2$.
A similar analysis with the BKR inequality gives
\be\label{eq:S2S2'} \E[S_2-S'_2]\pe \sum_{\substack{(u,u'),\\(v,v')}} \proba(\{u,v\}\in N(p_1,p_2;M), |\C_{p_2}(u)|\geq M) \proba(u'\lr v', |C_{p_2}(u')|\geq M) \, .\ee

We split the last sum according to whether or not the distance between the edges $(u,v)$ and $(u',v')$ has reached the plateau or not. If  $\langle u-v\rangle \leq V/\chi(p_2)^{1/(d-2)}$, then we ignore the event $\{u,v\}\in N(p_1,p_2;M)$ and use the following claim.

\begin{claim}\label{clm:TwoEdgesLargeCompShort} 
$$\sum_{\substack{(u,v), (u',v')\in E \\ \l u - u'\r^{d-2} \leq V/\chi(p)}} \proba_p(u \lr u', |\C(u)|\geq M) \proba_p(v \lr v', |\C(v)|\geq M) \pe {V \chi(p)^2 \over M^2} \Big ({V \over \chi(p)}\Big )^{1/(d-2)} \, .$$
\end{claim}
\begin{proof} If $u \lr u', |\C(u)|\geq \M$ occurs, then the random variable 
$$ Z := \big | z \in V : \exists w \text{ with } \{u \lr w\} \circ \{w \lr z\} \circ \{w \lr u'\} \big | \geq \M \, .$$
By the BK inequality and the union bound we have
$$ \E_p Z \leq \sum_{z,w} \proba_p(u \lr w)\proba_p(w \lr z)\proba_p(w \lr u') = \chi(p) T_2(u-u')  \pe \chi(p) \l u-u' \r^{4-d} \, .$$
where the equality is obtained by summing over $z$ then over $w$, and the last inequality is due to \cref{lem:T345}. Thus, using Markov's inequality, the sum of the claim is bounded above by 
$$ {C \chi(p)^2 \over M^2} \sum_{\substack{(u,v), (u',v')\in E\\ \l u - u'\r \leq (V/\chi(p))^{1/(d-2)}}} \l u-u' \r^{4-d} \l v-v' \r^{4-d} \leq {C V \chi(p)^2 \over M^2} \sum_{w \in V: \l w\r \leq (V/\chi(p))^{1/(d-2)}} \l w \r^{8-2d} \, .  
$$
A basic calculation gives $\sum_{w \in V, \l w \r \leq R} \l w \r ^{8-2d} \pe R$ as long as $d\geq 7$ and the claim follows.
\end{proof}
\noindent So, using \cref{clm:TwoEdgesLargeCompShort}, the sum \eqref{eq:S2S2'} over edges with  $\langle u-v\rangle \leq V/\chi(p_2)^{1/(d-2)}$ is bounded above by
\[ {CV\X(p_2)^2\over M^2} \Big( {V \over \X(p_2)} \Big )^{1/(d-2)} \pe V^{6/5} \chi(p_2)^{9/5} / M^2 \pe \sqrt{M} \X(p_2)^3/\X(p_1)^2, \]
where the first inequality is since $d\geq 7$ and the second is by our choice of $M=V^\mp$ and as long as $\delta_0>0$ (from the statement of \cref{thm:contraction}) is small enough.

 Secondly, for the sum \eqref{eq:S2S2'} over edges with $\langle u-v\rangle \geq V/\chi(p)^{1/(d-2)}$, we drop the $|\C_{p_2}(u')|\geq M$ event and use \cref{thm:plateau} to bound $\proba(u'\lr v')\pe \X(p_2)/V$, and we upper bound the corresponding sum by
 \begin{eqnarray*} \sum_{(u,u'),(v,v')\in E} \proba(\{u,v\}\in N(p_1,p_2;M))\X(p_2)V^{-1}  &\pe& \X(p_2)\E[|N(p_1,p_2,M)|] V^{-1}  
 \\ &\pe&  \sqrt{M} \X(p_2)^3 / \X(p_1)^2 
 \, , \end{eqnarray*}
 using \cref{lem:Np1p2Bound} and \cref{thm:subcriticalBCHSS}. Therefore by the Markov's inequality and recalling $S'_0\leq S_0$; $S'_1\leq S_1$, $S'_2\leq S_2$ we get
 \[ [(\cs m/V)^2S_0-cm/VS_1+S_2]\leq [(\cs m/V)^2S'_0-\cs m/VS'_1+S'_2]+O_{\proba}((1+\cs^2)\sqrt{M}\chi(p_2)^3/\chi(p_1)^2 ), \]
 which concludes the proof. \end{proof}

Our second order of business is to compare the random variables $W(p_1,\cnst)$ and $\W(p_1)$ defined in \eqref{eq:defWMatrix} and \eqref{eq:defW}. Obviously we have that $\W(p_1) \leq W(p_1,\cnst)$ for any configuration $\omega_{p_1}$. The following lemma shows that the converse holds up to a multiplicative constant whenever $\omega_{p_1}\in \good$.

\begin{lemma} \label{lem:WstarCloseToMinimal} If $\omega_{p_1}\in \good$, then
$$ \|W(p_1,\cnst) \|_F^2 \pe  \W(p_1) \, , \qquad \mathrm{and} \qquad  \cnst \pe 1 \, .$$
\end{lemma}
\begin{proof} Let $\cmin>0$ denote the minimizer of  $\inf _{\cs>0} \|W(p_1,\cs )\|_F$ so that $\W(p_1)=\|W(p_1,\cmin)\|_F^2$ . First, by definition \eqref{eq:defWMatrix} and since $\omega_{p_1}\in \good$, using \eqref{eq:a2b2} we have
\be\label{eq:DistanceClose} \|W(p_1,\cnst) - W(p_1,\cmin)\|_F \leq {|\cmin-\cnst| m \over V}\sum    _{a}|a|^2 \pe |\cmin-\cnst| \chi(p_1) \, ,\ee
Hence it suffices to provide a lower bound on $\|W(p_1,\cmin)\|_F$ of the same order. %
We have that
\be |\Tr(\ti{Q} W(p_1,\cmin))| \leq \|\ti{Q}\|_F \|W(p_1,\cmin)\|_F \leq (1+o(1)) {\eps_1 \over \eps_2} \|W(p_1,\cmin)\|_F \, , \label{eq:MinorW(p,t)Step1}\ee
by \eqref{eq:QIQnorm} and \eqref{eq:QIQRank1} of \cref{lem:SpectrumQ}. Furthermore, by definition of $\cnst$ in \eqref{def:Const} we have
\be |\Tr(\ti{Q} W(p_1,\cmin))| = |\Tr(\ti{Q} W(p_1,\cmin)) - \Tr(\ti{Q} W(p_1,\cnst))| = {m |\cmin-\cnst| \over V} \Tr(\ti{Q} N) \, , \label{eq:MinorW(p,t)Step2} \ee
where $N$ is the matrix $N_{a,b}= {\bf 1}_{a \neq b}|a||b|$. By \eqref{def:TildeP2} 
$$ \Tr(\ti{Q}N) = {\lambda_1(Q) \over 1 - \lambda_1(Q)} \sum_{a\neq b} |a||b| v_a v_b = {\lambda_1(Q) \over 1 - \lambda_1(Q)} \langle x,v \rangle^2 - {\lambda_1(Q) \over 1 - \lambda_1(Q)} \sum_{a} |a|^2 v_a^2 \, ,$$
where $v$ is the unit eigenvector of $Q$ corresponding to $\lambda_1(Q)$ and $x$ is the vector $x_a=|a|$. The first term on the right hand side above is dominant and we bound it below by $(1-o_{m}(1)) V/\eps_2$ using \eqref{eq:lambda1Upper}, \eqref{eq:lambda1Lower}, \eqref{eq:xv} of \cref{lem:SpectrumQ}. For the second term is at most order $ (\eps_1/\eps_2) \max_a |a|^2$ which is bounded by $(\eps_2 \eps_1)^{-1} \log V$ by \eqref{eq:maxa} and this is of course $o(V/\eps_2)$ by \eqref{eq:eps2Fixed}. We deduce that $ \Tr(\ti{Q}N) \geq (1-o_{m}(1)) V/\eps_2$ so by \eqref{eq:MinorW(p,t)Step1} and \eqref{eq:MinorW(p,t)Step2} we have %
$$ \|W(p_1,\cmin)\|_F \succeq (1-o_{m}(1)) |\cmin-\cnst|  \chi(p_1).$$
Together with \eqref{eq:DistanceClose}, this gives the first bound of the lemma. The second bound is easier. By Cauchy-Schwartz we have 
$$ \Tr(\ti{Q}\Delta) \leq \|\ti{Q}\|_F \|\Delta\|_F = {\lambda_1(Q) \over 1 - \lambda_1(Q)} \sqrt{\sum_{a,b} \Delta_{a,b}^2} \pe 1/\eps_2 \, ,$$
where the last inequality is due to \eqref{eq:lambda1Upper}, \eqref{eq:lambda1Lower} of \cref{lem:SpectrumQ} and since $\omega_{p_1}\in \good$ using \eqref{eq:Delta2}. We put this bound together with our previous lower bound for $\Tr(\ti{Q}N)$ into the definition of $\cnst$ in \eqref{eq:defCnst}, yielding that $\cnst \pe 1$.
\end{proof}

We now have all the necessary ingredients to show how \cref{thm:QuenchedContraction} implies \cref{thm:contraction}.

\begin{proof}[Proof of \cref{thm:contraction}] First \cref{thm:eps2starPosition} allows us to apply \cref{thm:QuenchedContraction} at the (deterministic) point $p_2$ and obtain that for any $\omega_{p_1}\in \good\cap \good'$ we have
$$ \E \Big [ \sum_{\substack{ a,b,c,d \in \comp_{p_1,M} \\ a \lr b \nlr c\lr d }} W(p_1,\cnst)_{a,c} W(p_1,\cnst)_{b,d}  \mid \omega_{p_1} \Big ] \pe \|W(p_1,\cnst)\|_F^2 (\eps_1/\eps_2)^2 \cdot  \frac{\log(V)^{5}}{\eps_2^3 V} $$
where $\cnst=\cnst(\omega_{p_1}, \eps_2^*)$ is defined in \eqref{eq:defCnst}. By \eqref{eq:WHPgood} and \eqref{eq:WHPgood2}, the probability of  $\omega_{p_1}\in \good\cap\good'$ is $1-o(1)$ so we deduce that 
$$ \sum_{\substack{ a,b,c,d \in \comp_{p_1,M} \\ a \lr b \nlr c\lr d }} W(p_1,\cnst)_{a,c} W(p_1,\cnst)_{b,d} = O_\proba \Big ( \|W(p_1,\cnst)\|_F^2 (\eps_1/\eps_2)^2  \frac{\log(V)^{5}}{\eps_2^3 V} \Big ) \, .$$
Together with \cref{pro:SecondMomentSmallMass} we obtain
$$ \|W(p_2,\cnst)\|_F^2 = O_\proba \Big ( \|W(p_1,\cnst)\|_F^2 (\eps_1/\eps_2)^2  \frac{\log(V)^{5}}{\eps_2^3 V} + (1+{\cnst}^2) \sqrt{M}\chi(p_2)^3/\chi(p_1)^2 \Big ) \, .$$
Invoking \cref{lem:WstarCloseToMinimal} concludes the proof.\end{proof}

\subsection{Proof of \cref{thm:main_goal}} \label{sec:proofMainGoal}
The proof of \cref{thm:main_goal} will follow by performing $6$ applications of \cref{thm:contraction}. Let $\delta_0>0$ be the constant from \cref{thm:contraction}. For each $j=1,\ldots, 7$ we choose number $\delta_1 > \delta_2 > \ldots > \delta_7>0$ such that for all $j=1,\ldots, 6$
$$ \delta_1 < \delta_0 \, , \qquad \delta_{j+1} \in (0,\delta_j/4) \, , \qquad    1/24 - 2\delta_1 > 16\delta_7 \, ,$$
and then set
$$ \eps_j := V^{\delta_j - 1/3} \, , \qquad  p_j := p_c(\Z^d) - \eps_j \, .$$
Invoking \cref{thm:contraction} six times yields
\begin{eqnarray*}
\W(p_7) \eps_7^2 &=& O_\proba \Big (  {\W(p_6) \eps_6^2\log(V)^5 \over \eps_7^3 V} + \sqrt{M} \eps_6^2 / \eps_7 \Big ) = \ldots \\ 
&=& O_\proba \Big (  {\W(p_1) \eps_1^2 \log(V)^{30}\over  (\eps_7^3 V)(\eps_6^3 V)\cdots (\eps_2^3 V)} + \sqrt{M} \sum_{j=1}^6 \eps_j^2/\eps_{j+1} \Big )
\end{eqnarray*}
The easy bound $\W(p_1)\eps_1^2 \pe 1$ follows immediately since $\E\W(p_1) = O(1)$ by \cref{thm:Delta2} and \cref{lem:ConcentrationFixedp2} (setting $\cs=1$ to bound the infimum in $\W(p_1)$). Hence, setting $M=V^{\mp}$ yields that
$$ \W(p_7) \eps_7^2 = O_\proba \Big ( {\log ^{30} V \over (\eps_7^3 V)^6 } + V^{-1/24} \sum_{j=1}^6 V^{2\delta_j - \delta_{j+1}} \Big ) \, .$$
By \cref{thm:subcriticalBCHSS} we have that $\eps_7 \asymp 1/\X(p_7)$, so the first term in the parenthesis is $o((\X(p_7)/V^{1/3})^{16})$ and the second term is bounded by $V^{-1/24 + 2\delta_1}$ which is also $o((\X(p_7)/V^{1/3})^{16})$ by our choice of $\delta_j$'s. This concludes the proof. \qed

\newcommand{\tbs}{\textbackslash}
\newcommand{\arms}{\mathrm{Arms}}
\newcommand{\kriic}{\qiic^{k,r}}

\newcommand{\noKArms}{\text{No-}k\text{-Crossings}}
\newcommand{\noTwoArms}{\text{No-}2\text{-Crossings}}
\newcommand{\noThreeArms}{\text{No-}3\text{-Crossings}}
\newcommand{\noFourArms}{\text{No-}4\text{-Crossings}}

\newcommand{\AllOpen}{\text{AllOpen}}

\newcommand{\A}{\mathcal{A}}
\newcommand{\noBifurcation}{\text{NoBifurcation}}
\newcommand{\noConn}{\text{NoConn}}
\newcommand{\goodRadii}{\text{GoodRadii}}

\section{The $k$-arm $r$-IIC} \label{sec:IICs}

In this section we construct the IIC as described in \cref{sec:iicintro}. Except for \cref{sec:TorusVsZd}, we work only with the infinite lattice $\Z^d$ in high dimensions and write $p_c=p_c(\Z^d)$. 

We begin with some notation and definitions. Recall that for any integer $r\geq 1$, the box $B_r$ is the set of vertices $\{-r,\ldots,r\}^d$. For $x \in \Z^d$ we write as usual $\C(x)$ for the cluster of $x$ in $\omega_p$ and  $\C(x\tbs r)$ for the subgraph of $\C(x)$ consisting on all edges of $\C(x)$ with at least one end-point in $\Z^d \setminus B_r$.
We slightly abuse notation and for $R>r$ we define the {\bf annulus} $B_R \setminus B_r$ to be the set of edges of $\Z^d$ which have at least one end point in $B_R \setminus B_r$. We always use the term \emph{annulus} when we mean this set of edges, and otherwise $B_R\setminus B_r$ is just the corresponding vertex set. Note that some edges of the annulus contain vertices of $\ell_\infty$ norm $R+L$ and $r-L+1$ where $L$ is the spread-out parameter. A {\bf crossing} of the annulus $B_R\setminus B_r$ is an open path using only edges of the annulus $B_R\setminus B_r$ starting from a vertex of $\Z^d \setminus B_{R-L}$ and ending in a vertex inside $B_r$. We note that if $x \in \Z^d\setminus B_{R-1}$ and $x\lr B_r$, then $\C(x\tbs r)$ contains a crossing of the annulus $B_R \setminus B_r$. For $k\geq 2$ we denote by $\noKArms(r,R)$ the event that there are no $k$ disjoint crossings of the annulus $B_R\setminus B_r$.

A {\bf marked configuration} is a subset of the edges $E(\Z^d)$ together with a marking in $0 \cup \mathbb{N}$ for each edge which respects the components, that is, the markings on edges which belong to the same component are equal. By the term \emph{component} used here, we mean a connected component of the edges of the configuration; sometimes the configuration will be  part of a larger configuration in which some components join to a bigger one. We note that removing edges from a marked configuration and keeping the markings on the remaining edges results in another marked configuration.

Let $R>r$ and $k\geq 1$ be fixed. By $\omega_{r,R}$ we denote a marked configuration whose edge set is a subset of the annulus  $B_R\setminus B_r$ and the markings are in $\{0,1,\ldots,k\}$.
For vertices $x_1,\ldots,x_k \in \Z^d\setminus B_R$ and a marked configuration $\omega_{r,R}$ we define the event (See \cref{fig:armsevent}.)
$$ \arms(\omega_{r,R}; x_1,\ldots,x_k) $$ 
as the intersection of the following events:
\begin{itemize}
\item The open edges of $\omega_p$ restricted to the annulus $B_R\setminus B_r$ are precisely the edge set of $\omega_{r,R}$,
\item $\forall i \in \{1,\ldots,k\} \quad x_i  \lr B_r$ in $\omega_p$,
\item $\forall i\neq j \in \{1,\ldots,k\} \quad \C(x_i \tbs r) \cap \C(x_j \tbs r) = \emptyset$,
\item $\forall i \in \{1,\ldots,k\}$  a component of $\omega_{r,R}$ is marked with $i$ if an only if it is connected to $x_i$ in an open path in $\omega_p$ that does not use a vertex of $B_r$. 
\end{itemize}

\noindent We are now ready to state the main theorem of this section.

\begin{figure}[!ht]
        \centering
        \includegraphics[width=.4 \textwidth]{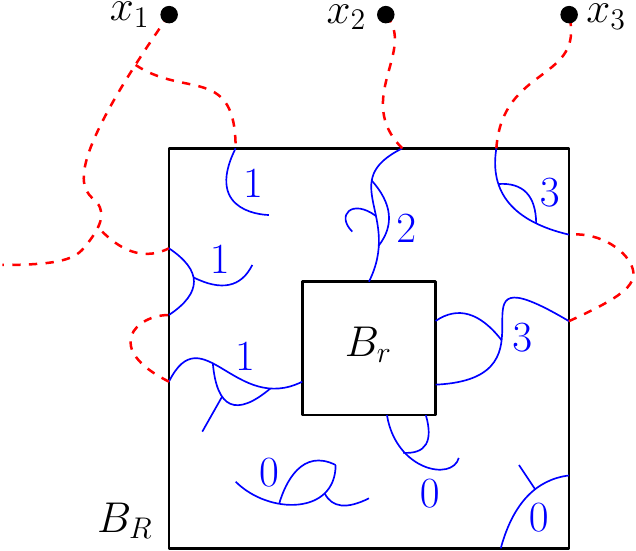}
        \caption{The $\arms(\omega_{r,R}; x_1,x_2,x_3)$ event. Dashed red lines represent open edges not in $\omega_{r,R}$ and normal blue lines are open edges in $\omega_{r,R}$. Numbers next to components of $\omega_{r,R}$ are the corresponding markings. In addition, there are no $4$ disjoint crossings.}
        \label{fig:armsevent}
\end{figure}

\begin{theorem}\label{thm:kArm-rIIC} For each integer $k\geq 1$ there exists $r_0(k)\geq 0$ so that for all $r\geq r_0(k)$ there exists a non-zero finite measure $\kriic$ supported on marked configuration with edges having at least one endpoint in $\Z^d\setminus B_r$ and contain precisely $k$ disjoint infinite components, each  marked distinctly with $\{1,\ldots,k\}$ and each component contains a vertex of $B_r$, and, such that for each $R>r$ and any marked configuration $\omega_{r,R}$ one has 
\be\label{eq:kArm-rIIC} \kriic(\omega_{r,R}) = \lim _{p \uparrow p_c} {1 \over \chi(p)^k} \sum_{x_1,\ldots, x_k \in\Z^d \setminus B_R} \proba_p( \arms(\omega_{r,R}; x_1,\ldots,x_k)) \, ,\ee
where as usual by $\kriic(\omega_{r,R})$ we mean the measure of marked configuration as above containing $\omega_{r,R}$ as a subset.
\end{theorem}

\begin{remark} Note that \cref{thm:kArm-rIIC} gives a nonzero finite measure. One can normalize by the total mass and obtain a probability measure. We do not provide any information here about the value of the constant. In the case that $k=1$ and $r=0$, it can be seen from the proof that the constant is $1$ and that the resulting measure is precisely \eqref{eq:iic}. 
\end{remark}
\begin{remark}\label{rmk:prevIIC} As mentioned earlier, in \cite{BiIIC}, the authors construct a bi-infinite IIC by conditioning on the origin to be connected to $x$ and $x'$ disjointly at $p=p_c$, and taking a limit as $|x|,|x'|, |x-x'|$ tend to infinity. It can be seen from the proof that in \cref{thm:kArm-rIIC} one has $r_0(2)=0$, so the case $k=2$ and $r=0$ constructed above is similar to the one constructed in \cite{BiIIC} with the important distinction that in \cite{BiIIC} the construction is done at $p=p_c$ while here it is at $p \uparrow p_c$. This distinction already appears in the case $k=1$ constructed in \cite{HofstadJaraiIIC} where the two limits objects are shown to be the same. 

We suspect that in this case the two constructions are also the same (with slightly adjustments due to the markings), but we do not address this in this paper. We believe that the induction on $k$ scheme presented here in the proof of \cref{thm:kArm-rIIC} can be performed at $p=p_c$ to construct the bi-infinite IIC of \cite{BiIIC}. This topic is further studied in an upcoming work of the first author and Hutchcroft \cite{BRH25}.
\end{remark}

The main difficulty of the proof is the existence of the limit. This will be proved by induction on $k$. The following lemma is the base case $k=1$.

\begin{lemma}\label{lem:1iic} For any $R>r\geq 0$ and any marked configuration $\omega_{r,R}$ the limit
$$ \lim _{p \uparrow p_c} {1 \over \chi(p)} \sum_{x\in \Z^d \setminus B_R} \proba(\arms(\omega_{r,R}; x)) \, ,$$
exists.    
\end{lemma}
\begin{proof} We begin by noting that the event $\arms(\omega_{r,R}; x)$ is independent of the status of the edges with both endpoints in $B_r$. Hence
\be\label{eq:openAllBrEdges}\proba(\arms(\omega_{r,R}; x)) = {1 \over p^{|E(B_r)|}} \proba(\arms(\omega_{r,R}; x), \AllOpen(r)) \, ,\ee
where $\AllOpen(r)$ is the event that all edges with both endpoints in $B_r$ are open, and $E(B_r)$ is the set of such edges. It is routine to check by definition that the intersection of $\arms(\omega_{r,R}; x)$ and $\AllOpen(r)$
equals the event that $0 \lr x$ and the open edges of $\omega_p$ restricted to the annulus $B_R\setminus B_r$ equals the edge set of $\omega_{r,R}$, the edges $E(B_r)$ are open and each component of $\omega_{r,R}$ marked with $1$ is connected to $x$ with an open path that does not enter $B_r$. We would now like to use \eqref{eq:iic} but there is an obstacle: we cannot determine whether components marked with $1$ in $\omega_{r,R}$ are connected to $x$ by examining a fixed number of edges. For this reason we proceed as follows. Let  $\eps>0$ be arbitrary and set $R_2=R_2(\eps,r) > R$ to be a large number that will be chosen later. It is  routine to check that when $x \not \in B_{R_2}$ the event 
$$ \{\arms(\omega_{r,R}; x), \AllOpen(r), \noTwoArms(r,R_2) \} $$
equals the event
\begin{itemize}
\item $0 \lr x$, and,
\item The open edges of $\omega_p$ restricted to the annulus $B_R\setminus B_r$ equals the edge set of $\omega_{r,R}$, and,
\item $\AllOpen(r), \noTwoArms(r,R_2)$, and,
\item Each component of $\omega_{r,R}$ is marked with $1$ if and only if it is connected to the unique component of $\omega_p$ restricted to the annulus $B_{R_2}\setminus B_r$ which contains a crossing of $B_{R_2}\setminus B_r$ in an open path that does not use a vertex of $B_r$.
\end{itemize}
The intersection of the last three events is measurable with respect to the $\omega_p$ status of edges with at least one endpoint in $B_{R_2}$. Hence, we may consider each of the possible configurations in the annulus $B_{R_2}\setminus B_R$ such that its union with the configuration $\omega_{r,R}$ and $E(B_r)$ satisfies the intersection of the last three events. We then intersect the event $0 \lr x$ with each such configuration and apply \eqref{eq:iic}. We sum over these possible configuration of $B_{R_2}\setminus B_R$ and obtain that there exists some constant $C=C( r,R,R_2, \omega_{r,R})\in[0,\infty)$ such that
$$ \lim_{p \uparrow p_c} {1 \over \chi(p)}\sum_{x\in \Z^d \setminus B_{R_2}} \proba_p (\arms(\omega_{r,R}; x), \AllOpen(r), \noTwoArms(r,R_2) ) = C \, .$$
We note that $\sup_{R_2} C( r,R,R_2, \omega_{r,R}) \leq 1$ by dropping all events on the left hand side except $0 \lr x$ and using the fact that the left hand side of the above is non-decreasing in $R_2$. We deduce that the constants $C( r,R,R_2, \omega_{r,R})$ converge to a limit $C(r,R,\omega_{r,R})\in[0,1]$ as $R_2 \to \infty$. We fix an arbitrary $\eps>0$ and choose $R_2$ large enough so that $|C( r,R,R_2, \omega_{r,R}) - C(r,R,\omega_{r,R})| \leq \eps$. We also have that 
$$ \sum_{x\in \Z^d \setminus B_{R_2}} \proba_p(x \lr B_r, \noTwoArms(r,R_2)^c) \leq \sum_{x\in \Z^d \setminus B_{R_2}} \proba_p(x \lr B_r \circ  B_r \lr  \Z^d \setminus B_{R_2-L}) \, ,$$
since if there are two disjoint crossings $\gamma_1,\gamma_2$ of the annulus $B_{R_2} \setminus B_{r}$ and a path $\gamma$ from $x$ to a vertex of $B_r$, then we can traverse $\gamma$ from $x$ until the first hitting point of $\gamma_1 \cup \gamma_2 \cup B_r$ resulting in a path from $x$ to $B_r$ and a disjoint crossing. We apply the BK inequality, and note that by opening all edges in $B_r$ we have that $\proba_p(x \lr  B_r)\leq p^{-|E(B_r)|}\proba_p(0 \lr x)$. Hence
$$ \sum_{x\in \Z^d \setminus B_{R_2}} \proba_p(x \lr B_r, \noTwoArms(r,R_2)^c) \leq C(r)\chi(p) \proba_p( B_r \lr  \Z^d \setminus B_{R_2-L}) \, .$$
Therefore, since $p\leq p_c$ we may choose $R_2=R_2(\eps,r)>0$ large enough so that the last term is at most $\eps \chi(p)$. We conclude that for every $r$ and $\eps>0$ there exists there exist $R_2$ large enough $\delta>0$ small enough such that for all $p \in [p_c-\delta,p_c]$ we have
$$\Big | {1 \over \chi(p)} \sum_{x\in \Z^d \setminus B_R} \proba_p(\arms(\omega_{r,R}; x), \AllOpen(r)) - C(r,R,\omega_{r,R}) \Big | \leq \eps \, ,$$
which together with \eqref{eq:openAllBrEdges} yields the desired result. 
\end{proof}

\begin{proof}[Proof of \cref{thm:kArm-rIIC}] We will show by induction on $k$ that for any $R>r\geq 1$ and any marked configuration $\omega_{r,R}$ we have that
\be\label{eq:iic_goal}  {1 \over \chi(p)^{k+1}} \sum_{x_1,\ldots,x_{k+1} \in \Z^d \setminus B_R} \proba_p(\arms( \omega_{r,R}; x_1,\ldots,x_{k+1})) \ee
has a limit as $p \uparrow p_c$. The case $k=1$ is precisely \cref{lem:1iic}. To perform the induction step we would like to condition on the event $$\arms( \omega_{r,R}^{k}; x_1,\ldots,x_k)\, ,$$  
where $\omega_{r,R}^{k}$ is obtained from $\omega_{r,R}$ by erasing the $(k+1)$-th marking. To perform this we need to condition on much more. 

We write $\omega_{r,R}^0$ for the configuration $\omega_{r,R}$ without all marks erased; i.e., $\omega_{r,R}^0$ is simply the edge subset in $\omega_{r,R}$ of the annulus $B_R\setminus B_r$. To estimate \eqref{eq:iic_goal} we condition on  $\omega^0_{r,R}$ and on $\C(x_i \tbs r)=A_i$ for $i=1,\ldots, k$ and disjoint finite vertex subsets $\{A_i\}_{i=1}^k$; by $\C(x_i \tbs r)=A_i$ we mean that we condition on all the open and closed edges with at least one end point in $A_i$; these edges precisely determine the event $\C(x_i \tbs r)=A_i$. We require that the sets $\{A_i\}_{i=1}^k$ are \textbf{admissible}, that is,
$$ \proba_p(\arms( \omega_{r,R}; x_1,\ldots,x_{k+1}), \C(x_i\tbs r)=A_i \,\,\forall i\in [k]) > 0 \, .$$
Note that in particular this implies that the $A_i$'s are disjoint and compatible with $\omega_{r,R}$ both with respect to the edges status of $\omega_{r,R}^0$ and the markings, i.e., a component marked $i$ in $\omega_{r,R}$ is contained in $A_i$. We write $\omega_{r,R}^0$ both for the edge subset of the annulus and the corresponding event that these are precisely the $p$-open edges of the annulus.
In this way, the sum in \eqref{eq:iic_goal} can be written as
\begin{eqnarray}\label{eq:bigSum} \sum_{x_1,\ldots,x_{k+1} \in \Z^d \setminus B_R} \sum_{\substack{A_1,\ldots, A_k\\ \text{admissible}}}&\,& \proba_p(\omega_{r,R}^0, \C(x_i\tbs r)=A_i \,\, \forall i\in [k]) \cdot \nonumber \\ &\cdot& \proba_p( \A \mid \omega_{r,R}^0, \C(x_i\tbs r)=A_i \,\, \forall i\in [k]) \, \,\end{eqnarray}
where $\A$ is the event that $x_{k+1}\lr B_r$ and $\C(x_{k+1} \tbs r)$ is disjoint from $A_1 \cup \ldots \cup A_k$ and that each component of $\omega_{r,R}$ that is marked with $k+1$ is contained in $\C(x_{k+1}\tbs r)$. Denote the union of these components in $\omega_{r,R}$ by $\C_{k+1}(\omega_{r,R})$. The second probability in \eqref{eq:bigSum} equals 
\be\label{eq:offFar} \proba_p(\{x_{k+1} \lr B_r, \C_{k+1}(\omega_{r,R}) \subset \C(x_{k+1} \tbs r)\} \off A_1 \cup \ldots \cup A_k \mid \omega_{r,R}^0) \, ,\ee
so that the sum in \eqref{eq:iic_goal} equals
\begin{equation}
\sum_{x_1,\ldots,x_{k+1}} \sum_{\substack{A_1,\ldots, A_k\\ \text{admissible}}} \, \proba_p(\omega_{r,R}^0, \C(x_i\tbs r)=A_i \,\, \forall i\in [k]) \cdot \eqref{eq:offFar} \, .\label{eq:iic_goal2}
\end{equation}

We now take a large $R_2 > R$ that will be chosen later and replace \eqref{eq:offFar} in the above by 
\be\label{eq:offClose} \proba_p(\{x_{k+1} \lr B_r, \C_{k+1}(\omega_{r,R}) \subset \C(x_{k+1} \tbs r)\} \off (A_1 \cup \ldots \cup A_k)\cap B_{R_2}  \mid \omega_{r,R}^0) \, ,\ee
and show that for any $\eps>0$ as long as $R_2$ is large enough we have
\be\label{eq:HandleError}
\sum_{x_1,\ldots,x_{k+1}} \sum_{\substack{A_1,\ldots, A_k\\ \text{admissible}}}&\,& \proba_p(\omega_{r,R}^0, \C(x_i\tbs r)=A_i \,\, \forall i\in [k]) \cdot [ \eqref{eq:offClose} - \eqref{eq:offFar}] \leq \eps \chi(p)^{k+1} \, .
\ee
Indeed \eqref{eq:offClose}$-$\eqref{eq:offFar} is bounded above by $1/\proba_p(\omega_{r,R}^0)$ times the probability that $x_{k+1}$ is connected to $B_r$ by a path that touches a vertex in $(A_1\cup \ldots \cup A_k)\cap (\Z^d \setminus B_{R_2})$. By the BK inequality this is bounded above by
$$ \sum_{v \in (A_1\cup \ldots \cup A_k)\cap \Z^d \setminus B_{R_2}} \proba_p(x_{k+1} \lr v) \proba_p(v \lr B_r) \, .$$
We put this back into the left hand side of \eqref{eq:HandleError}
 and change the order of summation between the sum over $v$ and over $A_1,\ldots,A_k$ to obtain that the left hand side of \eqref{eq:HandleError} is at most 
$$ \sum_{x_1,\ldots,x_{k+1}} \sum_{v \in \Z^d \setminus B_{R_2}}\proba_p(x_{k+1} \lr v) \proba_p(v \lr B_r)  \sum_{\substack{A_1,\ldots, A_k\\ \text{admissible,} \\v\in A_1 \cup \ldots \cup A_k}} \proba_p(\omega^0_{r,R}, \C(x_i\tbs r)=A_i \,\, \forall i\in [k])\, .$$
We crudely upper bound the last sum over $A_1,\ldots,A_k$ by the probability that for some $i\in\{1,\ldots, k\}$ we have $x_i \lr B_r$ and $x_i \lr v$ disjointly from $x_j \lr B_r$ for all other $j\neq i$. Using BK inequality we bound the above by
$$ \sum_{x_1,\ldots,x_{k+1}} \sum_{i=1}^k \sum_{v \in \Z^d \setminus B_{R_2}}\proba_p(x_{k+1} \lr v) \proba_p(v \lr  B_r) \proba_p(x_i \lr B_r, x_i \lr v)\prod_{j \neq i}^k \proba_p(x_j \lr B_r) \, .$$
We sum over $x_{k+1}$ the first term, and over $x_j$ for $j\neq i$ the last terms, giving a bound of $C(r)^k\chi(p)^k$. Next, as usual, if $x_i \lr  B_r$ and $x_i \lr v$, then there exists $z$ such that  $\{x_i \lr z\} \circ \{z \lr v\} \circ \{z \lr  B_r\}$ occurs. We use this with the BK inequality, sum the term $\proba_p(x_i \lr z)$ over $x_i$ to obtain another factor of $\chi(p)$. We deduce that the left hand side of \eqref{eq:HandleError} is at most 
$$ C(r)^k \chi(p)^{k+1} \sum_{z, v : v\in \Z^d \setminus B_{R_2}} \proba_p(z \lr v)\proba_p(z \lr B_r) \proba_p(v \lr B_r) \, .$$
As usual we bound $\proba_p(z \lr B_r)\leq C(r)\proba_p(z \lr 0)$ and recognize that, up to a factor of $C(r)^2$, the sum above is the open triangle diagram, which tends to $0$ as $R_2 \to \infty$, see \cite[Lemma 2.1]{BarskyAizenman91}. This verifies the inequality in \eqref{eq:HandleError}.  

Therefore, up to an additive error of $\eps \chi(p)^{k+1}$, we have that \eqref{eq:iic_goal} equals
$$ {1 \over \chi(p)^{k+1}} \sum_{x_1,\ldots,x_{k+1}} \sum_{\substack{A_1,\ldots,A_k \\ \text{admissible}}}  \proba_p(\omega_{r,R}^0, \C(x_i\tbs r)=A_i \,\, \forall i\in [k]) \cdot \eqref{eq:offClose}. $$
We now decompose $A_i=D_i \uplus H_i$ where $D_i=A_i \cap B_{R_2}$ and $H_i=A_i \setminus D_i$ and rewrite the last sum as
\be \label{eq:bigSumBetter} {1 \over \chi(p)^{k+1}} \sum_{x_1,\ldots,x_{k+1}} \sum_{D_1,\ldots,D_k} \sum_{\substack{H_1,\ldots,H_k \\ \{D_i\cup H_i\}_{i=1}^k \text{admissible}}} \! \! \! \! \! \! \! \!  \proba_p(\omega_{r,R}^0, \C(x_i\tbs r)=D_i \cup H_i \,\, \forall i\in [k]) \cdot \eqref{eq:offClose}. \ee
We will show that this converges as $p \uparrow p_c$. Since the term \eqref{eq:offClose} in the above may be pulled out of the sum over $x_1,\ldots, x_k, H_1,\ldots H_k$, the convergence of \eqref{eq:bigSumBetter} will  follow if we show that for fixed vertex subsets $D_1,\ldots, D_k$ of $B_{R_2}\setminus B_r$ the sums
\be\label{eq:offCloseMidStep} {1 \over \chi(p)^{k}} \sum_{x_1,\ldots,x_{k}} \sum_{\substack{H_1,\ldots,H_k \\ \{D_i\cup H_i\}_{i=1}^k \text{admissible}}}  \proba_p(\omega_{r,R}^0, \C(x_i\tbs r)=D_i \cup H_i \,\, \forall i\in [k]) \, ,\ee
and 
\be\label{eq:xk+1}  {1 \over \proba_p(\omega_{r,R}^0) \chi(p)} \sum_{x_{k+1}}  \proba_p(\omega_{r,R}^0, x_{k+1} \lr B_r, \C_{k+1}(\omega_{r,R}) \subset \C(x_{k+1} \tbs r) \off D_1 \cup \ldots \cup D_k) \, ,\ee
converge as $p \uparrow p_c$.

To deal with \eqref{eq:offCloseMidStep} we wish to apply our induction hypothesis. The event $\C(x_i\tbs r)=D_i \cup H_i$ is precisely the event that every edge that has one endpoint in $D_i \cup H_i$ and the other not in it is closed unless both these end points are in $B_r$ (in which case they can be either open or closed) and that the $p$-open edges with both endpoints in $D_i\cup H_i$ span a connected graph on the vertex set $D_i\cup H_i$. Given such a configuration, we may consider the open edges that have at least one endpoints in $D_i$. The graph spanned on this set of edges is not necessarily connected and may have several connected components which we call \emph{the components of $D_i$}. These components join to a single component when adding the open edges with both endpoints in $D_i\cup H_i$. With this slight abuse of notation, we claim that \eqref{eq:offCloseMidStep} equals 
$$  {1 \over \chi(p)^{k}}  
\sum_{x_1,\ldots,x_{k}} \sum_{\omega_{r,R_2}} \proba_p(\arms( \omega_{r,R_2}; x_1,\ldots,x_{k})) \, ,$$
where the first summation is over marked configurations $\omega_{r,R_2}$ in $B_{R_2}\setminus B_r$ which satisfy:
\begin{enumerate}
\item $\omega_{r,R_2}$ agrees with the configuration $\omega_{r,R}^0$ in the annulus $B_R \setminus B_r$, and,
\item The components of $D_i$ are precisely the components of $\omega_{r,R_2}$ that are marked with $i$, for each $i\in[k]$, that is,  both the open and closed edges of the components of $D_i$ and of $\omega_{r,R_2}$ that are marked with $i$ agree.
\end{enumerate}

The equality between the two last sums follows directly by the definition of the various events involved. Thus, our induction hypothesis (used with $R_2,r$ and $\omega_{r,R_2}$) followed by summing over $\omega_{r,R_2}$ gives that \eqref{eq:offCloseMidStep} converges as $p\uparrow p_c$ to a constant depending on $R_2, D_1,\ldots, D_k$ and $\omega_{r,R}$.

Similarly, we have that \eqref{eq:xk+1} equals 
$$ {1 \over \proba_p(\omega_{r,R}^0) \chi(p)} \sum_{x_{k+1}} \sum_{\omega_{r,R_2}}\proba_p(\arms(\omega_{r,R_2}; x_{k+1}) \, ,$$
where the second sum is over marked configurations $\omega_{r,R_2}$ (with markings in $\{0,1\}$) satisfying
\begin{enumerate}
    \item $\omega_{r,R_2}$ agrees with $\omega_{r,R}^0$ on the edges of the annulus $B_R\setminus B_r$,
    \item There are no edges in $\omega_{r,R_2}$ that have an endpoint in $D_1 \cup \ldots \cup D_k$,
    \item The components of $\omega_{r,R}$ that are marked with $k+1$ are subsets of components of $\omega_{r,R_2}$ that are marked with $1$.
\end{enumerate} 
Thus, appealing to \cref{lem:1iic} on each $\omega_{r,R_2}$ separately, then summing over them, shows that \eqref{eq:xk+1} converges to a constant depending on $R_2, D_1,\ldots, D_k$ and $\omega_{r,R}$. Since there are only finitely many choices of $D_1,\ldots,D_k$ we may sum over them and conclude that \eqref{eq:bigSumBetter} converges as $p\uparrow p_c$. Denote the limit by $C=C(\omega_{r,R}, R_2)$. We deduce by \eqref{eq:HandleError} that 
$$ \Big | {1 \over \chi(p)^{k+1}} \sum_{x_1,\ldots,x_{k+1}} \proba_p(\arms( \omega_{r,R}; x_1,\ldots,x_{k+1})) - C(\omega_{r,R},R_2) \Big | \leq \eps \, ,$$
where $R_2=R_2(\eps)>0$ is large enough and $p<p_c$ is close enough to $p_c$. Again $\sup_{R_2} C(\omega_{r,R},R_2)<\infty$ since \eqref{eq:bigSumBetter} is bounded above using the BK inequality by $\chi(p)^{-k-1}\sum_{x_1,\ldots,x_{k+1}} \prod_{i=1}^{k+1}\proba_p(x_i \lr B_r)$. Hence we may choose a limit point $C(\omega_{r,R})$ of them as $R_2\to \infty$. Since the left term on the left hand side does not depend on $R_2$ we learn that all limit points as $p \uparrow p_c$ of 
$${1 \over \chi(p)^{k+1}} \sum_{x_1,\ldots,x_{k+1}} \proba_p(\arms( \omega_{r,R}; x_1,\ldots,x_{k+1}))$$
are $C(\omega_{r,R})$. This shows that \eqref{eq:iic_goal} converges as $p \uparrow p_c$ hence the limit in \eqref{eq:kArm-rIIC} exists and defines consistent measures on finite marked configurations. By Kolmogorov's extension theorem we may extend these measures to the measure $\kriic(\cdot)$ supported on infinite configurations.

The measure $\kriic$ can be a zero measure, for instance when $r=0$ and $k$ is larger than the degree, so no $k$ disjoint paths from $0$ may exist. However, it is routine to check using the BK inequality and the triangle condition that for any $k$, if $r$ is large enough, 
$$ {1 \over \chi(p)^k} \sum_{x_1,\ldots, x_k} \proba_p(\forall i\in [k] \,\, x_i \lr B_r, \C(x_i\tbs r)\cap \C(x_j \tbs r) \,\, \forall i\neq j) = \emptyset ) \, ,$$ 
is uniformly bounded away from $0$ as $p \uparrow p_c$. Furthermore, if $x_i \lr B_r$ for each $i$ and $\C(x_i\tbs r)\cap \C(x_j \tbs r)=\emptyset$ and there exists at least $k+1$ disjoint crossings of the annulus $B_R \setminus B_r$ we obtain that 
$$ \{x_1 \lr B_r\} \circ \cdots \circ \{x_k \lr B_r\} \circ \{B_r \lr \Z^d\setminus B_{R-1}\} \, .$$
Indeed, take arbitrary disjoint crossings $\gamma_1,\ldots, \gamma_{m}$ with $m \geq k+1$ and $m$ is the maximum number of disjoint crossings. Now for each $i\in[k]$ take the path from $x_i$ to $B_r$ and find the first time it hits $\gamma_1 \cup \ldots \cup \gamma_m$. Use that first visited crossing to reach $B_r$ and since $m\geq k+1$ at least one of $\gamma_1,\ldots,\gamma_m$ was not used. Using the BK inequality we get that we get that 
$$ \lim_{p \uparrow p_c} \lim_{R \to \infty} {1 \over \chi(p)^k} \sum_{x_1,\ldots, x_k} \proba_p\Big (\forall i\in [k] \,\, x_i \lr B_r, \forall i\neq j \,\, \C(x_i\tbs r)\cap \C(x_j \tbs r) = \emptyset, \,\, k+1 \textrm{ crossings} \Big ) = 0 \, ,$$ 
implying the support of $\kriic$ are configurations with precisely $k$ infinite components. This concludes the proof.
\end{proof}

\subsection{Proof of \cref{thm:sharpChiChi2Zd}}

\subsubsection{Proof of the first assertion of \cref{thm:sharpChiChi2Zd}} \label{sec:firstAssertion}
We will prove that there exists a constant $\Cnst=\Cnst(d,L)\in(0,\infty)$ such that
\be\label{eq:SharpChiChi1st} \lim _{p \uparrow p_c} {\chi'(p) \over \chi(p)^2} = \Cnst \, .\ee
Given \eqref{eq:SharpChiChi1st} the first assertion of \cref{thm:sharpChiChi2Zd} with $\Cnst_1 = \Cnst^{-1}$ follows by integration. Indeed, for any $p_1<p_2<p_c$ we have
$$ {1 \over \chi(p_1)} - {1 \over \chi(p_2)} = -\int_{p_1}^{p_2} { \chi'(p) \over \chi(p)^2}dp  \, .$$
Then take $p_2 \uparrow p_c$ and obtain (taking $p=p_1$) that $\chi(p)^{-1} = -\int_{p}^{p_c}  \chi'(p)/\chi(p)^2 dp$ which is $(\Cnst+o(1))(p_c - p)$ as $p\uparrow p_c$ by \eqref{eq:SharpChiChi1st}.

We now turn to prove \eqref{eq:SharpChiChi1st}. By Russo's formula (see \cite[Chapter 2.4]{Grimm99}) for any $p<p_c$ we have
\begin{eqnarray*} (1-p)\chi'(p) &=& \sum_{x\in \Z^d} \sum_{v \sim v'} \proba_p(0 \lr v, v' \lr x, v \nlr v') \\ &=&  \sum_{b \sim 0} \sum_{x,y\in \Z^d} \proba_p(0 \lr x, b \lr y, 0 \nlr b) \, ,\end{eqnarray*}
where the last inequality is due to symmetry. We now fix large $R>r$ and approximate the sum above by 
$$ S_p(r,R) := \sum_{b \sim 0, x,y \in \Z^d} \proba_p(0 \lr x, b \lr y, 0 \nlr b, \noThreeArms(r,R)) \, .$$
We first show that this sum indeed approximates  $(1-p)\chi'(p)$ when $R$ is large. We claim that for $x,y\in \Z^d\setminus B_R$, if the event $\{0 \lr x, b \lr y, 0 \nlr b\}$ holds and there \emph{are} three disjoint crossings of the annulus $B_R \setminus B_r$, then the event 
\be\label{eq:mustOccur} \{x \lr  B_r\}\circ\{y \lr  B_r\}\circ\{ B_r \lr \Z^d \setminus B_{R-1}\} \, ,\ee
occurs. Indeed, let $k\geq 3$ be the maximal number of disjoint crossings of the annulus $B_R\setminus B_r$ and denote an arbitrary choice of them by $\gamma_1,\ldots, \gamma_k$. The path from $x$ to $0$ must intersect $\gamma_1 \cup \ldots \cup \gamma_k$ otherwise we would have another disjoint crossings, contradicting the maximality of $k$. Denote by $\gamma_i$, $i\in [k]$, the crossing which the path from $x$ to $0$ visit at the first time it intersects $\gamma_1 \cup \ldots \cup \gamma_k$. Similarly, the path from $y$ to $b$ touches first $\gamma_j$ with $j\neq i$ since $0 \nlr b$. We form a new crossing $\gamma_i'$ by considering the first time the path from $x$ to $0$ visits $B_R$ until the first time it visits $\gamma_i$ and concatenating the part of $\gamma_i$ from there that reaches $B_r$. Similarly we form $\gamma_j'$.  
In this construction the paths $\gamma'_i$,$\gamma'_j$, $\{\gamma_m\}_{m\neq i,j}$ are disjoints so when $k\geq 3$ the event \eqref{eq:mustOccur} must hold. Using the BK inequality we deduce that
$$ (1-p)\chi'(p)-S_p(r,R) \leq \sum_{b\sim 0,x,y \in \Z^d} \proba_p(x \lr  B_r)\proba_p(y \lr  B_r) \proba_p( B_r \lr \Z^d \setminus B_{R-1}) \, .$$
Since $\proba_p(B_r \lr x) \leq C(r)\proba_p(0\lr x)$ for some constant $C(r)$ that depends only on $r,d$ and $L$, and since $ \proba_p( B_r \lr \Z^d \setminus B_{R-1}) \to 0$ as $R \to \infty$ uniformly for all $p \leq p_c$, we obtain that for any $\eps>0$ there exists $R>r$ such that for all $p \leq p_c$ we have
\be\label{eq:chiPapproximate} (1-p)\chi'(p)-S_p(r,R) \leq \eps \chi(p)^2 \, .\ee
Next, to estimate $S_p(r,R)$ we claim that the event 
$$ \{0 \lr x, b \lr y, 0 \nlr b\} \cap \noThreeArms(r,R) $$ 
equals the event 
\be\label{eq:EqualEvents} \biguplus _{\omega_{r,R}} \arms(\omega_{r,R}; x,y) \cap \A(\omega_{r,R})  \, ,\ee
where the union is over marked configurations $\omega_{r,R}$ which have two disjoint crossings of the annulus but no more than two, with precisely two distinct components which contain these crossings, one marked with $1$ and the other with $2$ (there may be other components, but none contain a crossing) and $\A(\omega_{r,R})$ is the event that, when replacing the open edges of the annulus $B_R\backslash B_r$ by those of $\omega_{r,R}$, we have:
\begin{itemize}
    \item $0$ is connected to the aforementioned component of $\omega_{r,R}$ marked with $1$ with an open path that does not leave $B_R$, and,
    \item $b$ is connected to the aforementioned component of $\omega_{r,R}$ marked with $2$ with an open path that does not leave $B_R$, and,
    \item there is no open path between $0$ and $b$ inside $B_{R}$.
\end{itemize}

To see the equality between the two events we note that the first is contained in the second by definition; indeed, if the first event holds, then we may consider the configuration in the annulus and observe that it must have precisely two components which contain a crossing, one is a part of $\C(x\tbs r)$ and the second is a part of $\C(y\tbs r)$. We mark the first with $1$ and the second with $2$ and denote the resulting configuration by $\omega_{r,R}$; it is routine to check that $\arms(\omega_{r,R};x,y)$ and $\A(\omega_{r,R})$ hold.

Conversely, the second event immediately implies that $0$ is connected to $x$ and that $b$ is connected to $y$, so it remains to verify that it also implies that $0 \nlr b$. Assume to the contrary that $0 \lr b$ and let $\gamma$ be an open path from $0$ to $b$. Also denote by $\omega_1$ and $\omega_2$ the connected components of $\omega_{r,R}$ which contain a crossing of the annulus and are marked with $1$ and $2$, respectively. Since $0$ is not connected to $b$ within $B_R$, due to $\A(\omega_{r,R})$, the path $\gamma$ must contain a vertex $\Z^d\setminus B_R$. Consider the following two disjoint sub-paths of $\gamma$.  The first, $\gamma_1$, is from the last visit of $\gamma$ in $B_r$ before its first visit to $\Z^d\setminus B_R$ until the first visit to $\Z^d\setminus B_R$. The second, $\gamma_2$, is from the last visit of $\gamma$ to $\Z^d\setminus B_R$ until reaching $b$. Note that both $\gamma_1$ and $\gamma_2$ contain a crossing of the annulus $B_R \setminus B_r$ and also, the path $\gamma$ from the last vertex of $\gamma_1$ to the first vertex of $\gamma_2$ cannot visit $B_{r}$ since that would yield three disjoint crossings of the annulus. 

Now, if $\gamma_1$ intersects the vertices of $\omega_1$, then $\gamma_2$ cannot intersect the vertices of $\omega_1 \cup \omega_2$. Indeed, if it intersects $\omega_1$, then we deduce that there is an open path inside $B_{R}$ connecting $0$ to $b$, violating $\A(\omega_{r,R})$, and, if it intersects $\omega_2$, then it violates $\C(x\tbs r) \cap \C(y\tbs r)=\emptyset$. Similarly, if $\gamma_1$ intersects the vertices of $\omega_2$, then $\gamma_2$ cannot intersect $\omega_1 \cup \omega_2$. We deduce that one of $\gamma_1$ or $\gamma_2$ is disjoint from $\omega_1 \cup \omega_2$, violating $\noThreeArms(r,R)$.

Given a fixed $\omega_{r,R}$, to verify the event $\A(\omega_{r,R})$ we need to look additionally only at the edges with both endpoints in $B_r$. In particular, $\A(\omega_{r,R})$ is independent of $\arms(\omega_{r,R};x,y)$. Thus, applying \cref{thm:kArm-rIIC} gives that
$$ \lim_{p \uparrow p_c} {S_p(r,R) \over \chi(p)^2} =  \sum_{\substack{\omega_{r,R}}} \qiic^{2,r}(\omega_{r,R}) \proba_{p_c}(\A(\omega_{r,R})) \, ,
$$
where the sum is over the same marked configurations as in \eqref{eq:EqualEvents}. We do not care about the precise value of the limit on the right hand side and simply denote it by $C(r,R)$. By \cref{thm:abDeltaAlternative} we have that $\chi'(p)/\chi(p)^2 \asymp 1$ and so due to \eqref{eq:chiPapproximate} we deduce that 
\be \label{eq:CrBounded} 0 < \inf_{r,R>r} C(r,R) \leq \sup_{r,R>r} C(r,R)<\infty  \, .\ee Furthermore, again by \eqref{eq:chiPapproximate} we learn that for any $r>0$ there exists $R=R(r)>r$ such that 
$$ \Big | {(1-p)\chi'(p) \over \chi(p)^2} - {S_p(r,R) \over \chi(p)^2} \Big | \leq 1/r \, .$$
We denote $C(r)=C(r,R(r))$. By \eqref{eq:CrBounded} there exists a sequence $r_n$ such that $C(r_n)\to C(d,L) \in (0,\infty)$. Denote $R_n=R(r_n)$, so that
$$ \Big | {(1-p)\chi'(p) \over \chi(p)^2} - {S_p(r_n,R_n) \over \chi(p)^2} \Big | \leq 1/r_n \, .$$
For any fixed $n$ and take the $\limsup$ as $p\uparrow p_c$ to get that 
$$ \Big |\limsup_{p \uparrow p_c} {(1-p)\chi'(p) \over \chi(p)^2} - C(r_n)\Big |\leq {1 \over r_n} \, ,$$
and the same estimate holds for the $\liminf$ as well. We then take $n\to \infty$ and \eqref{eq:SharpChiChi1st} follows. \qed

\subsubsection{Proof of the second assertion of \cref{thm:sharpChiChi2Zd}}
If $0\lr x$ and $0 \lr y$, then there must exists $z$ such that $\{0\lr z\}\circ\{z \lr x\}\circ\{z \lr y\}$. Let $N(0,x,y)$ denote the vertices $z\in \Z^d$ satisfying this so that 
\begin{eqnarray*} \E|\C(0)|^2 &=& \sum_{x,y\in \Z^d} \proba_p(0 \lr x, 0 \lr y) \\ &=& \sum_{x,y} \sum_{k\geq 1} {1 \over k} \sum_{z} \proba_p(\{0\lr z\}\circ\{z \lr x\}\circ\{z \lr y\}, |N(0,x,y)|=k) \, ,
\end{eqnarray*}
so by symmetry
$$ \E |\C(0)|^2 = \sum_{x_1,x_2,x_3} \sum_{k\geq 1} {1 \over k} \proba_p(0\in N(x_1,x_2,x_3), |N(x_1,x_2,x_3)|=k) \, .$$
where $N(x_1,x_2,x_3)$ is the set of vertices $v$ for which $\{v \lr x_1\} \circ \{v \lr x_2\}\circ\{v \lr x_3\}$ occurs. 

Similarly to the proof in \cref{sec:firstAssertion}, we would like to determine locally that $0\in N(x_1,x_2,x_3)$ and $|N(x_1,x_2,x_3)|=k$ occurs so we can apply \cref{thm:kArm-rIIC}. In light of what we have done in \cref{sec:firstAssertion}, a natural strategy is to take radii $R\gg r \gg 1$ and to bound vertices $v\in B_r$ so that $v$ has three disjoint open paths connecting it to components of $\omega_{r,R}$ containing crossings marked with $1,2$ and $3$. However, the existence of such disjoint paths does not guarantee that $v$ has three disjoint paths to $x,y$ and $z$. Indeed, $v$ could be connected to the component marked with $1$ via  path that first traverses the component marked with $2$ in such a way that the disjoint path from $v$ that reaches the component marked with $2$ cannot be continued to a path reaching $y$ without touching the first path. To overcome this difficulty we proceed as follows.

Let $r_3 \gg r_2 \gg r_1 \gg r_0 \gg 1$ be radii that will be chosen later depending on some arbitrary $\eps>0$. We define the event \goodRadii$(r_0,r_1,r_2,r_3)$ as the intersection of the events: 
\begin{enumerate}
\item $\noFourArms(r_0,r_1), \noFourArms(r_1,r_2), \noFourArms(r_2,r_3)$
\item $\noConn(r_0; x_1,x_2,x_3)$ 
\end{enumerate}
where $\noConn(r_0;x_1,x_2,x_3)$ is the event that $\C(x_1\tbs r_0), \C(x_2 \tbs r_0)$ and $\C(x_3\tbs r_0)$ are pairwise disjoint. We consider the sum 
$$ S_p(r_0,r_1,r_2,r_3) = \sum_{x_1,x_2,x_3 \in \Z^d\setminus B_{r_3}} \sum_{k\geq 1} {1 \over k} \proba_p(0\in N(x_1,x_2,x_3), |N(x_1,x_2,x_3)|=k, \goodRadii(r_0,r_1,r_2,r_3))\, .$$

We show that this approximates $\E|\C(0)|^2$. By the same deduction performed below \eqref{eq:mustOccur} that if the event
$$ 0\in N(x_1,x_2,x_3), \noFourArms(r_i,r_{i+1})^c, \noConn(r_i;x_1,x_2,x_3) \, ,$$
occurs, then the event 
$$ \{x_1 \lr B_{r_i}\}\circ \{x_2 \lr B_{r_i}\} \circ \{x_3 \lr B_{r_i}\} \circ \{\Z^d\setminus B_{r_{i+1}-L} \lr B_{r_i}\}$$
must also occur. Therefore as we argued to obtain \eqref{eq:chiPapproximate} we get that for any $\eps>0$ we can choose $r_3 \gg r_2 \gg r_1\gg r_0\gg 1$ such that for all $i=0,1,2$ we have
$$ \sum_{x_1,x_2,x_3} \proba_p(0\in N(x_1,x_2,x_3), \noFourArms(r_i,r_{i+1})^c, \noConn(r_i;x_1,x_2,x_3)) \leq \eps \chi(p)^3 \, .$$
Furthermore, there exists a function $\eps(r_0)>0$ with $\eps(r_0)\to 0$ as $r_0\to \infty$ such that
$$ \sum_{x_1, x_2, x_3} \proba_p(0\in N(x_1,x_2,x_3),  \noConn(r_0; x_1, x_2, x_3)^c) \leq \eps(r_0) \chi(p)^3 \, ,$$
which is shown as follows. Since $0\in N(x_1,x_2,x_3)$ let $\gamma_1,\gamma_2,\gamma_3$ be disjoint open paths connecting $0$ to $x_1,x_2,x_3$ respectively. Furthermore, since $\noConn(r_0;x_1,x_2,x_3)$ does not hold, assume without loss of generality that $\C(x_1\tbs r) = \C(x_2\tbs r)$ and so there are vertices $z\in \gamma_1$ and $w\in \gamma_2$ with $z,w\in \Z^d \setminus B_r$ so that there is an open path between $z$ and $w$ that is disjoint from $\gamma_1$ and $\gamma_2$. We apply the BK inequality, sum over $x_1,x_2,x_3$ to get a factor of $\chi(p)^3$ then over $z,w$ which is the open triangle diagram, giving a contribution of at most $\eps(r_0)$ with $\eps(r_0)\to 0$ as $r_0\to\infty$.

We deduce by the last two assertions that for any $\eps>0$ we can choose $r_3 \gg r_2 \gg r_1 \gg r_0 \gg 1$ such that
\be\label{eq:Chi2Approx} \Big |  {\E|\C(0)|^2 - S_p(r_0,r_1,r_2,r_3)\over \chi(p)^3} \Big | \leq \eps \, .\ee
Next we prove that $S_p(r_0,r_1,r_2,r_3)/\chi(p)^3$ converges as $p\uparrow p_c$. 

\begin{claim}\label{clm:local3Connection} Let $x_1,x_2,x_3\in \Z^d \setminus B_{r_3}$ and assume that  $0\in N(x_1,x_2,x_3)$ and $\goodRadii(r_0,r_1,r_2,r_3)$ hold. Then 
\begin{enumerate} 
\item If $v\not \in B_{r_1}$, then $v \not \in N(x_1,x_2,x_3)$.
\item If $v \in B_{r_1}$, then $v \in N(x_1,x_2,x_3)$ if and only if there exists three disjoint open paths from $v$ to $\Z^d \setminus B_{r_3}$.
\end{enumerate}
\end{claim}
\begin{proof} For the first assertion, assume by contradiction that $v\not \in B_{r_1}$ and $v \in N(x_1,x_2,x_3)$. So there exists disjoint open paths $\gamma_1, \gamma_2$ and $\gamma_3$ from $v$ to $x_1, x_2$ and $x_3$, respectively. Since $\noConn(r_0; x_1,x_2,x_3)$ holds, it follows that at most one of these paths can remain in $\Z^d \setminus B_{r_0}$. So two of these paths must enter $B_{r_0}$ \emph{before} visiting $\cup_{i=1}^3 \C(x_i \tbs r_0)$. This gives two disjoint crossings of the annulus $B_{r_1}\setminus B_{r_0}$. Together with the additional three disjoint crossings in $\cup_{i=1}^3 \C(x_i \tbs r_0)$ we get five disjoint crossings of the annulus $B_{r_1} \setminus B_{r_0}$, violating $\noFourArms(r_0,r_1)$.

For the second implication, one direction is trivial since $x_1,x_2,x_3 \in \Z^d \setminus B_{r_3}$. For the other direction, let $\{\gamma_i\}_{i=1}^3$ be disjoint open paths from $v$ to $\Z^d \setminus B_{r_3}$. Denote their terminal vertices by $\{w_i\}_{i=1}^3$ and assume without loss of generality that $w_i$ is the only vertex of $\gamma_i$ in $\Z^d \setminus B_{r_3}$ (otherwise erase the part of $\gamma_i$ after $w_i$). Let $\gamma_i^{(1)} \subset \gamma_i$ be the sub-path of $\gamma_i$ from $v$ until its last visit to $B_{r_1}$, and $\gamma_i^{(2)} \subset \gamma_i$ be the sub-path of $\gamma_i$ from the last visit to $B_{r_1}$ until the terminal vertex $w_i$. So $\gamma_i^{(1)}$ and $\gamma_i^{(2)}$ have disjoint edges and appending them yields $\gamma_i$.

For each $i\in\{1,2,3\}$, the path $\gamma_i^{(1)}$ may not contain a vertex of $\Z^d \setminus B_{r_2}$. Otherwise, it contains two disjoint crossings of the annulus $B_{r_2}\setminus B_{r_1}$, one crossing is contained in the path from $v$ to $\Z^d \setminus B_{r_2}$, and the other on the way back to $B_{r_1}$. So together with the two additional crossings in $\gamma_j$ with $j\in\{1,2,3\}\setminus\{i\}$, we get that $\noFourArms(r_1,r_2)$ is violated. 

The paths $\{\gamma_i^{(2)}\}_{i=1}^3$ form $3$ disjoint crossings of the annulus $B_{r_3}\setminus B_{r_1}$. Since $0\in N(x_1,x_2,x_3)$ and $\noConn(r_1; x_1,x_2,x_3)$ occurs, the components $\C(x_1\tbs r_1), \C(x_2\tbs r_1)$ and $\C(x_3\tbs r_1)$ also contain $3$ disjoint crossings of the same annulus. We deduce that each $\gamma_i^{(2)}$ must intersect the vertices of either $\C(x_1\tbs r_1), \C(x_2\tbs r_1)$ or $\C(x_3\tbs r_1)$, otherwise $\noFourArms(r_1,r_3)$ is violated. Furthermore, each $\gamma_i^{(2)}$  cannot intersect two of these clusters, otherwise $\noConn(r_1;x_1,x_2,x_3)$ is violated. Thus we may assume without loss of generality that $\gamma_i^{(2)}$ intersects the vertices of $\C(x_i\tbs r_1)$ but does not intersect $\C(x_{j} \tbs r_1)$ for $j\in\{1,2,3\}\setminus \{i\}$.

Next, since $w_i$ belongs to $\C(x_i\tbs r_1)$, there exists an open path from $w_i$ to $x_i$ in $\C(x_i\tbs r_1)$. This path is not necessarily disjoint from $\gamma_i^{(2)}$, but there is a vertex $u_i$ which is the last vertex of this path that belongs to $\gamma_i^{(2)}$. Denote the corresponding path in $\C(x_i\tbs r_1)$ from $u_i$ to $x_i$ by $\eta_i$; the edge set of this path is disjoint from the edges $\gamma_i^{(2)}$ by definition. Also, $\eta_i$ cannot enter $B_{r_2}$, since otherwise $\eta_i$ and $\gamma_i^{(2)}$ contain two disjoint crossings of the annulus $B_{r_3}\setminus B_{r_2}$, which together with the two other crossings in $\C(x_j\tbs r_1)$ for $j\in\{1,2,3\}\setminus \{i\}$ violate $\noFourArms(r_2,r_3)$. In particular we deduce that $\eta_i$ is disjoint from $\gamma_i^{(1)}$ since the latter does not contain a vertex of $\Z^d \setminus B_{r_2}$.

Lastly, we consider the paths $\beta_i = \gamma_i' \uplus \eta_i$ where $\gamma_i'\subset \gamma_i$ is obtained by traversing $\gamma_i$ from $v$ until the first visit to $u_i$; so $\gamma_i'$ and $\eta_i$ have disjoint edges. To conclude the proof we claim that $\{\beta_i\}_{i=1}^3$ have pairwise disjoint edges. Indeed, $\{\gamma_i'\}_{i=1}^3$ are pairwise disjoint be definition, and $\{\eta_i\}_{i=1}^3$ are pairwise disjoint since the vertex set of $\eta_i$ is contained in $\C(x_i\tbs r_1)$ and $\noConn(r_1; x_1,x_2,x_3)$ occurs. Lastly, $\gamma_i'$ is comprised from $\gamma_i^{(1)}$ and the sub-path of $\gamma_i^{(2)}$ from the last visit to $B_{r-1}$ until the first visit to $u_i$. The first, $\gamma_i^{(1)}$, does not contain a vertex from $\Z^d \setminus B_{r_2}$ so it is disjoint from $\eta_j$ for all $j\in \{1,2,3\}$; the vertices of the second are contained in $\C(x_i; r_1)$ so they do not intersect $\eta_j$ for $j \in \{1,2,3\}\setminus\{i\}$.  
\end{proof}

Write $N_{r_1,r_3}(x_1,x_2,x_3)$ for the number of $v\in B_{r_1}$ for which there exists three disjoint paths emanating from $v$ to $\Z^d \setminus B_{r_3}$. 
By \cref{clm:local3Connection}, for any $r_3>r_2>r_1>r_0>1$ we may rewrite $S_p(r_0,r_1,r_2,r_3)$ as the sum
\begin{eqnarray}\label{eq:rewriteSp} \sum_{\omega_{r_0,r_3}} \sum_{x_1,x_2,x_3 \in \Z^d\setminus B_{r_3}} \sum_{k\geq 1} {1 \over k}  \proba_p \Big (\arms(\omega_{r_0,r_3}; x_1,x_2,x_3), |N_{r_1,r_3}(x_1,x_2,x_3)|=k, \\ 0 \in N_{r_1,r_3}(x_1,x_2,x_3), \goodRadii(r_0,r_1,r_2,r_3) \Big ) \, .
\end{eqnarray}
over marked configuration $\omega_{r_0,r_3}$ of the annulus $B_{r_3}\setminus B_{r_0}$.
Conditioned on the event that the $\omega_p$ configuration in the annulus $B_{r_3}\setminus B_{r_0}$ is $\omega_{r_0,r_4}$, the event 
$$ |N_{r_1,r_3}(x_1,x_2,x_3)|=k, \\ 0 \in N_{r_1,r_3}(x_1,x_2,x_3) $$
is determined just by considering additionally the edges with both endpoints in $B_{r_0}$. Hence the probability in \eqref{eq:rewriteSp} equals $\proba_p(\arms(\omega_{r_0,r_3}; x_1,x_2,x_3)) C(\omega_{r_0,r_4})$ where 
the last term is a constant depending only on $\omega_{r_0,r_4}$. Therefore, 
we may apply \cref{thm:kArm-rIIC} for each $\omega_{r_0,r_4}$ separately, then sum over all possible $\omega_{r_0,r_4}$.
We obtain that there exists a constant $C(r_0,r_1,r_2,r_3)$ such that
$$ \lim_{p \uparrow p_c} {S_p(r_0,r_1,r_2,r_3) \over \chi(p)^3} =  C(r_0,r_1,r_2,r_3) \, .$$
As before we have that $C(r_0,r_1,r_2,r_3) \leq 1$ for all $r_3>r_2>r_1>r_0>1$. We plug this into 
\eqref{eq:Chi2Approx} and by choosing  $r_3\gg r_2\gg r_1\gg r_0\gg 1$ we obtain the desired result.
\qed

\subsection{Susceptibility and mean cluster size squared on the torus} \label{sec:TorusVsZd}

Here we prove \cref{lem:chi''} then \cref{lem:sharpChiChi2Torus}. The proofs are based on a coupling, due to Heydenreich and van der Hofstad \cite{HeydenHofstad1}, between percolation on the infinite lattice $\Z^d$ and on the finite torus $\Z_n^d$. Unlike the rest of \cref{sec:IICs}, here we must use the notation $\chi_\Z$ and $\chi_\T$ for the susceptibility on $\Z^d$ and on the torus, respectively. We begin by an estimate on their second derivative.

\begin{proof}[Proof of \cref{lem:chi''}] We will just prove the bound for $\X_\T''(p)$. The bound for $\X_\Z''(p)$ follows by the same calculation using the fact that $\X_\Z$ is analytic at the subcritical phase \cite{Kesten81}. By Russo's formula (see (2.33) in \cite{Grimm99}) $$\X_\T''(p) = \sum_{v \in \Z_n^d} \E_p N^{\rm ser}(v) - \sum_{v \in \Z_n^d} \E_p N^{\rm par}(v) \, , $$
where $N^{\rm ser}(v)$ is the number of distinct ordered pairs of edges $e,f$ such that $0 \lr v$ when $e$ and $f$ are both $p$-open, but not if one of them is $p$-closed, and, $N^{\rm par}(v)$ is the number of distinct ordered pairs of edges $e,f$ such that $0 \lr v$ if one of $e$ or $f$ are $p$-open, but not when both are $p$-closed. If $e=(u,u'),f=(w,w')$ are counted in $N^{\rm ser}(v)$, then there must be an open path from $0$ to $v$ traversing through these two edges, hence
$$ \sum_{v \in \Z_n^d} \E_p N^{\rm ser}(v) \leq 2 \sum_{v,(u,u'), (w,w')} \proba_p(0 \lr u) \proba_p(u'\lr w) \proba_p(w' \lr v) \pe \X_\T(p)^3 \, .$$
Next, if $e=(u,u'),f=(w,w')$ are counted in $N^{\rm par}(v)$, then there must exists vertices $x,y$ such that 
$$ \{0 \lr x\} \circ \{x \lr u\} \circ \{u' \lr y\} \circ \{x \lr w\} \circ \{w' \lr y\} \circ \{y \lr v\} \, .$$
We use the BK inequality and use $\proba_p(u' \lr y) \pe \proba_p(u \lr y)$ since $(u,u')$ is an edge to get
\begin{eqnarray*} \sum_{v \in \Z_n^d} \E_p N^{\rm par}(v) &\pe& \sum_{x,u,w,y,v} \proba_p(0 \lr x)\proba_p(x \lr u)\proba_p(u \lr y)\proba_p(x \lr w)\proba_p(w \lr y) \proba_p(y \lr v)\\ &=& \chi_\T(p)^2 T_4(0) \pe \chi_\T(p)^6/V \, ,\end{eqnarray*}
where the last inequality is due to \cref{lem:T345}. Since $\chi_\T(p) = o(V^{1/3})$ the last bound is $o(\chi_\T(p)^3)$. 
\end{proof}

\begin{proof}[Proof of \cref{lem:sharpChiChi2Torus}] In what follows we abbreviate $p_c=p_c(\Z^d)$.
In \cite{HeydenHofstad1} the authors provide a coupling such that 
\be \label{eq:CouplingUpper} |\C^\T(0)| \leq |\C^\Z(0)| \, ,\ee
with probability $1$, and such that 
\be\label{eq:CouplingLower} \chi_\T(p) \geq \chi_\Z(p)\Big ( 1 - (1/2+2pd\chi_\Z(p)) \sum_{\substack{x\neq y \in \Z^d \\ x \stackrel{n}{\sim} y }} \tau_{\Z,p}(0,x)\tau_{\Z,p}(0,y) \Big ) \, ,\ee
where $\tau_{\Z,p}(x,y)=\proba_p(x \lr y)$ is the probability that $x$ is connected to $y$ in percolation on $\Z^d$ with edge probability $p$ and $x \stackrel{n}{\sim} y$ means that $x-y \equiv 0 \mod n$ (recall that $n$ is the side length of the torus $\Z_n^d$). Inequality \eqref{eq:CouplingUpper} is shown in \cite[Propoisition 2.1]{HeydenHofstad1} and \eqref{eq:CouplingUpper} follows by plugging in \cite[(5.9)]{HeydenHofstad1} and \cite[(5.13)]{HeydenHofstad1} into \cite[(5.5)]{HeydenHofstad1}. Furthermore, in \cite[Lemma 2.3]{HeydenHofstad2} the same authors bound the sum on the right hand side of \eqref{eq:CouplingLower} when $p \leq p_c - K^{-1}V^{-1/3}$ for some large constant $K$. There they bound the sum by $V^{-1/3}$, but inspecting their proof (the term (A) in the proof of \cite[Lemma 2.3]{HeydenHofstad2} is dominant) shows the slightly better bound 
$$ \sum_{\substack{x\neq y \in \Z^d \\ x \stackrel{n}{\sim} y }} \tau_{\Z,p}(0,x)\tau_{\Z,p}(0,y) \pe V^{-1} \eps^{-2} \, ,$$
holds (we abbreviate $\eps=\eps(n)$). 
Plugging this into \eqref{eq:CouplingLower} gives the first equality of  \eqref{eq:sharpChi} and using \cref{thm:sharpChiChi2Zd} gives the second equality.

Next, to prove \eqref{eq:sharpChi2} we first note that for any constant $A\geq 1$, using \eqref{eq:sharpChi}, we have
\be\label{eq:sharpTorusMidstep}\E_p \Big ( |\C^\Z(0)|{\bf 1}_{|\C^\Z(0)| \leq A \chi_\Z(p)^2} \big [ |\C^\Z(0)| - |\C^\T(0)| \big ] \Big ) \pe A V^{-1} \eps^{-6} \, .\ee
Next, we recall the classical bound \cite[Proposition 5.1]{TreeGraph}
$$ \proba_p(|\C^\Z(0)|\geq k) \leq (e/k)^{1/2} \exp\left (\frac{k}{2\chi_\Z(p)^2}\right ) \, ,$$
valid for all $p<p_c$ and $k\geq \chi_\Z(p)^2$. We use it together with a straightforward calculation to obtain that for any $p<p_c$ and any $A\geq 1$ we have
$$ \E_p|\C^\Z(0)|^2{\bf 1}_{|\C^\Z(0)| \geq A \chi_\Z(p)^2}  \pe \sum_{k \geq A^2 \chi_\Z(p)^4} k^{-1/4} e^{-\sqrt{k}/2\chi_\Z(p)^2} \pe \chi_\Z(p)^3 e^{-cA} \, ,$$
for some universal constant $c>0$. We put this into \eqref{eq:sharpTorusMidstep} and obtain that 
$$ \E_p \Big [ |\C^\Z(0)| \big ( |\C^\Z(0)| - |\C^\T(0)| \big ) \Big ] \pe A V^{-1} \eps^{-6}+ \eps^{-3} e^{-cA} \, ,$$
for any $A\geq 1$. Since $V^{-1} \eps^{-6} = o(\eps^{-3})$ we can choose $A=\log(\eps^3 V)\to \infty$ so that the last bound is $o(\chi(p)^3)$. Therefore, using \eqref{eq:CouplingUpper} we get 
$$ \E_{p}\big [\C^\Z(0)^2 - \C^\T(0)^2\big ] \leq 2\E_p \Big [ |\C^\Z(0)| \big ( |\C^\Z(0)| - |\C^\T(0)| \big ) \Big ]  = o(\chi(p)^3) \, ,$$
which by \cref{thm:sharpChiChi2Zd} concludes the proof of \eqref{eq:sharpChi2}. 

To prove \eqref{eq:sharpChi'Chi2}, integration by parts yields that for any  $p<p_2<p_c$ we have
\[ \chi_\T(p_2)-\chi_\T(p)=(p_2-p)\chi'_\T(p)+ \int_p^{p_2} \X_\T''(t)(p_2-t)dt.\]
Let $\eta\in(0,1/2)$ and let $p^-=p-(p_c-p)\eta$ and let $p^+=p+(p_c-p)\eta$. By the equality above with $p_2=p^+$ we have 
\[ (p^+ -p)\chi'_\T(p)\geq \chi_\T(p^+)-\chi_\T(p)-(p^+-p)^2\max_{t\in [p,p^+]} |\X_\T''(p)| \, . \]
Using \eqref{eq:sharpChi} and \cref{lem:chi''} we get 
\begin{eqnarray*}
\eta(p_c-p)\chi'_\T(p) &\geq& (\Cnst_1+o(1))(p_c-p)^{-1}(1-\eta)^{-1} -(\Cnst_1+o(1))(p_c-p)^{-1} - C \eta^2 (p_c-p)^{-1} 
\\ &=& (\Cnst_1+o(1))(p_c-p)^{-1} \eta (1-\eta)^{-1} - C \eta^2 (p_c-p)^{-1}
\, , 
\end{eqnarray*}
for some universal constant $C\in(0,\infty)$. Dividing by $\eta(p_c-p)$ and taking $\eta\to 0$, we deduce 
\[ \chi_\T'(p)\geq (\Cnst_1+o(1))(p_c-p)^{-2}. \]
Similarly, by using $p^-$ we get 
\[ \X_\T'(p)\leq (\Cnst_1+o(1))(p_c-p)^{-2}, \] 
which together with \eqref{eq:sharpChi} concludes the proof of \eqref{eq:sharpChi'Chi2}.

Lastly, to prove \eqref{eq:Chi'TvsZ} let us write $\X_\Delta(p)=\X_\Z(p)-\X_\T(p)$. Using integration by parts again, for every $p^+\in(p,p_c)$ we have
$$ (p^+-p) \X_\Delta'(p) = \X_\Delta(p^+)-\X_\Delta(p) - \int_p^{p^+} \X_\Delta''(t)(p^+-t)dt \, ,$$
so 
\be\label{eq:chiDelta'} | \X_\Delta'(p) | \leq {|\X_\Delta(p^+)-\X_\Delta(p)| \over (p^+-p)} + \max_{p'\in [p,p^+]}|\X''_\Delta(p')|(p^+-p) \, .\ee
We take 
$$ p^+= p+(p_c-p) (\eps^{3}V)^{-1/2} = p+\eps (\eps^{3}V)^{-1/2} \, ,$$ so that $p^+\in (p,p+\eps/2)$ as long as $n$ is large enough, in which case $\X_\T(p^+)\pe \eps^{-1}$.  By \eqref{eq:sharpChi} with \cref{thm:subcriticalBCHSS} we have that $\X_\Delta(p') \pe (\eps^3 V)^{-1} \X_\T(p')$ for any $p' \in [p,p+\eps/2]$. We put these into \eqref{eq:chiDelta'} and get
$$ | \X_\Delta'(p) | \pe \frac{ (\eps ^3 V)^{-1} \eps^{-1}}{\eps(\eps^{3}V)^{-1/2} } + \eps(\eps^{3}V)^{-1/2} \max_{p'\in [p,p^+]}|\X''_\Delta(p')| \, .$$
The first term is $(\eps ^3 V)^{-1/2}\eps^{-2}$ as desired. For the second term we bound $|\X''_\Delta(p')| \leq |\X''_\T(p')|+|\X''_\Z(p')|$ so by \cref{lem:chi''} we have $|\X''_\Delta(p')| \pe \eps^{-3}$ and the proof \eqref{eq:Chi'TvsZ} and the lemma is concluded.
\end{proof}

\section{Subcritical estimates}\label{sec:subcriticalEstimates}
In the whole section $p$ is such that $p\leq p_c$ and $\X(p)\geq V^{0.3}$.

\subsection{Some generic sums}
 Let $\star$ denote the usual convolution operator on the torus $\Z_n^d$, that is, 
 $$ (f \star g) (x) = \sum_{y\in \Z_n^d} f(x-y)g(y) \, ,$$
for functions $f,g:\Z_n^d \to \R$.
 \begin{claim}\label{clm:Convolution} Assume that $a\geq b>0$ are fixed and $f,g:\Z_n^d \to \R$ satisfy $|f(x)| \pe \l x \r^{-a}$ and $|g(x)| \pe \l x \r^{-b}$. Then there exists $C=C(d,L,a,b)\in (0,\infty)$ such that 
 $$ (f \star g) (x) \leq \begin{cases}
		C\l x \r^{-b}, & \text{if $a>d$}\\
            C \l x \r ^{d - (a+b)}, & \text{if $a<d$ and $a+b>d$} \\
            C \log V, & \text{if $a<d$ and $a+b=d$} \\
            C V^{{d-(a+b)\over d}}, & \text{if $a<d$ and $a+b<d \, .$}
		 \end{cases}$$
 \end{claim}
 \begin{proof}
 The proof is precisely that of \cite[Proposition 1.7]{HHS03}. There the authors prove bound in the first two cases on $\Z^d$ and the proof works verbatim. On $\Z^d$ in the other two cases, the convolution does not necessarily converge; on $\Z_n^d$ however, we get the bounds above. We omit the details. %
 \end{proof}
 
We will frequently encounter in this section the following sums.
\begin{eqnarray*}
T_2(u-v) &:=& \sum_z \proba_p(u \lr z) \proba_p(z \lr v)  \, , \\
T_3(u-v) &:=& \sum_{y,z} \proba_p(u \lr y) \proba_p(y \lr z) \proba_p(z \lr v) \, ,\\
T_4(u-v) &:=& \sum_{x,y,z} \proba_p(u \lr x)\proba_p(x \lr y)\proba_p(y \lr z) \proba_p(z \lr v) \, .
\end{eqnarray*}

\begin{lemma}\label{lem:T345} For every $u,v\in V$ and $p\leq p_c$ such that $\X(p)\geq V^{0.3}$ we have
\begin{eqnarray*}
T_2(u-v) &\pe& \l u-v \r ^{4-d} + \chi(p)^2/V \, , \\
T_3(u-v) &\pe&  \l u-v\r^{6-d} + \chi(p)^3/V  \, ,\\
T_4(u-v) &\pe&  \chi(p)^4/V. %
\end{eqnarray*}
\end{lemma}
\begin{proof} For $T_2(u-v)$ we apply \cref{thm:plateau} to get
$$ T_2(u-v) \pe \sum_{z} \Big (\l u-z \r^{2-d} + V^{-1} \chi(p)\Big )\Big (\l v-z \r^{2-d} + V^{-1} \chi(p) \Big ) \, .$$
By \cref{clm:Convolution}
$$ \sum_z \l u-z \r^{2-d}\l v-z \r^{2-d} \pe \l u-v \r^{4-d} \, ,$$
and the cross term 
$$ V^{-1} \chi(p) \sum_{z} \l v-z \r^{2-d} \pe V^{-1} \chi(p) V^{2/d} = o(\chi(p)^2/V) \, ,$$
since $d \geq 7$. For $T_3$ we apply \cref{thm:plateau} and obtain
$$ T_3(u-v) = \sum_{y} \proba_p(u \lr y) T_2(y-v) \pe \sum_y \Big ( \l u-y \r^{2-d} + V^{-1}\chi(p) \Big ) \Big ( \l y-v \r^{4-d} + V^{-1} \chi^2(p) \Big ) \, ,$$
from which the desired bound follows again using \cref{clm:Convolution} and a straightforward calculation using $d\geq 7$ and $\chi(p) \pe V^{1/3}$. For $T_4(u-v)$ we proceed similarly since \cref{clm:Convolution} gives that
$$ \sum_{x} \l u-x \r^{2-d} \l v - x \r^{6-d} \pe \l u - v \r^{8-d} + f_4(V) \, ,$$
where 
$$
f_4(V) =\begin{cases}
			V^{1/d}, & \text{if $d=7$}\\
            \log V, & \text{if $d=8$} \\
            1, & \text{if $d>8$}
		 \end{cases} $$
Using $d\geq 7$, and $\X(p)\geq V^{0.3}$ we get $\l u - v \r^{8-d} + f_4(V)\pe \X(p)^4/V$.
\end{proof}

\subsection{Proof of \cref{thm:MaxDelta}} \label{subsec:M1}

The proof will be obtained by estimating high moments of $\sum_{a,b} \Delta_{a,b}$ and will require the analysis of certain diagrams for which we provide the following auxiliary result. Let $T$ be a finite tree with vertex set $V(T)=\{1,2,\dots,t\}$. Assume that $L=\{1,\ldots, \ell\}$ is a subset of $T$'s leaves with $1 \leq \ell < t$. Given $\ell$ vertices of the torus $u_1,\ldots, u_{\ell} \in \Z_n^d$ we define
\be\label{def:STL} S_{T,L}(u_1,\ldots, u_{\ell}) = \sum_{\substack{u_{\ell+1},\ldots,u_{t}}} \prod_{(i,j)\in T} \proba_p(u_i \lr u_j)\, .\ee

\begin{lemma} \label{lem:SumTreeSquare} Let $T$ be a finite tree of maximum degree at most $3$ with vertices $\{1,2,\dots,t\}$ and $L=\{1,2,\dots, \ell\}$ is a subset of leaves in $T$. Assume that  $(t,\ell)\neq (4,3)$. Then
$$ \sum_{u_1,\ldots, u_\ell} S_{T,L}^2(u_1,\ldots, u_\ell) \leq {C ^t \chi(p) ^{2t-2} \over V^{\ell-2}} \, ,$$
for some constant $C=C(d,L)\in(0,\infty)$.
\end{lemma}

The proof will be performed by induction.To perform it, given such a pair $(T,L)$, assume that $\rho$ is a vertex in $V(T)\setminus L$ and that $u_\rho, v_\rho \in V(G)$ are vertices of the torus and define
$$ S_{T,L,\rho}(u_1,\ldots,u_\ell,u_\rho) = \sum_{\substack{u_{i} : i \in \{\ell+1,\ldots,t\} \setminus \{\rho\}}} \prod_{(i,j)\in T} \proba_p(u_i \lr u_j) \, .$$

\begin{lemma}\label{lem:FourSmallCases} Suppose $(T,L,\rho)$ is one of the four triplets depicted in \cref{fig:ReduceTree2}. Denoting $|V(T)|=t$ and $|L|=\ell$, we have
\be\label{eq:induction}
\max_{u_\rho, v_\rho \in V(G)} \sum_{u_1,\ldots, u_\ell} S_{T,L,\rho}(u_1,\ldots, u_\ell,u_\rho) S_{T,L,\rho}(u_1,\ldots, u_\ell,v_\rho)  \leq {C ^t \chi(p) ^{2t-2} \over V^{\ell}} \, ,
\ee
for some $C=C(d,L)\in(0,\infty)$.
\begin{figure}[!ht]
        \centering
        \includegraphics[width=.8\textwidth]{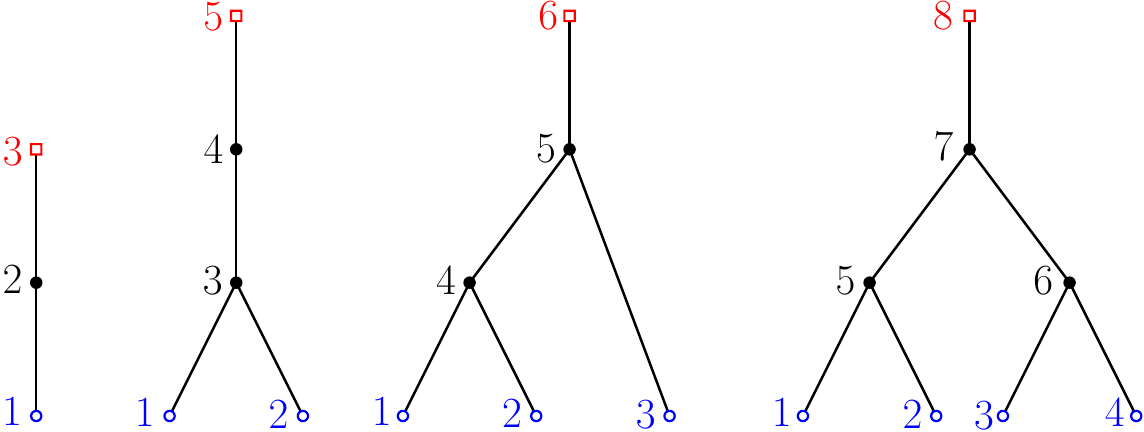}
        \caption{The four triplets $(T,L,\rho)$ of \cref{lem:FourSmallCases}. In each, $L$ is colored blue, and $\rho$ is colored red.%
        }
        \label{fig:ReduceTree2}
\end{figure}
\end{lemma}
\begin{proof} It will be convenient to first estimate the left hand side of \eqref{eq:induction} for the ``smaller'' cases depicted in \cref{fig:SubReduceTree2}; for these the bound in \eqref{eq:induction} does not hold and  we need other estimates.
\begin{figure}[!ht]
        \centering
        \includegraphics[width=.8\textwidth]{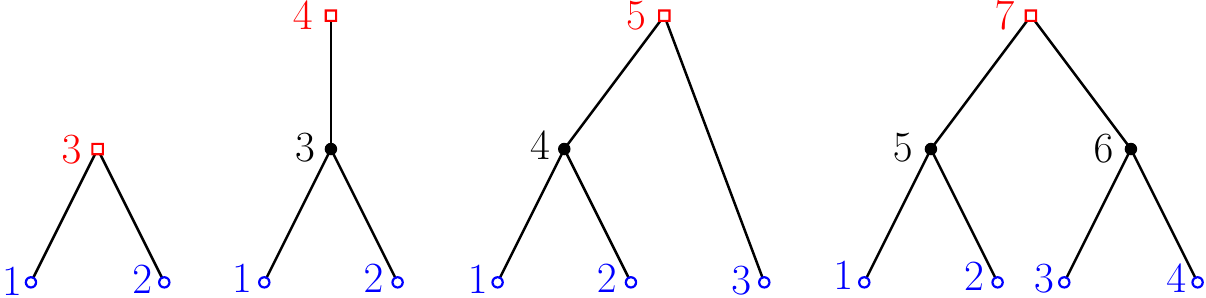}
        \caption{Four additional triplets $(T,L,\rho)$. Again, $L$ is in blue, and $\rho$ is in red. %
        }
        \label{fig:SubReduceTree2}
\end{figure}
For $i=1,2,3,4$ denote by $S_i(u_\rho,v_\rho)=S_i(v_\rho-u_\rho)$ the sum on the left hand side of \eqref{eq:induction} in the four cases depicted in \cref{fig:SubReduceTree2} ordered from left to right. We estimate them one by one. First,
$$ S_1(u_3,v_3) = \sum_{u_1,u_2}   \proba_p(u_1 \lr u_3)\proba_p(u_2 \lr u_3)\proba_p(u_1 \lr v_3)\proba_p(u_2 \lr v_3) \, ,$$
which equals
\begin{eqnarray} \left (\sum_{u}\proba_p(u \lr u_3)\proba(u\lr v_3) \right )^2=T_2(u_3-v_3)^2 \pe \l u_3-v_3 \r^{8-2d} + {\chi(p)^4\over V^2} \, , \label{eq:RecS1} \end{eqnarray}
where the last inequality is due to \cref{lem:T345}. Secondly we have
\begin{eqnarray*} S_2(u_4,v_4)&=&\sum_{u_3,v_3} \proba_p(u_4\lr u_3) S_1(u_3-v_3)\proba_p(v_3\lr v_4 ) \\ &=& \sum_{z} S_1(z) \sum_{v_3} \proba_p(u_4 - z \lr v_3)\proba_p(v_3 \lr v_4) = \sum_z S_1(z) T_2(u_4- v_4 -z) 
\end{eqnarray*}
So by \cref{lem:T345} and \eqref{eq:RecS1} we get
\[ S_2(u_4,v_4) \pe \sum_{z} \Big (\l u_4 -v_4-z \r^{4-d}+{\chi(p)^2\over V} \Big )\Big (\l z \r^{8-2d}+{\chi(p)^4\over V^2} \Big ) \, . \]
By \cref{clm:Convolution} (second case since $d>6$) we have that 
$$ \sum_z \l u_4 -v_4-z \r^{4-d} \l z \r^{8-2d} \pe \l u_4 -v_4 \r^{12-2d} \, ,$$
and the cross terms 
$$ {\chi(p)^4 \over V^2} \sum_{z} \l u_4 -v_4-z \r^{4-d} \pe {\chi(p)^4 V^{4/d} \over V^2} \pe {\chi(p)^6 \over V^2} \, ,$$
$${\chi(p)^2 \over V} \sum_z \l z \r^{8-2d} \pe {\chi(p)^2 V^{1/d}\over V} \pe {\chi(p)^6 \over V^2} \, ,$$
since we assume that $\chi(p)  \geq V^{0.3}$ and $d\geq 7$. Thus,
\begin{equation}  S_2(u_4,v_4) \pe \l u_4-v_4 \r^{12-2d}+{\chi(p)^6 \over V^2} \, . \label{eq:RecS2} \end{equation}
Thirdly, by \eqref{eq:RecS2} and \cref{lem:T345}, we have
\[ S_3(u_5,v_5)= S_2(u_5,v_5) T_2(u_5-v_5)\pe \Big (\l u_5-v_5 \r^{12-2d}+{\chi(p)^6 \over V^2} \Big ) \Big (\l u_5-v_5 \r^{4-d}+{\chi(p)^2\over V} \Big  ) \, , \]
which gives
\begin{equation} S_3(u_5,v_5)\pe \l u_5-v_5 \r^{16-3d}+ \l u_5-v_5 \r^{4-d} {\chi(p)^6 \over V^2}+{\chi(p)^8 \over V^3} \, , \label{eq:RecS3} \end{equation}
since $\l u_5-v_5 \r^{12-2d} \chi(p)^2/V \leq \l u_5-v_5 \r^{4-d} \chi(p)^6/V^2$ because $\l u_5-v_5 \r \leq V^{1/d}$ and $\chi(p) \geq V^{0.3}$ and $d\geq 7$. Lastly, by \eqref{eq:RecS2} we have
\begin{equation} S_4(u_7,v_7)= S_2(u_7,v_7)^2 \pe \l u_7-v_7 \r^{24-4d}+{\chi(p)^{12} \over V^4} \, . \label{eq:RecS4} \end{equation}

We have finished upper bounding the auxiliary diagrams depicted in \cref{fig:SubReduceTree2}. For $i=5,6,7,8$ denote by $S_i(u_\rho,v_\rho)$ the sum on the left hand side of \eqref{eq:induction} in the four cases depicted in \cref{fig:ReduceTree2} ordered from left to right. We estimate them now one by one. First, by \cref{lem:T345}, we have
\[ S_5(u_3,v_3)=T_4(u_3-v_3)%
\pe {\chi(p)^4 \over V} \, ,\]
, yielding the desired bound for the first diagram in \cref{fig:ReduceTree2}. Next we have 
\[ S_6(u_5,v_5)=\sum_{u_4,v_4} \proba(u_5\lr u_4) S_2(u_4,v_4)\proba(v_4\lr v_5) \, , \]
which we may rewrite by symmetry as 
\[ S_6(u_5,v_5)=\sum_z S_2(z) T_2(u_5-v_5-z) \, , \]
in the same way we did for $S_2$ below \eqref{eq:RecS1}. Using \cref{lem:T345} and \eqref{eq:RecS2} we obtain
\[ S_6(u_5,v_5)\pe \sum_z \Big (\l z \r^{12-2d}+{\chi(p)^6 \over V^2} \Big ) \Big (\l u_5-v_5-z \r^{4-d}+{\chi(p)^2 \over V}\Big ) \, . \]
Summing over $z$ the second term in each parenthesis gives the desired term of $\chi(p)^8/V^2$. Summing using \cref{clm:Convolution} over $z$ the first term in each parenthesis gives 
$$ \sum_z \l z \r^{12-2d} \l u_5-v_5-z \r^{4-d} \pe  V^{2/d} \pe {\chi(p)^8 \over V^2} $$
since since $d \geq 7$ and $\chi(p) \geq V^{0.3}$. The cross terms are also of lower order, indeed,
$$ {\chi(p)^2 \over V} \sum_z \l z \r^{12-2d} \pe {\chi(p)^2 V^{5/d} \over V} \pe {\chi(p)^8 \over V^2} \, ,$$
$$ {\chi(p)^6 \over V^2} \sum_{z} \l u_5-v_5-z \r^{4-d} \pe {\chi(p)^6 V^{4/d} \over V^2} \pe {\chi(p)^8 \over V^2} \, , $$
since $d \geq 7$ and $\chi(p) \geq V^{0.3}$. Hence $S_6(u_5,v_5)\pe \chi(p)^8/V^2$ as required. Similarly, 
\[ S_7(u_6,v_6)=\sum_{u_5,v_5} \proba(u_6\lr u_5) S_3(u_5,v_5)\proba(v_5\lr v_6) = \sum_z S_3(z) T_2(u_6-v_6-z)\, . \]
Hence, by \cref{lem:T345} and \eqref{eq:RecS3}, we get
\[ S_7(u_6,v_6)\pe \sum_z \Big (\l z \r^{16-3d}+ \l z \r^{4-d} {\chi(p)^6 \over V^2}+{\chi(p)^8 \over V^3} \Big )\Big (  \l u_6-v_6-z \r^{4-d}+{\chi(p)^2 \over V} \Big ) \, .\]
As before the main term of $\chi(p)^{10}/V^3$ comes from the last term in each of the two parenthesis. For the other terms we bound, using $d\geq 7$ and \cref{clm:Convolution}, the following:
\[ \sum_z \l z \r^{16-3d} \pe V^{2/d} \quad \text{and} \quad \sum_z \l z \r^{4-d} \pe V^{4/d} \, , \]
and
\[ \sum_z \l z \r^{16-3d} \l u_6-v_6-z \r^{4-d} \pe \l u_6 - v_6 \r ^{20-3d} \leq 1 \, ,\]
and 
\[ \sum_z \l z \r^{4-d} \l u_6-v_6-z \r^{4-d} \pe V^{1/d} \, ,\]
Since $\chi(p) \geq V^{0.3}$ it is now straightforward that $S_7(u_6, v_6)\pe \chi(p)^{10}/V^3$ as required. Finally, we have 
\[ S_8(u_8,v_8)=\sum_{u_7,v_7} \proba(u_8\lr u_7) S_4(u_7,v_7)\proba(v_7\lr v_8) = \sum_z S_4(z) T_2(u_8-v_8-z) \, . \]
By symmetry as before, \cref{lem:T345} and \eqref{eq:RecS4} we get
\[ S_8(u_8,v_8)\pe \sum_z \Big (\l z \r^{24-4d}+{\chi(p)^{12} \over V^4} \Big ) \Big (\l u_8-v_8-z \r^{4-d}+{\chi(p)^2 \over V} \Big ) \, , \]
As usual, the main term $\X(p)^{14}/V^4$ comes from the product of the second terms in the parentheses. For the other terms we have
\[ \sum_z \l z \r^{24-4d} \pe V^{3/d} \quad \text{and} \quad \sum_z \l u_8 -v_8 -z \r^{4-d} \pe V^{4/d} \, \quad \text{and} \quad \sum_z \l u_8 -v_8 -z \r^{4-d}\l z \r^{24-4d} \pe \log(V), \]
and the required bound $S_8(u_8,v_8) \pe \chi(p)^{14}/V^4$ follows since $\chi(p)\geq V^{0.3}$ and $d\geq 7$.
\end{proof}

\begin{proof}[Proof of \cref{lem:SumTreeSquare}] We prove by induction. For convenience let us denote 
$$ F(T,L) = \sum_{u_1,\ldots, u_\ell} S_{T,L}^2(u_1,\ldots, u_\ell) \, .$$

The base of the induction is the tree on one edge, i.e., the case $(t,\ell)=(2,1)$ in which $F(T,L)=V\chi(p)^2$, or, the case $(t,\ell)=(3,2)$ (two leaves and one parent) in which $F(T,L)=\sum_{u,v}T_2(u-v)^2$ which an immediate calculation with \cref{lem:T345} gives a bound of $O(\chi(p)^4)$ since $\chi(p)\geq V^{0.3}$, or, one of the trees depicted in \cref{fig:ReduceTree2} for which \cref{lem:FourSmallCases} and summing over $u_\rho$ and $v_\rho$ also yields the desired result. 

For the induction step we may first assume that $L$ is the entire set of leaves of $T$. Indeed, otherwise, if $m \in V(T) \setminus L$ is a leaf of $T$, we may sum over $v_m$ and $u_m$ and get that $F(T,L) = F(T\setminus \{m\}, L) \chi(p)^2$ and use the induction hypothesis. Next, we claim that we can find in $(T,L)$ a copy  $(T',L',\rho)$ of one of the four diagram depicted in \cref{fig:ReduceTree2} inside $(T,L)$ such that:
\begin{enumerate} 
\item[(a)] Any edge connecting $T'$ to $T\setminus T'$ is an edge between $\rho$ and $T\setminus T'$, and,
\item[(b)] The tree obtained by removing vertices $V(T')\setminus \rho$ from $T$  is \emph{not} a tree on $4$ vertices with $3$ leaves.
\end{enumerate}
To show this let us define the {\bf depth} of a vertex in $T$ in a recursive fashion. First, we set all leaves to have depth $0$. Next, once we have defined the set of tree vertices of height at most $i$ for some $i\geq 0$, if this set is not all $V(T)$, then removing it from $V(T)$ results in a forest which must have a vertex of degree $0$ or $1$ --- we set these vertices to have depth $i+1$ and proceed recursively. Note that vertices of non-zero depth are not leaves of $T$.

If there is a vertex of depth $1$ and degree $2$, then we are either in the base case $(t,\ell)=(3,2)$ of the induction, or this vertex has one neighbor of depth $0$ (i.e., in $L$) and the other of depth $2$ in which case there is a copy of the first  tree depicted in \cref{fig:ReduceTree2} in $(T,L)$ satisfying (a) and (b). Thus we may assume that all vertices of depth $1$ have degree $3$. Since we excluded the case $(t,\ell)=(4,3)$, we obtain that if a vertex of depth $1$ and degree $3$ exists, then it must be connected to a vertex of depth $2$. If this vertex of depth $2$ has degree $2$, one of its neighbors is of depth $1$ and has two additional leaves connected to it, and the second neighbor we set as the root $\rho$; we have found a copy of the second tree in \cref{fig:ReduceTree2} satisyfing (a) and (b). Finally, if there exist a vertex $v$ of depth $2$ and degree $3$, then one of its neighbors $v_1$ must be a vertex of depth $1$ and degree $3$ so $v_1$ has two leaves attached to it. The vertex $v$ must have two additional neighbors $v_2,v_3$. If one of them, say $v_2$ has depth $0$ we set $v_3$ to be the root $\rho$ and note that we have reached the third case of \cref{fig:ReduceTree2}. Lastly, if both $v_2$ and $v_3$ have non-zero depth, one of them, say $v_2$, must be of depth $1$ and therefore has two additional leaves attached to it; we again set $v_3$ to be the root $\rho$ and note that we have reached the fourth case of \cref{fig:ReduceTree2}.

Having obtained such a triplet $(T',L',\rho)$ where $T'$ is a subtree of $T$ satisfying (a) and (b) above and $L' \subset L$, we may label the vertices so that $L'=\{1,\ldots, \ell'\}$, $L=\{1,\ldots, \ell\}$, then 
$$ S_{T,L}(u_1,\ldots,u_\ell) = \sum_{u_\rho} S_{T\setminus(T' \setminus \rho), L\setminus L', \rho}(u_{\ell'+1},\ldots,u_{\ell},u_\rho) S_{T',L',\rho'}(u_1,\ldots, u_{\ell'},u_\rho) \, .$$
So, by squaring and summing over $u_1,\ldots,u_\ell$ then applying \cref{lem:FourSmallCases} we obtain
\begin{eqnarray*} F(T,L) &\leq& {C^{t'}\chi^{2t'-2} \over V^{\ell'}} \sum_{\substack{u_{\ell'+1},\ldots, u_\ell\\ u_\rho,v_\rho}}  S_{T\setminus(T' \setminus \rho), L\setminus L', \rho}(u_{\ell'+1},\ldots,u_{\ell},u_\rho)S_{T\setminus(T' \setminus \rho), L\setminus L', \rho}(u_{\ell'+1},\ldots,u_{\ell},v_\rho)
\\ &=& {C^{t'} \chi^{2t'-2} \over V^{\ell'}} \sum_{\substack{u_{\ell'+1},\ldots, u_\ell}} S_{T\setminus(T' \setminus \rho), L\setminus L'} ^2 (u_{\ell'+1},\ldots,u_{\ell}) \, .
\end{eqnarray*}
By our induction hypothesis 
$$ \sum_{\substack{u_{\ell'+1},\ldots, u_\ell}} S_{T\setminus(T' \setminus \rho), L\setminus L'} ^2 (u_{\ell'+1},\ldots,u_{\ell}) \leq {C^{t-t'} \chi^{2(t-t'+1)-2} \over V^{\ell-\ell'-2}} \, ,$$
which yields
$$ F(T,L) \leq {C^t \chi^{2t-2} \over V^{\ell-2} } \, ,$$
as required. \end{proof}

\begin{proof}[Proof of \cref{thm:MaxDelta}] Though possible, we make no attempt in finding tight power of $\log$ on the right hand side of the theorem. For any integer $k\geq 2$ we have
\begin{eqnarray*} \E\Big [\sum_{a,b}\Delta_{a,b}^k \Big ] &=& \sum_{(u_1,v_1),\dots, (u_k,v_k) \in E} \proba_p(u_2,\ldots,u_k\in \C(u_1), v_2, \ldots,v_k \in \C(v_1), \C(u_1)\neq \C(v_1)) 
\\ &\leq& \sum_{(u_1,v_1),\dots, (u_k,v_k) \in E} \proba_p(u_2,\ldots,u_k\in \C(u_1)) \proba_p(v_2, \ldots,v_k \in \C(v_1))\, , \end{eqnarray*}
by the BK inequality. Let $\mathcal{T}_k$ denote the set of labeled trees on vertex set $\{1,2,\ldots,2k-2\}$ such that $\{1,\ldots, k\}$ are leaves and $\{k+1,\ldots,2k-2\}$ have degree $3$. We now appeal to the classical tree-graph inequalities (see \cite{TreeGraph}) to obtain that
$$ \E\Big [\sum_{a,b}\Delta_{a,b}^k \Big ] \leq \frac{1}{(k-2)!^2} \sum_{T_u,T_v\in \mathcal T_k} \sum_{\substack{(u_1,v_1),\dots, (u_k,v_k) \in E \\ u_{k+1},\dots, u_{2k-2} \\ v_{k+1},\dots, v_{2k-2}  }} \,\, \prod_{(i,j)\in T_u} \proba_p(u_i\slr u_j) \prod_{(i,j)\in T_v}\proba_p(v_i\slr v_j) \, ,
$$
where the factor $1/(k-2)!^2$ is to account for the extra labels we give to the vertices of degree 3 in $T_u, T_v$.
Note that when $i\in \{1,2,\dots, k\}$ we have $\proba(v_i\slr v_j)\leq C \proba(u_i\slr v_j)$ for some constant $C=C(d,L)$ in $(0,\infty)$ since $(u_i,v_i)$ is an edge of the torus. So up to a multiplicative error of $C^k$ we may identify $u_i$ and $v_i$ for all $i=1,\ldots,k$ in the sum above. 
Thus, we may rewrite the last inequality as 
$$ \E\Big [\sum_{a,b}\Delta_{a,b}^k \Big ] \leq \frac{C^k}{(k-2)!^2} \sum_{T_u, T_v \in \mathcal{T}_k} \,\, \sum_{u_1,\ldots, u_k} S_{T_u, \{1,\ldots,k\}}(u_1,\ldots,u_k) S_{T_v, \{1,\ldots,k\}}(u_1,\ldots,u_k) \, .$$
where $S_{T,L}$ was defined in \eqref{def:STL}. We apply Cauchy-Schwartz and obtain
$$ \E\Big [\sum_{a,b}\Delta_{a,b}^k \Big ] \leq \frac{C^k}{(k-2)!^2} |\mathcal{T}_k| \sum_{T \in \mathcal{T}_k} \sum_{u_1,\ldots, u_k}  S_{T, \{1,\ldots,k\}}^2(u_1,\ldots,u_k) \, .
$$
Applying \cref{lem:SumTreeSquare} with $\ell=k$ and $t=2k-2$ gives
$$ \E\Big [\sum_{a,b}\Delta_{a,b}^k \Big ] \leq \frac{C^k}{(k-2)!^2} |\mathcal{T}_k|^2 \chi(p)^{4k-6}/V^{k-2} \, .$$
By \cite[(6.96)]{Grimm99} we have  
$|\mathcal{T}_k| = (2k-4)!2^{2-k}$ which we bound by $(2k)^{k}(k-2)!$. This gives
$$ \E\Big [\sum_{a,b}\Delta_{a,b}^k \Big ] \leq C^k k^{2k} \chi(p)^{4k-6}/V^{k-2} \, .$$
By Markov's inequality
$$ \proba_p( \max_{a,b} \Delta_{a,b} \geq \log(V)^{3} \chi^4/V) \leq C^k k^{2k} V^2 \chi(p)^{-6} \log(V)^{-3k}  \, .$$
We choose $k=\log V$ and conclude that 
$$\proba_p( \max_{a,b} \Delta_{a,b} \geq \log(V)^{3} \chi^4/V) =o(V^{-100}) \, ,$$ 
as desired.
\end{proof}

\subsection{Proof of \cref{thm:MaxSumDelta}}

We need two quick auxiliary estimates.

\begin{lemma} \label{lem:TriangleWithBigComp} If $p\leq p_c(\Z^d)$ and $M\geq V^{4/7}$, then
\[ \max_{u,v} \sum_{x,y} \proba(u\lr x) \proba(x \lr y, \, |\C(x)|\geq M) \proba(y \lr v)\pe  \chi^3/V. \]
\end{lemma}
\begin{proof} We split the sum according to whether one of the distances $\l u-x \r, \l x-y \r$ or $\l y-v \r$ is at least $L(p)= (V/\X(p))^{1/(d-2)}$, or not. If one of these distances is at least $L(p)$, then we omit the event $|\C(x)|\geq M$ and estimate the corresponding connection probability using \cref{thm:plateau} and pull out a $C\chi(p)/V$ factor. Summing now over $x$ and $y$ gives another $\chi(p)^2$ factor giving the desired bound.

Otherwise, we have that all three distances $\l u-x\r,\l x-y\r,\l y-v\r$ are at most $L(p)$. If $x \lr y, |\C(x)|\geq \M$ occurs, then the random variable 
$$ Z := \big | z \in V : \exists w \text{ with } \{x \lr w\} \circ \{w \lr z\} \circ \{w \lr y\} \big | \geq \M \, .$$
By the BK inequality and the union bound we have
$$ \E_p Z \leq \sum_{z,w} \proba_p(x \lr w)\proba_p(w \lr z)\proba_p(w \lr y) = \chi(p) T_2(x-y) \, ,$$
so by Markov's inequality 
\[ \proba(x \lr y, \, |\C(x)|\geq M) \leq \frac{\X(p)}{M}T_2(x-y) \leq \X(p)/M \Big ( \l x-y \r^{4-d} + {\chi(p)^2 \over V} \Big ) \, , \]
where the last inequality is by Lemma \ref{lem:T345}. Applying \cref{thm:plateau} for the other two connection probability, using the fact that the distance is at most $L(p)$, yields that the sum of the lemma (corresponding to the case $\l u-x\r,\l x-y\r,\l y-v\r \leq L(p)$) is at most  
\be\label{eq:TriangleCompStep} \frac{C\X(p)}{M}\sum_{x,y} \langle u-x\rangle^{2-d} \Big ( \l x-y \r^{4-d} + {\chi(p)^2 \over V} \Big )\langle y-v\rangle^{2-d} \, . \ee
Two applications of \cref{clm:Convolution} gives 
\[ \sum_{x,y}\l u-x\r^{2-d}\l x-y\r^{4-d} \l y-v\r^{2-d} \pe \sum_y \l u-y\r^{6-d} \l y-v\r^{2-d} \pe V^{1/d} \, ,\]
so after opening the parenthesis in \eqref{eq:TriangleCompStep}, since $d\geq 7$, the first term is at most $C\X(p)V^{1/7}/M$ which is $O(\X(p)^3/V)$ since $M\geq V^{4/7}$. For the second term \cref{clm:Convolution} gives
\[ \sum_{x,y}\langle u-x\rangle^{2-d} \langle y-v\rangle^{2-d} \pe V^{4/d} \, , \]
so the second term in \eqref{eq:TriangleCompStep} is at most $C\X(p)^3V^{4/7}/(MV)$ since $d\geq 7$ which again is $O(\X(p)^3/V)$ since $M\geq V^{4/7}$.
\end{proof}

\begin{lemma}\label{lem:ComponentsLarge} If $p\leq p_c(\Z^d)$ and $M\geq V^{4/7}$, then for any integer $k$ we have 
$$ \E_p \Big [ \sum_{a \in \comp_p} \sum_{\substack{b_1,\ldots,b_k \in \comp_{p,M}\\ |\{b_1,\ldots,b_k\}|=k}} \prod_{1 \leq i \leq k} \Delta_{a,b_i}^2 \Big ] \leq (Ck \X(p)^5/V)^k V/\X(p)^3 \, ,$$   
for some constant $C=C(d)\in(0,\infty)$.
\end{lemma}
\begin{proof}
The expectation on the left hand side equals the sum over $2k$ edges of $G$, denoted $(u_1,v_1)$ to $(u_{2k},v_{2k})$, of the probability of the event that
\begin{itemize}
\item $u_2,\ldots,u_{2k} \in \C(u_1)$ and
\item $v_1 \lr v_2, v_3 \lr v_4,\ldots, v_{2k-1}\lr v_{2k}$ and
\item $\C(u_1), \C(v_2), \C(v_4),\ldots,\C(v_{2k})$ are pairwise disjoint and $|\C(v_2)|, |\C(v_4)|,\ldots,|\C(v_{2k})|$ are all at least $M$.
\end{itemize}
By BK inequality we obtain that this is at most
$$ \sum_{(u_1,v_1),\ldots,(u_{2k},v_{2k})\in E} \proba_p( u_2,\ldots,u_{2k} \in \C(u_1)) \prod_{1 \leq i \leq k} \proba_p(v_{2i-1} \lr v_{2i}, |\C(v_{2i})| \geq M) \, .$$
As usual, since $(u_i,v_i)$ is an edge, there exists a constant $C=C(d,L)\in(1,\infty)$ such that 
$$ \proba_p(v_{2i-1} \lr v_{2i}, |\C(v_{2i})| \geq M) \leq C \proba_p(u_{2i-1} \lr u_{2i}, |\C(u_{2i})| \geq M) \, ,$$
so using this and tree graph inequalities (\cite{TreeGraph}) we get an upper bound on the expectation of
\be\label{eq:maxSumLemmaStep} C^k \sum_{u_1,\ldots, u_{2k}} \sum_{T \in \mathcal{T}_{2k}} S_{T, \{1,\ldots, 2k\}}(u_1,\ldots, u_{2k}) \prod_{1 \leq i \leq k}\proba_p(u_{2i-1} \lr u_{2i}, |\C(u_{2i})| \geq M) \, ,\ee
where $S_{T,\{1,\ldots, 2k\}}$ was defined in \eqref{def:STL} and $\mathcal{T}_{2k}$ again is the set of labeled trees on vertex set $\nolinebreak \{1,2,\ldots,4k-2\}$ such that $\{1,\ldots, 2k\}$ are leaves and $\{2k+1,\ldots,4k-2\}$ have degree $3$. 

We rewrite $S_{T,\{1,2,\dots, 2k\}}(u_1,\ldots,u_{2k})$ as follows. For every $T\in \mathcal{T}_{2k}$ and any leaf $i\in \{1,\ldots, 2k\}$ we write $N_T(i)$ for the unique neighbor of $i$ in $T$. Also, we write $T'$ for the tree obtained from $T$ after removing the leaves $\{1,\ldots, 2k\}$. Note that $T'$ has $2k-2$ vertices and $2k-3$ edges. We keep the vertex labeling of $T'$ to be consistent with $T$ so that $V(T')=\{2k+1,\ldots, 4k-2\}$ and denote its set of leaves by $L'$. Then
\[S_{T,\{1,2,\dots, 2k\}}(u_1,\dots, u_{2k})= \sum_{u_i: i \in L'} S_{T',L'}((u_i)_{i\in L'}) \prod_{1\leq i \leq 2k}\proba(u_i\lr u_{N_T(i)}) \, .\]
Thus \eqref{eq:maxSumLemmaStep} is bounded by
\begin{align*} & C^k \sum_{T \in \mathcal{T}_{2k}} \sum_{u_i: i \in L'}  S_{T',L'}((u_i)_{i\in L'}) \\ & \times \prod_{1\leq i \leq k} \Big ( \sum_{u_{2i-1},u_{2i}} \proba_p(u_{N_T(u_{2i-1})} \slr u_{2i-1}) \proba_p(u_{2i-1} \slr u_{2i}, |\C(u_{2i})|\geq M) \proba_p(u_{2i} \slr u_{N_T(u_{2i})}) \Big ) \, .
\end{align*}
Each of the $k$ sums over $u_{2i-1},u_{2i}$ is bounded by \cref{lem:TriangleWithBigComp} by $C\X^3/V$, so the product is bounded by $(C\X^3/V)^k$. Since $T'$ has $2k-2$ vertices, the sum over $u_i$ with $i\in L'$ gives a $V\X(p)^{2k-3}$ factor. Lastly, as before we crudely bound $|\mathcal{T}_{2k}|\leq (4k)^{4k}$. Putting all together gives a bound of $(Ck \X(p)^5/V)^k V/\X(p)^3$, concluding the proof. 
\end{proof}

\begin{proof}[Proof of \cref{thm:MaxSumDelta}] For each $a\in \comp_{p}$ write $S_a = \sum_{b \in \comp_{p,M}} \Delta_{a,b}^2$. So that for any integer $k\geq 1$ we have
\[ (S_a)^k=\sum_{b_1,\ldots, b_k} \Delta_{a,b_1}^2\cdots\Delta_{a,b_k}^2 \, .\]
This sum is composed of distinct $k$-tuples of $b_1,\ldots, b_k$ and $k$-tuples in which $b_i=b_j$ for some $i\neq j$ and $i,j\in\{1,\ldots,k\}$. In the latter case, we may bound $\Delta_{a,b_i}^2 \leq \max_{a,b} \Delta_{a,b}^2$. We get that for any $q \in \comp_{p}$ and any $k\geq 1$ we have
$$ (S_q)^k \leq \binom{k}{2}(S_q)^{k-1} \max_{a,b} \Delta_{a,b}^2 + \sum_{a} \sum_{|\{b_1,\ldots,b_k\}|=k} \Delta_{a,b_1}^2 \cdots \Delta_{a,b_k}^2 \, .$$
Since $x\leq y+z$ implies that either $x\leq 2y$ or $x\leq 2z$, we get that either 
$$ S_q \leq k^2\max_{a,b} \Delta_{a,b}^2 \, ,$$ 
or that 
$$ S_q \leq 2\Big (\sum_{a} \sum_{|\{b_1,\ldots,b_k\}|=k} \Delta_{a,b_1}^2 \cdots \Delta_{a,b_k}^2\Big )^{1/k} \, .$$
Both right hand sides do not depend on $q$ so the same bounds apply to $\max_q S_q$. Thus, if $\max_q S_q \geq \log(V)^{10}\X(p)^5/V$ then either 
$$ \max_{a,b} \Delta_{a,b} \geq \log(V)^5 (\X(p)^5/V)^{1/2}/k \, ,$$
or 
$$\sum_{a} \sum_{|\{b_1,\ldots,b_k\}|=k} \Delta_{a,b_1}^2 \cdots \Delta_{a,b_k}^2 \geq ( \log(V)^{10} \X(p)^5/V)^k/2 \, .$$
Taking $k=\lceil \log (V) \rceil$,
the first event occurs with probability $o(V^{-100})$ by \cref{thm:MaxDelta} and since $\chi(p)\leq V^{1/3}$. By \cref{lem:ComponentsLarge} and Markov's inequality the second event occurs with probability at most 
$$ (Ck/\log(V)^{10})^k V/ \X(p)^3 \, ,$$
and using $k=\lceil \log (V) \rceil$ gives another upper bound of $o(V^{-100})$, concluding the proof.
\end{proof}

\subsection{Second and third moment of component sizes}
\begin{proof}[Proof of \cref{lem:ConcentrationFixedp2}] First we have that 
$$ \E \sum_{A \in \comp_{p}} |A|^2 = \sum_{u} \E|\C(u)| = V\chi(p) \, .$$
We compute the second moment 
\begin{eqnarray*} \sum_{u,v} \E[|\C(u)||\C(v)|] &=& \sum_{u,v} \E[|\C(u)||\C(v)|{\bf 1}_{u \nlr v}] + \sum_{u,v} \E_{p_2}[|\C(u)||\C(v)|{\bf 1}_{u \lr v}] \, ,
\end{eqnarray*}
which is bounded above using the BK inequality by
$V^2\chi(p)^2 + V\E |\C(u)|^3$. By tree graph inequalities \eqref{eq:treegraph1} we get that $\E |\C(u)|^3 \pe \chi(p)^5$ and since $\chi(p) = o(V^{1/3})$ we get that
$$ \Var \sum_{A \in \comp_{p}} |A|^2 \pe V\chi(p)^5 \, ,$$
giving the first error term in the first assertion of the lemma using Chebyshev's inequality. We also bound the sum over components of size less than $M$. Indeed
$$ \E\left [ \sum_{A \in \comp_{p}, |A|\leq M} |A|^2\right]= V\E[|\C(u)| {\bf 1}_{|\C(u)|\leq M}] \pe V\sqrt{M} \, ,$$
using the bound $\proba_p(|\C(v)|\geq k)\pe k^{-1/2}$ by \cite[Theorem 1.3]{BCHSS3}. Markov's inequality then gives the second error term in the first assertion of the lemma.

Next we apply \cref{lem:sharpChiChi2Torus} to get that 
$$ \E \sum_{A \in \comp_{p}} |A|^3 = \E \sum_{u} \E|\C(u)|^2 = (\Cnst_2+o(1)) V\chi(p)^3 \, .$$
The second moment is estimated similarly according to whether $u \lr v$ or not, and using the BK inequality and tree graph inequalities \eqref{eq:treegraph2}. This gives
$$ \E \sum_{u,v}|\C(u)|^2 |\C(v)|^2 \leq \Big ( \E \sum_{u} \E|\C(u)|^2 \Big )^2 + O(V\chi(p)^9) \, .$$
Hence
$$ \Var \sum_{A \in \comp_{p}} |A|^3 \pe V^{1/2} \chi(p)^{9/2} = o(V\chi(p)^3) \, ,$$
since $\chi(p)=o(V^{1/3})$. Lastly, we again use the bound $\proba_p(|\C(v)|\geq k)\pe k^{-1/2}$ to obtain
$$ \E\left [ \sum_{A \in \comp_{p}, |A|\leq M} |A|^3\right]= V\E[|\C(u)|^2 {\bf 1}_{|\C(u)|\leq M}] \pe VM^{3/2} = o(V\chi(p)^3) \, ,$$
and Chebyshev's and Markov's inequalities together with the last two estimates conclude the proof.
\end{proof}

\begin{proof} [Proof of \cref{lem:Np1p2Bound}] The proof is similar to that of \cite[Lemma 4.15]{HypercubePaper}. If $\{u,v\}\in N(p_1,p_2; M)$, then there must exist a component $A \in \comp_{p_1}$ with $|A|\leq M$ and a $p_2$-open path $\gamma$ from $u$ to $v$ passing through $A$. Let us first handle the case in which $u,v\not \in A$. In this case we may consider the first edge $(x',x)$ entering $A$ (that is, $x' \not \in A$ and $x\in A$ and this is the first edge on $\gamma$ traversed from $u$ to $v$ with this property) and the last edge $(y,y')$ leaving $A$ (that is, $y \in A$ and $y' \not \in A$ and this is the last edge in $\gamma$ with this property). 
Hence there must exists two edges $(x,x')$ and $(y,y')$ such that the following events occur disjointly: 
\begin{enumerate}
    \item $u \lr x'$ in $\omega_{p_2}$,
    \item $v \lr y'$ in $\omega_{p_2}$
    \item The edges $(x,x'),(y,y')$ are in $\omega_{p_2}\setminus \omega_{p_1}$ and $x' \lr y'$ in $\omega_{p_1}$.
\end{enumerate}
We apply the BKR inequality. Summing over $u$ and $v$ gives a $\chi(p_2)^2$ contribution. Next, conditioned on $\omega_{p_1}$ and fixing $x'$ and $y'$, there are at most $m^2$ choices of $x,y$, and the conditional probability that $(x,x'),(y,y')$ are in $\omega_{p_2}\setminus \omega_{p_1}$ is $((p_2-p_1)/(1-p_1))^2$. Hence, summing over $x',y'$, the contribution arising from the third event is bounded above by
$$ m^2 \Big ( {p_2-p_1 \over 1-p_1} \Big)^2 V \E_{p_1} \left [ |\C(v)|{\bf 1}_{|\C(v)|\leq M} \right ] \pe m^2 (p_2-p_1)^2 V \sqrt{M} \, ,$$
where the last inequality is since $\proba_{p_1}(|\C(v)|\geq k)\pe k^{-1/2}$. We get a total contribution of $\chi(p_2)^2m^2 (p_2-p_1)^2 V \sqrt{M}$ from the case $u,v \not \in A$.

The two other cases follow similar reasoning. If $u,v \in A$, then $v \in \C(u)$ and $|\C(u)|\leq M$. Summing over $u,v$ in this case gives a bound of $V\sqrt{M}$. Lastly, if $u\not \in A$ and $v\in A$, the path connecting $u$ to $v$ has a last edge $(x,x')$ entering $A$. A similar analysis using the BKR inequality as before gives a contribution of $\chi(p_2)m(p_2-p_1)V\sqrt{M}$. This concludes the proof.
\end{proof}

\subsection{Proof of \cref{thm:Delta2}}
The proof is a second moment argument, as presented in the following lemmas.

\begin{lemma} \label{lem:Delta2FirstMom}
    $$ \E_p \sum_{a,b \in \comp_{p,M}} \Delta_{a,b}^2 \pe \chi(p)^2 \, . $$
\end{lemma}
\begin{proof}
We have
$$ \E_p \sum_{a,b \in \comp_{p,M}} \Delta_{a,b}^2 = \sum_{(u,v), (u',v')\in E} \proba_p(u \lr u' \nlr v \lr v', \C(u)\geq \M , \C(v)\geq \M) \, .$$
By the BK inequality this is at most 
\be\label{eq:Delta2SplitSum} \sum_{(u,v), (u',v')\in E} \proba_p(u \lr u', |\C(u)|\geq \M) \proba_p(v \lr v', |\C(v)|\geq \M) \, .\ee
We split this sum according to whether $\l u - u' \r$ has reached the plateau or not, i.e, whether $\l u - u' \r \leq (V/\chi(p))^{1/(d-2)}$ or not. When it is, we use \cref{clm:TwoEdgesLargeCompShort} and get a bound of 
$$ {V \chi(p)^2 \over M^2} \Big ({V \over \chi(p)}\Big )^{1/(d-2)} \leq \chi(p)^2 {V^{6/5} \over M^2 \chi(p)^{1/5}} = o(\chi(p)^2) \, ,$$
since $d\geq 7$ and $M=V^{\mp}$. For the sum \eqref{eq:Delta2SplitSum} over edges with $\l u - u' \r \geq (V/\chi(p))^{1/(d-2)}$ we drop the $|\C(u)|\geq M$ and $|\C(v)|\geq M$ and apply \cref{thm:plateau}. This immediately gives a bound of $O(\chi(p)^2)$ as desired. 
\end{proof}

Before proceeding, let us mention the following classical estimate which we will use here and in the next section. For any $p<p_c(\Z^d)$ we have
\be\label{eq:AizenmanTreeGraphMax} \proba_p \Big ( \max_{a\in \comp_{p}} |a| \leq \X(p)^2 \log(V)^3 \Big ) \geq 1- V^{-100} \, , \ee
see \cite{TreeGraph}, or \cite[(6.77)]{Grimm99}. The next lemma, together with \cref{lem:Delta2FirstMom} and Chebychev's inequality completes the proof of \cref{thm:Delta2}.

\begin{lemma} \label{lem:Delta2SecondMom}
    $$ \mathrm{Var} \Tr(\Delta^2) = o(\chi(p)^4)  \, . $$
\end{lemma}
\begin{proof}
Write $S=\Tr(\Delta^2)=\sum_{a\neq b\in \comp_{p,M}} \Delta_{a,b}^2$. We have 
\begin{equation} \E[S^2] \leq 2 \E\Big [ \sum_{\substack{\{a,\{b,c\}\} \\ a\neq b, a\neq c}} \Delta_{a,b}^2 \Delta_{a,c}^2 \Big ]+2\E \Big[ \sum_{\substack{\{\{a,b\},\{c,d\}\}\\ |\{a,b,c,d\}|=4}} \Delta_{a,b}^2\Delta_{c,d}^2 \Big] \, ,
\label{eq:Delta2SecondMom_Cases} \end{equation}
since $S^2$ goes over ordered pairs of unordered distinct pairs of elements of $\comp_{p,M}$.
We estimate each of the above sum apart. We split the first sum according to whether $N=\max_{a}\sum_{b}\Delta_{a,b}^2$ is small or not. By \cref{thm:MaxSumDelta}, we have 
\[ \E\bigg [\sum_{\substack{a,b,c \\ a\neq b, a\neq c}} \Delta_{a,b}^2\Delta_{a,c}^2 \1_{N>\log(V)^{10} \X^5/V} \bigg] \pe  V^7V^{-100}=o(1).\]
Also, by \cref{lem:Delta2FirstMom}
\[ \E\bigg [\sum_{a} \Delta_{a,b}^2\Delta_{a,c}^2 \1_{N\leq \log(V)^{10}\X^5/V} \bigg] \leq {\log(V)^{10} \X(p)^5 \over V} \E\bigg [\sum_{a\neq b}\Delta_{a,b}^2 \bigg]\pe \log(V)^{10}\X(p)^7/V=o(\X(p)^4) \, ,\]
since $\X(p)\leq V^{1/3}/\log(V)^{10}$ by our assumption on $p$ and \cref{thm:subcriticalBCHSS}. It remains to bound from above the second sum on the right hand side of \eqref{eq:Delta2SecondMom_Cases} which we denote by $S'$. %
We now slightly abuse notation and for two disjoint vertex subsets $a,b$ of $\Z_n^d$ we write $\Delta_{a,b}$ for the number of edges with one endpoint in $a$ and the other in $b$. With this notation we have that
$$ \E S = \sum_{\substack{a, b \subset \Z_n^d , |a|,|b| \geq M \\ a \cap b = \emptyset}} \Delta_{a,b}^2 \proba(a,b\in \comp_{p,M}) \, .$$
Given two disjoint vertex subsets $a,b$ of size at least $M$ we denote 
$$ S(a,b):=\sum_{\substack{c, d \subset \Z_n^d, |c|,|d|\geq M \\ c \cap d = \emptyset,(c\cup d)\cap (a\cup b) = \emptyset }} \Delta_{c,d}^2 {\bf 1}_{c,d\in \comp_{p.M}} \, .$$
With this notation we have 
$$ [\E S]^2 \geq 2 \sum_{\substack{a, b \subset \Z_n^d , |a|,|b| \geq M \\ a \cap b = \emptyset}} \Delta_{a,b}^2 \proba(a,b\in \comp_{p,M}) \E[S(a,b)]\, ,$$
and
\[ \E[S']= 2\sum_{\substack{a, b \subset \Z_n^d , |a|,|b| \geq M \\ a \cap b = \emptyset}} \Delta_{a,b}^2 \E \left [S(a,b) {\bf 1}_{a,b\in \comp_{p,M}} \right ] \, .\]
We denote
$$\nabla(a,b):=\E[S(a,b)|a,b\in \comp_{p,M}]-\E[S(a,b)] \, .$$ 
so by all the above we get that \begin{equation} \E[S']-[\E S]^2 \leq  2 \sum_{\substack{a, b \subset \Z_n^d , |a|,|b| \geq M \\ a \cap b = \emptyset}} \Delta_{a,b}^2\proba(a,b\in \comp_{p,M})\nabla(a,b) \, . \label{eq:Var_Tr(Delta^2)_good} \end{equation}
We wish to bound from above $\nabla(a,b)$. Note that $S(a,b)$ is not monotone with respect to adding edges (for instance, if an edge between $c$ and $d$ is opened then the event $c,d\in \comp_{p,M}$ cannot hold) we cannot conclude that $\nabla \leq 0$ which would have concluded the proof. Instead we must estimate directly the difference. The conditional distribution of $S(a,b)$ given $a,b\in \comp_{p,M}$ is equal to the conditional distribution of $S(a,b)$ given that all the edges touching  $a\cup b$ are all closed. In other words,
\[ \nabla(a,b)=\E[S(a,b) \off a\cup b]-\E[S(a,b)] \, .\]
Hence $\nabla(a,b)$ equals
\begin{eqnarray*} \sum_{\substack{(u,v),(u',v')\in E\\ u,v,u',v' \not \in a\cup b}} &\,& \proba\Big ( \big \{ u\lr u', v\lr v', \C(u)\neq \C(v), |\C(u)|,|\C(v)|\geq M \big \} \off a \cup b \Big ) \\ &-& \proba\Big ( u\lr u', v\lr v', \C(u)\neq \C(v), |\C(u)|,|\C(v)|\geq M \Big )
\end{eqnarray*}
Thus, using the inequality $\proba(A)-\proba(B) \leq \proba(A\setminus B)$ valid for all events $A,B$ we may bound
\[ \nabla(a,b)\leq \sum_{(u,v),(u',v')\in E} \proba\Big ( \big \{u \lr u', v\lr v', \C(u)\neq \C(v), |\C(u)|\geq M,|\C(v)|\geq M\big \} \without a\cup b \Big ) \, .\]
where the event ``$A$ without $a\cup b$'' means that $A$ does not occur, but if one closes all edges touching $a\cup b$, then $A$ does occur.
For any two edges $(u,v),(u',v')$, the event above holds only when there exists two disjoint paths connecting $u$ to $u'$ and $v$ to $v'$ and a third disjoint path connecting the first two paths and going through $a\cup b$ and that $|\C(u)|\geq M$ (in this case of course $\C(u)=\C(v)$). By the union about we get
\[ \nabla(a,b)\leq \sum_{\substack{ (u,v), (u',v') \in E\\ x,y\in V, w\in a\cup b}} \proba_p\Big ( \{u \slr x\} \circ \{x\slr u'\} \circ \{x\slr w\} \circ \{w\slr y\} \circ \{v\slr y\} \circ \{y\slr v'\}, |\C(u)|\geq M \Big ) \, .\]
We note that by transitivity, if we fix $w$ and sum over all the rest we obtain the same number. Thus we may sum over $w \in V$ (instead of $w\in a \cup b$) and gain a multiplicative factor of $|a\cup b|/V$. Additionally, at the price of a  multiplicative factor of $O(m)$, we may identify the vertices $u,v$ to just a single vertex $u$ and similarly the vertices $u',v'$ to a single vertex $u'$ (formally, in each configuration we open the edges $(u,v)$ and $(u',v')$ if they are closed). Therefore,
\begin{equation} \nabla(a,b) \pe \frac{|a\cup b|}{V} \sum_{\substack{u,u'\\x,y,w}}\proba_p\Big ( \{u \slr x\} \circ \{x\slr u'\} \circ \{x\slr w\} \circ \{w\slr y\} \circ \{u\slr y\} \circ \{y\slr u'\}, |\C(u)|\geq M \Big )\, . \label{eq:def_K_{2,3}}\end{equation}
We now estimate the sum on the right hand side according to whether the distance between $x$ and $y$ is at least $L=(V/\chi(p)^2)^{1/(d-4)}$ or not. In the former case we drop the event $|\C(u)|\geq M$ in the above and use the BK inequality to bound 
\begin{eqnarray*} \sum_{\substack{u,u',x,y,w\\ \l x-y\r \geq  L}}\proba_p(u \slr x) \proba_p(x\slr u')\proba_p(x\slr w) \proba_p(w\slr y)\proba_p(u\slr y)\proba_p(y\slr u') &=& \sum_{\l x-y\r \geq  L} T_2(x-y)^3 \\ &\pe& \X(p)^6/V \, ,\end{eqnarray*}
by \cref{lem:T345} and summing over $x,y$ with $\l x-y\r \geq  L$. Otherwise we use Markov's inequality and bound the probability on the right hand side of \eqref{eq:def_K_{2,3}} by
$$ {1 \over M} \sum_{z} \proba_p \Big ( z \slr u, \{u \slr x\} \circ \{x\slr u'\} \circ \{x\slr w\} \circ \{w\slr y\} \circ \{u\slr y\} \circ \{y\slr u'\}  \Big ) \, .$$
To bound this we traverse on the path from $z$ to $u$ until it visits one of the six disjoint paths in the event above. By adding the vertex of this first visit, we get a seventh disjoint path. No matter on which of the six paths the new vertex is on, after summing on it as well as on $z,u,u',w$ we get a contribution of $\chi(p)T_2(x-y)^2T_3(x-y)$. Therefore, 
\begin{eqnarray*}
\sum_{\substack{u,u',x,y,w\\ \l x-y\r \leq  L}} &\,& \proba_p\Big ( \{u \slr x\} \circ \{x\slr u'\} \circ \{x\slr w\} \circ \{w\slr y\} \circ \{u\slr y\} \circ \{y\slr u'\}, |\C(u)|\geq M \Big ) \\ &\leq& 
\X(p) \sum_{x,y: \l x-y\r \leq  L} T_2(x-y)^2 T_3(x-y) \, .
\end{eqnarray*} 
By \cref{lem:T345}, this is at most (up to a multiplicative constant $C(d,L)$)
$$ \X(p)\sum_{x,y: \l x-y\r \leq  L}\Big (\l x-y\r^{8-2d} + {\chi(p)^4 \over V^2} \Big ) \Big (\l x-y\r^{6-d}+ {\X(p)^3\over V}\Big ).$$
We use the crude bounds $\sum_{\l x-y\r \leq  L} \l x-y\r^{14-3d}\pe V\log(V)$ and $\sum_{\l x-y\r \leq  L} \l x-y\r^{8-2d}\pe V^{1+1/d}$ since $d\geq 7$ and $L\leq V^{1/d}$ as well as the sharp bound $\sum_{\l x-y\r \leq  L} \l x-y\r^{6-d}\pe L^6$. Therefore, the right hand side of \eqref{eq:def_K_{2,3}}, in the case that $\l x-y\r \leq  L$ is bounded above
$$ {\X(p) \over M} \Big ( V \log V + {V^{1+1/d} \chi(p)^3 \over V} + {L^6 \chi(p)^4 \over V^2} + {\chi^7 \over V^3} \Big ) \, .$$
By our choice of $L$ and since  $\chi(p)^3/V=o(1)$, each term in the parenthesis is bounded above by $V^{1+1/d}\leq V^{8/7}$. 

Putting these estimates into \eqref{eq:def_K_{2,3}} we obtain
$$ \nabla(a,b) \pe \frac{|a\cup b|}{V} \Big ( {\X(p)^6 \over V} + {V^{8/7}\X(p) \over M} \Big ) \, ,$$
which by \eqref{eq:Var_Tr(Delta^2)_good} we deduce
\[ \E[S']-\E[\Tr(\Delta^2)]^2 \pe \Big ( {\X(p)^6 \over V^2} + {V^{1/7}\X(p) \over M} \Big ) \sum_{\substack{a, b \subset \Z_n^d , |a|,|b| \geq M \\ a \cap b = \emptyset}}\Delta_{a,b}^2\proba(a,b\in \comp_p)|a\cup b| \, . \]
Thus, 
\[ \E[S']-\E[\Tr(\Delta^2)]^2 \pe \Big ( {\X(p)^6 \over V^2} + {V^{1/7}\X(p) \over M} \Big ) \E[\max_{a\in \comp_p} |a|\Tr(\Delta^2)] \, ,\]
so by \eqref{eq:AizenmanTreeGraphMax} (the error $V^{-100}$ is sufficient to make all other terms negligible) we obtain
\begin{eqnarray*}
\E[S']-\E[\Tr(\Delta^2)]^2 &\pe& \log(V)^3 \Big ( {\X(p)^8  \over V^2} + {V^{1/7}\X(p)^3 \over M} \Big ) \E [\Tr(\Delta^2)]  \\
    &\pe& {\X(p)^{10} \log(V)^3 \over V^2} + {\log(V)^3 V^{1/7}\X(p)^5 \over M}
\, ,
\end{eqnarray*}
where the last inequality is due to \cref{lem:Delta2FirstMom}. The first term is $o(\X(p)^4)$ by our assumption on $p$, and the second also since $\log(V)^3 V^{1/7}/M\X(p)=o(1)$ by our choice $M=V^{\mp}$. This concludes the proof.
\end{proof}

\subsection{Proof of \cref{thm:abDeltaAlternative}}

The first assertion of  \cref{thm:abDeltaAlternative} follows directly from \cite[Proposition A.1]{BCHSS1} combined with \cite[Theorem 2.1]{HeydenHofstad2}). The latter verifies that $p_c(\Z^d)(1-\eps(n))$ is in the subcritical phase. 

The second assertion of \cref{thm:abDeltaAlternative} is proved by first proving the corresponding estimate to the sum over $\comp_p$ instead of $\comp_{p,M}$. The corresponding sum over $\comp_p$ is easier since 
\be \label{eq:abDelta_E}
\E \bigg [\sum_{a \neq b \in \comp_p} |a||b|\Delta_{a,b} \bigg ] =  \sum_{x,y\in V, (v,v')\in E}\proba_p(x \lr v, y \lr v', v \nlr v') = (1-p)V\chi'(p) \, ,
\ee
where the last equality is Russo's formula, see \cite[Chapter 2.4]{Grimm99}.

\begin{lemma} \label{lem:abDelta_Var}
    \[ \Var \bigg [\sum_{a \neq b \in \comp_p} |a||b|\Delta_{a,b} \bigg ] =o(V^{4/3} \chi(p)^6) \, .\]
\end{lemma}

\begin{proof}[Proof of \cref{lem:abDelta_Var}] Write $S=\sum_{a\neq b\in \comp_{p}} |a||b|\Delta_{a,b}$. We have
\begin{equation} S^2 =  \sum_{a\neq b}\Delta_{a,b}^2|a|^2|b|^2+  \sum_{\substack{
\{a,\{b,c\}\}\\ |\{a,b,c\}|=3}}  \Delta_{a,b}\Delta_{a,c}|a|^2 |b| |c|+2\sum_{\substack{\{\{a,b\},\{c,d\}\} \\ |\{a,b,c,d\}|=4}} \Delta_{a,b}\Delta_{c,d}|a||b||c||d| \, ,
\label{eq:abDelta_Var_Cases}\end{equation}
since $S^2$ goes over ordered pairs of unordered distinct pairs of elements of $\comp_p$. We bound from above the expectation of each of the above sums separately. For the first sum we divide according to whether $\max_{a\in \comp_p} |a|\leq \X(p)^2 \log(V)^3$ or not. By \eqref{eq:AizenmanTreeGraphMax}
\[ \E\bigg [ \1_{\max_{a\in \comp_p} |a|> \X(p)^2\log(V)^3} \sum_{a\neq b}\Delta_{a,b}^2|a|^2|b|^2 \bigg ]\leq V^{10}V^{-100}= o(1) \, .\]
Also,
\[ \E\bigg [\1_{\max_{a\in \comp_p} |a|\leq \X(p)^2\log(V)^3} \sum_{a\neq b}\Delta_{a,b}^2|a|^2|b|^2 \bigg ]\leq \X(p)^8 \log(V)^{12} \E \bigg [\sum_{a\neq b}  \Delta_{a,b}^2\bigg ] \pe \X(p)^{10} \log(V)^{12} \, ,\]
where the last inequality is by \cref{lem:Delta2FirstMom}. Both upper bounds are $o(V^{4/3}\X(p)^6)$ by our assumption on $p$ (and \cref{thm:subcriticalBCHSS}). 
We proceed similarly for the second sum of \eqref{eq:abDelta_Var_Cases}. By  \eqref{eq:AizenmanTreeGraphMax} again we have
\[ \E\bigg [\1_{\max_{a\in \comp_p} |a|> \X(p)^2\log(V)^3} \sum_{\substack{a,b,c\\ |\{a,b,c\}|=3}}  \Delta_{a,b}\Delta_{a,c}|a|^2 |b| |c| \bigg ]\leq  V^{-80} \, .\]
Also,
\[ \E\bigg [\1_{\max_{a\in \comp_p} |a|\leq \X(p)^2\log(V)^3} \sum_{\substack{a,b,c\\ |\{a,b,c\}|=3}} \Delta_{a,b}\Delta_{a,c}|a|^2 |b| |c| \bigg ]\leq \X(p)^4\log(V)^6  \E\bigg [\sum_{\substack{a,b,c\\ |\{a,b,c\}|=3}}  \Delta_{a,b}\Delta_{a,c}|b||c| \bigg ] \, . \]
The last expectation is easy to bound using BK inequality, indeed,
\[ \E\bigg [\sum_{\substack{a,b,c\\ |\{a,b,c\}|=3}} \Delta_{a,b}\Delta_{a,c}|b||c| \bigg ]\leq \sum_{\substack{(a_1,b_1),(a_2,c_1) \in E, \\ b_2, c_2\in V}} \proba_p(a_1 \lr a_2)\proba_p(b_1 \lr b_2)\proba_p(c_1\lr c_2)\pe V\X(p)^3 \, ,\]
which is again $o(V^{4/3}\X(p)^6)$ by our assumption on $p$.

It remains to upper bound the third sum on the right hand side of  \eqref{eq:abDelta_Var_Cases}, which is the main contribution. We denote it by $S'$. We now make the same slight abuse of notation as in \cref{lem:Delta2SecondMom} and consider $a,b,c,d$ and pairwise disjoint vertex subsets of the torus $\Z_n^d$, so that
$$ \E S = \sum_{\substack{ a,b\subset \Z_n^d, a\cap b=\emptyset} } \Delta_{a,b}|a||b|\proba_p(a,b\in \comp_p) \, .$$
For two such disjoint subsets $a,b$ we write
$$ S(a,b):=\sum_{\substack{c, d \subset \Z_n^d, c \cap d = \emptyset, \\ (c\cup d) \cap (a\cup b) = \emptyset }}  \Delta_{c,d}|c||d| {\bf 1}_{c,d\in \comp_{p}} \, .$$ With this notation we have
\[ [\E S ]^2 \geq 2 \sum_{\substack{ a,b\subset \Z_n^d, a\cap b=\emptyset}} \Delta_{a,b}|a||b|\proba_p(a,b\in \comp_p) \E S(a,b) \, ,  \]
and 
\[ \E[S']= 2 \sum_{a,b} |a||b|\Delta_{a,b} \proba(a,b \in \comp_{p})\E \left [S(a,b)|a,b\in \comp_{p} \right ] \, .\]
We again set 
\[ \nabla(a,b):= \E[S(a,b)|a,b\in \comp_p]-\E[S(a,b)] \, ,\]
so that
\be\label{eq:ForTheFuture} \E[S']-\E[S]^2 \leq  2 \sum_{\substack{ a,b\subset \Z_n^d, a\cap b=\emptyset} } |a||b|\Delta_{a,b} \proba(a,b \in \comp_p)\nabla(a,b). \ee
The conditional distribution of $S(a,b)$ given $a,b\in \comp_{p}$ is equal to the conditional distribution of $S(a,b)$ given that all the edges touching  $a\cup b$ are all closed. In other words,
\[ \nabla(a,b)=\E[S\off a\cup b]-\E[S] \, .\]
Therefore $\nabla(a,b)$ equals
$$ \sum_{\substack{(u,v)\in E x,y\in V\\ u,v,x,y \not\in a\cup b}} \Big [ \proba_p(x \lr u, y \lr v, \C(u)\neq \C(v) \off a\cup b ) - \proba_p(x \lr u, y \lr v, \C(u)\neq \C(v)) \Big ]\, ,$$
so we may use the inequality $\proba(A) - \proba(B) \leq \proba(A\setminus B)$ to bound
\[ \nabla(a,b)\leq \sum_{\substack{(u,v)\in E \\ x,y\in V}}\proba_p( \{ x \lr u, y \lr v, \C(u)\neq \C(v) \} \without a\cup b) \, ,\]
where the event ``$A$ without $a\cup b$'' means that $A$ does not occur, but if one closes all edges touching $a\cup b$, then $A$ does occur. Since the events $x\lr u$ and $y \lr v$ \emph{are} monotone with respect to adding edges, if the event above holds, then there must be disjoint open paths connecting $x$ to $u$ and $y$ to $v$, but $\C(u)=\C(v)$ because these two paths are connected by a third path that visits $a\cup b$. Using BK inequality we deduce that 
\[\nabla(a,b)\leq \sum_{\substack{(u,v)\in E, w\in a\cup b \\ x,y,u',v'\in V}} \proba_p(x \slr u')\proba_p(u' \slr u)\proba_p(u' \slr w) \proba_p(w \slr v') \proba_p(y \slr v')\proba_p(v' \slr v) \, .\]
We note that by transitivity, if we fix $w$ and sum over all the rest we obtain the same number. Thus, 
\[\nabla(a,b)\leq \frac{|a\cup b|}{V}\sum_{\substack{(u,v)\in E\\ w,x,y,u',v'\in V}} \proba_p(x \slr u')\proba_p(u' \slr u)\proba_p(u' \slr w) \proba_p(w \slr v') \proba_p(y \slr v')\proba_p(v' \slr v) \, .\]
We first sum over $x,y$ and obtain a factor of $\X(p)^2$, yielding
\[ \nabla(a,b)\leq \frac{|a\cup b|}{V}\X(p)^2 \sum_{(u,v)\in E}T_4(u-v) \, .\]
\cref{lem:T345} bounds $T_4(u-v) \pe \chi(p)^4/V$ since $\chi(p)\geq V^{3/10}$ and $d\geq 7$. By summing over edges $(u,v)$ we get
\[ \nabla(a,b)\pe |a\cup b|\X(p)^6/V \, .\]
Therefore, by \eqref{eq:ForTheFuture}
\[ \E[S']-\E[S]^2\pe  {\X(p)^6 \over V} \E\Big [ \sum_{a \neq b \in \comp_p} |a||b||a\cup b|\Delta_{a,b} \Big ] \, , \]
hence,
\[ \E[S']-\E[S]^2\pe  {\X(p)^6 \over V} 
 \E[S\cdot \max_{a\in \comp_{p}} |a|] \, . \]
Applying \eqref{eq:AizenmanTreeGraphMax} gives 
\[ \E[S']-\E[S]^2\pe  {\X(p)^8 \log(V)^3 \over V} \E[S] \, ,\]
and by \eqref{eq:abDelta_E} and the first assertion of the \cref{thm:abDeltaAlternative} we obtain 
\[ \E[S']-\E[S]^2\pe {\X(p)^{10} \log(V)^3 \over V}  \pe  o(V^{4/3} \X^6(p)) \, , \]
which is $o(V^{4/3} \X^6(p))$ by our assumption on $p$, as required.
\end{proof}

By Chebychev's inequality, \eqref{eq:abDelta_E} and \cref{lem:abDelta_Var} implies that 
$$
\sum_{\substack{a \neq b \in \comp_{p}}}|a||b|\Delta_{a,b}= (1-p)\X'(p)V+o_{\proba}(V^{2/3}\X(p)^3) \, .
$$
Therefore, the following lemma together with Markov's inequality concludes the proof of \cref{thm:abDeltaAlternative}. We use our choice of $M=V^{\mp}$ and our assumption that $\chi(p) \geq V^{0.33}$ which guarantee that $\sqrt{M} V \chi(p) = o(V^{2/3} \chi(p)^3)$.

\begin{lemma} For every $M$ and $p$ with $\X(p)\leq V^{1/3}$ we have:
\[ \E\bigg \lbrack \sum_{a,b}|a||b|\Delta_{a,b}\1_{|a|\leq M} \bigg \rbrack \pe \sqrt{M}V \X(p). \]
\end{lemma}
\begin{proof} We have by the BKR inequality
\begin{align*} \E\bigg [\sum_{a,b}|a||b|\Delta_{a,b}\1_{|a|\leq M} \bigg] & =\sum_{(u,v),u',v'} \proba(u\lr u',v\lr v',u\nlr v,|\C(u)|\leq M) 
\\ & \leq \sum_{(u,v),u',v'} \proba(u\slr u',|\C(u)|\leq M)\proba(v\lr v') .\end{align*}
Summing over $v'$ gives a $\X(p)$ term and summing over $u'$ gets a term $\E[|\C(u)|\1_{|\C(u)|\leq M}]\pe \sqrt{M}$ using the bound $\proba_p(|\C(v)|\geq k)\pe k^{-1/2}$ of \cite[Theorem 1.3]{BCHSS3}.
\end{proof}

\subsection{Proof of \cref{lem:cmin}}
By differentiating \eqref{eq:defW1} we immediately get that $\cmin=\frac{VS_1}{mS_2}$ where 
\begin{eqnarray*}
    S_1 &=& \sum_{A\neq B \in \comp_{p,M}} |A||B|\Delta_{A,B} \, ,\\
    S_2 &=& \sum_{A\neq B \in \comp_{p,M}} |A|^2|B|^2 \, .
\end{eqnarray*}
\cref{thm:abDeltaAlternative} gives that
\[ S_1=(1-p)\X'(p)V+o_{\proba}(V^{2/3}\X(p)^3) = (1-p)\X'(p)V\Big (1 + o_{\proba}(V^{-1/3}\X(p))\Big ) \, ,\]
where the second equality is follows since $\X'(p)\asymp \X(p)^2$ by \cref{thm:abDeltaAlternative}.
Next we have that 
$$ S_2 = \Big (\sum_{A\in \comp_{p,M}} |A|^2\Big)^2 - \sum_{A\in \comp_{p,M}} |A|^4 \, .$$ 
\cref{lem:ConcentrationFixedp2} implies that that the first term is $V^2\chi(p)^2(1+O_\proba(V^{-1/2}\chi(p)^{3/2}+\sqrt{M}/\chi(p)))$ and the expectation of the second term is $V\E_p|\C(v)|^3$ which is of order $V\chi(p)^5$ by the tree graph inequalities \eqref{eq:treegraph2}. Together this yields
$$ S_2 = V^2\chi(p)^2 \Big ( 1 + O_\proba(V^{-1/2}\chi(p)^{3/2}) \Big ) \, ,$$
by our choice of $M=\chi(p)^5/V$. By dividing the proof is complete since $V^{-1/2}\chi(p)^{3/2} = o(V^{-1/3} \chi(p))$.
\qed

\subsection{Proof of \cref{thm:Delta4}}

We decompose $\Tr(\Delta^4)=S_1+S_2$ where 
\begin{eqnarray*} S_1 &:=& \sum_{\substack{a,b,c,d \in \comp_{p,M} \\ |\{a,b,c,d\}|=4}}\Delta_{a,b}\Delta_{b,c}\Delta_{c,d}\Delta_{d,a} \, , \\
 S_2 &:=& \sum_{\substack{a,b,c,d \in \comp_{p,M} \\ a\neq b,b\neq c,c\neq d, d\neq a\\ |\{a,b,c,d\}|<4}}\Delta_{a,b}\Delta_{b,c}\Delta_{c,d}\Delta_{d,a} \, .
\end{eqnarray*}

Thus, to prove \cref{thm:Delta4} it is enough to show that $S_1$ is concentrated around its mean (given in the next lemma) and that $S_2=o_{\proba}(\X(p)^4)$. We remark that throughout this section the corresponding sums are over $\comp_{p,M}$ instead of $\comp_p$ only because we found it convenient to occasionally use \cref{thm:MaxSumDelta}; it is possible to show that the small component do not contribute more to $\Tr(\Delta^4)$ but we do not need this.

\subsubsection{First moment of $S_1$}

We will use the following standard estimate for the simple random walk on the torus; we have not found this precise formulation in the literature so we provide a brief proof here. 
\begin{claim}\label{clm:srwtorus} Let $p_j$ for the probability that the simple random walk on the torus $Z_n^d$ returns to its starting vertex after $j$ steps. Then there exists $C=C(d)>0$ such that for any $j\geq 0$ we have
$$ p_j + p_{j+1} \leq Cj^{-d/2} + 2n^{-d} \, .$$
\end{claim}
\begin{proof} The eigenvalues of the transition matrix can be described as follows. 
For each $k =(k_1,\ldots,k_d) \in \Z_n^d$ we write 
$$ \lambda_k = {1 \over d} \sum_{i=1}^d \cos \Big ( {2\pi k_i \over n } \Big ) \, .$$
Therefore $p_j = {1 \over V} \sum_{k} \lambda_k^j$, hence
$$ p_j + p_{j+1} = {1 \over V} \sum_{k} \lambda_k^j(1+\lambda_k) \, .$$
The term $k=0$ gives $2V^{-1}$, so it suffices to bound the last sum over $k\neq 0$ by $O(j^{-d/2})$. Clearly $\lambda_k \in [-1,1]$ and since $\cos(t) -1+{t^2/2} = O(t^4)$ as $t\to 0$ we deduce that there exists positive constants $c_1,c_2, \eta$ (depending only on $d$) such that if $\|k\|_\infty \leq c_1 n$, then 
$$ 1-\lambda_k \geq 1- {c_2 \|k\|_2^2 \over n^2} \, ,$$
and $\lambda_k\in[0,1)$ and if $\|k\|_\infty \geq c_1 n$, then
$$ \lambda_k \leq 1-\eta \, .$$ 
In the first case we estimate $$\lambda_k^j(1+\lambda_k)\leq 2 e^{-j(1-\lambda_k)} \leq e^{-{c_2 j \|k\|_2^2 \over n^2}} \, ,$$
so that 
$$ \sum _{k \in \Z_n^d : \|k\|_\infty \leq c_1 n} \lambda_k^j(1+\lambda_k) \pe \sum _{k \in \Z_n^d : \|k\|_\infty \leq c_1 n}  e^{-{c_2 j \|k\|_2^2 \over n^2}} \pe \, j^{-d/2} \, ,$$
by a standard Gaussian sum. In the second case $\|k\|_\infty \leq c_1 n$, the sum decays exponentially in $j$. This concludes the proof.
\end{proof}

\begin{lemma} \label{lem:TrDelta4Main} If $p$ satisfies $\X(p)=o(V^{1/3})$ and $\X(p)\gg V^{0.3}$ we have
$$ \E_p S_1 \leq (1 + o_{m}(1))m^4\X(p)^4
$$    
\end{lemma}
\begin{proof} By the BK inequality, 
\[\E_p S_1 \leq \sum_{(u,u'),(v,v'),(x,x'),(y,y')}\proba(u\lr v')\proba(v\lr x')\proba(x\lr y')\proba(y\lr u').\]
Let $\tau_p(x) =\proba(0\lr x)$, and let $\cD(x)=1/m$ if $(0,x)$ is an edge of $G$, or $\cD(x)=0$ otherwise. We may rewrite the inequality above as \be\label{eq:tr4FirstStep}\E_p S_1 \leq m^4 V[\tau_p\star \cD\star \tau_p\star \cD\star \tau_p\star \cD\star \tau_p\star \cD(0)]\, .\ee
For a function $f:\T_n^d\to \R$ denote by 
 $\hat f$ denote the usual Fourier transform of $f$, that is, for each $k \in (\T_n^d)^*$
$$ \hat f (k) = \sum_{y \in \T_n^d} f(y) e^{ik \cdot y} \, ,$$
where $(\T_n^d)^\star= {2\pi \over n}\T_n^d$ stands for the dual of the torus. The inverse Fourier transform is given by 
$$ f(x) = {1 \over V} \sum_{k \in (\T_n^d)^\star} \hat f(k) e^{-ik\cdot x} \, ,$$
for $x\in \T_n^d$. Since $\widehat{f \star g}=\hat f \cdot \hat g$ we may rewrite \eqref{eq:tr4FirstStep} as 
\begin{equation*} \E_p S_1 \leq m^4 \sum_{k\in \T^\ast}(\hat \cD(k)\hat \tau_p(k))^4.\label{eq:FourierBound}\end{equation*}
Using the infrared bound from \cite[Theorem 6.1]{BCHSS2} we have
\[ \E_p S_1 \leq m^4 \sum_{k\in \T^\ast}\hat \cD(k)^4\left (\frac{1+o_{m}(1)}{1-\kappa_p\hat \cD(k)}\right )^4,  \]
where the $o_{m}(1)$ term is uniform over $k$ and $n$, and $\kappa_p=1-m(p_c-p)$. Since $\cD$ is an even map, $\hat D(k)$ is real for every $k$. Hence
\[ \E_p S_1 \leq (m^4+o_{m}(1)) \sum_{k\in \T^\ast}\hat \cD(k)^4 ({1-\kappa_p\hat \cD(k)} )^{-4} \, . \]
Then using the Taylor expansion of $x\mapsto (1-x)^{-4}$ (note that $|\kappa_p \hat{\cD}(k)|<1$) we get 
\[ \E_p S_1 \leq (m^4+o_{m}(1)) \sum_{k\in \T^\ast}\hat \cD(k)^4 \sum_{j=0}^\infty \binom{j+3}{3}(\kappa_p \hat \cD(k))^j. \]
Write $p_j$ for the probability that the simple random walk on the torus starting from $0$ is at $0$ after $j$ steps. Then 
\[ p_j=\frac{1}{V}\sum_{k\in \T^\star} \hat\cD(k)^j \, , \]
so 
\[ \E_p S_1 \leq  (m^4+o_{m}(1)) V  \sum_{j=0}^\infty \binom{j+3}{3}\kappa_p^j p_{j+4}. \]
We duplicate and divide by 2 the above sum, then change by $1$ the index of one of the sum, and use $\kappa_p\leq 1$ and $\binom{j+3}{3}\leq \binom{j+4}{3}$ to obtain
\[ \E_p S_1 \leq  (m^4+o_{m}(1)) V \sum_{j=-1}^\infty \binom{j+4}{3}\kappa_p^j \frac{p_{j+4}+p_{j+5}}{2}. \]
We apply \cref{clm:srwtorus} to obtain 
\[ \E_p S_1 \leq (m^4+o_{m}(1)) \sum_{j=-1}^\infty \binom{j+4}{3}\kappa_p^j + CV \sum_{j=-1}^\infty (j+4)^{3-d/2} \kappa_p^j \, , \]
The first sum on the right hand side equals $(1-\kappa_p)^{-4}\kappa_p^{-1}$ using the Taylor expansion of $(1-x)^{-4}$ again. For the second sum we bound $(j+4)^{3-d/2}$ by $Cj^{-1/2}$ since $d\geq 7$ and bound the sum $\sum_{j\geq 1} j^{-1/2}x^j \leq C(1-x)^{1/2}$ for $x\in(0,1)$ using the Taylor expansion of $(1-x)^{-1/2}$. %
We obtain that  
\[ \E_p S_1 \leq (m^4+o_{m}(1)) (1-\kappa_p)^{-4} + CV(1-\kappa_p)^{-1/2} \, .\]
Theorem 1.2 of \cite{BCHSS1} asserts that $p_c-p=(1+o_{m}(1))/\X(p)$ hence $(1-\kappa_p)^{-1}=(1+o_{m}(1))\X(p)$. Since we assumed that $\X(p)\gg V^{0.3}$ we get that the first term on the right hand side dominates the second and the proof is complete.
\end{proof}
 \subsubsection{The second moment of $S_1$}
\begin{lemma}\label{lem:Delta4VarS1} If $V^{0.3}\leq \X(p)\leq V^{1/3}/\log(V)^{20}$, then 
$$ \mathrm{Var}(S_1) = o(\X(p)^8).
$$    
\end{lemma}
\begin{proof} %

We have
\begin{equation} \label{eq:defS1}
    S_1^2= \sum_{\substack{a,b,c,d \\ |\{a,b,c,d\}|=4}}\Delta_{a,b}\Delta_{b,c}\Delta_{c,d}\Delta_{d,a}\sum_{\substack{a',b',c',d' \\ |\{a',b',c',d'\}|=4}}\Delta_{a',b'}\Delta_{b',c'}\Delta_{c',d'}\Delta_{d',a'} \, . 
\end{equation} 
We first argue that the contribution coming from terms for which $\{a,b,c,d\}\cap \{a',b',c',d'\}\neq \emptyset$ is negligible. Without loss of generality, we may focus on the case $a=a'$, that is, we wish to bound
$$  S_{\text{bad}} :=\sum_{\substack{a,b,c,d \\ |\{a,b,c,d\}|=4}}\Delta_{a,b}\Delta_{b,c}\Delta_{c,d}\Delta_{d,a}\sum_{\substack{b',c',d' \\ |\{a,b',c',d'\}|=4}}\Delta_{a,b'}\Delta_{b',c'}\Delta_{c',d'}\Delta_{d',a} \, .$$
For any $a$ we have by Cauchy Schwartz
\[ \sum_{\substack{b,c,d \\ |\{b,c,d\}|=3}} \Delta_{a.b}\Delta_{b,c}\Delta_{c,d} \Delta_{d,a} \leq \sqrt{\sum_{b\neq d} \Delta_{a,b}^2 \Delta_{a,d}^2}\sqrt{\sum_{b\neq d}\Big (\sum_{c\notin\{b,d\}} \Delta_{b,c}\Delta_{c,d}\Big )^2} \, . \]
We bound the first square root on the right hand side by $\max_{a} \sum_b \Delta_{a,b}^2$ which for notational convenience we denote by $N$. For the second we note that 
\begin{eqnarray*}
\sum_{b\neq d}\Big (\sum_{c\notin\{b,d\}} \Delta_{b,c}\Delta_{c,d}\Big )^2 = S_1 + \sum_{\substack{b,c,d : |\{b,c,d\}|=3}} \Delta_{b,c}^2 \Delta_{c,d}^2 
\end{eqnarray*}
and that 
$$ \sum_{\substack{b,c,d : |\{b,c,d\}|=3}} \Delta_{b,c}^2 \Delta_{c,d}^2 \leq \big (\max_{c} \sum_{b\neq c}\Delta_{b,c}^2 \big) \sum_{c\neq d} \Delta_{c,d}^2=N \|\Delta\|_F^2.$$
Since these last two bounds do not depend on $a$ it follows that 
\begin{equation*} S_{\text{bad}} %
 \leq N S_1 \sqrt{S_1+N \|\Delta\|_F^2}
 \leq N S_1^{3/2}+N^{3/2}S_1\|\Delta\|_F
 \leq 2N S_1^{3/2}+N^{5/2} \|\Delta\|_F^{3},
\end{equation*}
using $\sqrt{a+b}\leq \sqrt{a}+\sqrt{b}$ for any $a,b\geq 0$ for the second inequality and using $a^{2/3}b^{1/3}\leq a+b$ for the last inequality with $a^{2/3}=N^{2/3}S_1$ and $b^{1/3}=N^{5/6} \|\Delta\|_F$. To handle the first term on the right hand side we write
$$ N S_1^{3/2}=[N^2V/\X(p)^3]^{1/2}[S_1\X(p)/V^{1/3}]^{3/2} \, ,$$ and apply $a^{1/2}b^{3/2}\leq a^2+b^2$ to get
\[ \E[N S_1^{3/2}]\leq (V^2/\X(p)^6)\E[N^4]+(\X(p)^2/V^{2/3})\E[S_1^2] \, .\]
Since $N$ is at most polynomial in $V$,  \cref{thm:MaxSumDelta} implies that $\E N^4 \pe \log( V)^{20} \chi(p)^{20}/V^4$, hence $(V^2/\X(p)^6)\E[N^4] = o(\chi(p)^8)$ by our assumption on $p$. Also $\chi(p)^2/V^{2/3}=o(1)$ and we get that
\[ \E[N S_1^{3/2}]=o(\X(p)^8)+o(\E[S_1^2]) \, . \]
Also using \cref{thm:MaxSumDelta} and Jensen's inequality gives 
\begin{align*} \E[N^{5/2}\|\Delta\|_F^3] & \pe (\log(V)^{10}\X^5/V)^{5/2}\E[\|\Delta\|_F^3]+\proba(N>\log(V)^{10}\X^5/V)V^8
\\ & =o(\X^5)\E[\|\Delta\|_F^4]^{3/4}+o(1)=o(\X(p)^8) \, ,
\end{align*}
where we used \cref{lem:Delta2FirstMom,lem:Delta2SecondMom} to bound $\E[\|\Delta\|_F^4] \pe \X(p)^4$ and our assumption on $p$ to bound $(\log(V)^{10}\X^5/V)^{5/2} = o(\X(p)^5)$.
Therefore 
\begin{equation} \E[S_{\text{bad}}]=o\left (\X(p)^8)+o(\E[S_1^2])\right ).\label{eq:SbadTrDelta4}\end{equation}

We may thus focus on the contributions to $S_1^2$ coming from the terms in \eqref{eq:defS1} in which $\{a,b,c,d\}\cap \{a',b',c',d'\}= \emptyset$. To that aim we write  
\[ S_{\text{good}}:= \sum_{\substack{a,b,c,d \\ |\{a,b,c,d\}|=4}}\Delta_{a,b}\Delta_{b,c}\Delta_{c,d}\Delta_{d,a} \cdot S_1(a,b,c,d) \, ,\]
where
\[ S_1(a,b,c,d):=\sum_{\substack{a',b',c',d' \\ |\{a',b',c',d'\}|=4\\ \{a,b,c,d\}\cap \{a',b',c',d'\} =\emptyset }}\Delta_{a',b'}\Delta_{b',c'}\Delta_{c',d'}\Delta_{d',a'} \, .\]
We may rewrite $\E[S_{\text{good}}]$ as
\[\E[S_{\text{good}}] = \sum_{\substack{a,b,c,d \\ |\{a,b,c,d\}|=4}}\Delta_{a,b}\Delta_{b,c}\Delta_{c,d}\Delta_{d,a}\proba(a,b,c,d\in \comp_{p,M})\E\left [S_1(a,b,c,d) \mid a,b,c,d\in \comp_{p,M} \right ] \, .\]
Note that in sum we are summing over $4$ disjoint subsets $a,b,c,d$ of vertices of $\Z_n^d$. Therefore writing 
\[ \nabla(a,b,c,d):=\E[S_1(a,b,c,d) \mid a,b,c,d\in \comp_{p,M}]-\E[S_1], \]
we have 
\begin{equation} \E[S_{\text{good}}]-\E[S_1]^2 \leq  \sum_{\substack{a,b,c,d \\ |\{a,b,c,d\}|=4}}\Delta_{a,b}\Delta_{b,c}\Delta_{c,d}\Delta_{d,a}\proba(a,b,c,d\in \comp_{p,M})\nabla(a,b,c,d) \, . \label{eq:SgoodNabla}\end{equation}

We now upper bound $\nabla(a,b,c,d)$ for any fixed distinct subsets $a,b,c,d$. Note that the conditional distribution of $S_1(a,b,c,d)$ given  $a,b,c,d\in \comp_{p,M}$ is equal to the conditional distribution of $S_1$ given the closed and open set of edges that determine the event that $a,b,c,d$ are four disjoint connected components. Since the open set of edges does not matter for $S_1(a,b,c,d)$ we deduce that
\[ \nabla(a,b,c,d)=\E[S_1(a,b,c,d) \off a\cup b \cup c \cup d]-\E[S_1] \, . \]
As we have defined before, the event ``$A$ without $a\cup b\cup c \cup d$'' means that $A$ does not occur, but if one closes all edges touching $a\cup b \cup c \cup d$, then $A$ does occur. By the previous equality we have that $\nabla(a,b,c,d)$ is at most
\begin{eqnarray*} \sum_{\substack{(u,u'),(v,v'),\\(x,x'),(y,y')}}    \proba\left ( \{ u\slr v', v\slr x', x\slr y', y\slr u', |\{\C(u),\C(v),\C(x),\C(y)\}|=4 \} \text{ without } a\cup b\cup c\cup d\right )\, .
\end{eqnarray*}
Since $(u\slr v')\circ (v\slr x')\circ (x\slr y')\circ(y\slr u')$ still holds even after closing the edges touching $a\cup b\cup c\cup d$ we may consider disjoint paths $\gamma_u, \gamma_v, \gamma_x, \gamma_y$ from $u$ to $v'$, from $v$ to $x'$, from $x$ to $y'$ and from $y$ to $x'$ respectively that do not intersect $a\cup b\cup c\cup d$. Furthermore, since the event does not hold unless all edges touching $a\cup b\cup c\cup d$ are closed, we learn that  at least two vertices among $u,v,x,y$ are connected, and all open paths between them goes through $a\cup b\cup c\cup d$.  This implies that there exists another open path $\Gamma$ such that:
\begin{compactitem}
    \item $\Gamma$ is disjoint from $\gamma_u,\gamma_v,\gamma_x,\gamma_y$, and
    \item The starting point $w_1$ and endpoint $w_2$ of $\Gamma$ lies on $\gamma_u\cup \gamma_v\cup \gamma_x\cup \gamma_y$, and not on the same path. 
    \item There exists $w_3\in \Gamma\cap (a\cup b\cup c\cup d)$.
\end{compactitem}
Up to a factor of $4$, we may assume $w_1\in \gamma_u$ and that either $w_2\in \gamma_v$ or $w_2\in \gamma_x$. By the BK inequality it follows that $\nabla(a,b,c,d)$ is at most four times, the sum over edges $(u,u'),(v,v'),(x,x'),(y,y')$, and vertices $w_1,w_2,w_3$ with $w_3\in a\cup b\cup c\cup d$ of 
\begin{align*}  & \hspace{-1em} \proba(u\slr w_1)\proba(w_1\slr v')\proba(v\slr w_2)\proba(w_2\slr x')\proba(x \slr y')\proba(y \slr u')\proba(w_1 \slr w_3)\proba(w_3 \slr w_2)
\\  & + \proba(u \slr w_1)\proba(w_1 \slr v')\proba(v \slr x')\proba(x \slr w_2)\proba(w_2\slr y')\proba(y \slr u')\proba(w_1 \slr w_3)\tau(w_3 \slr w_2)
.\end{align*}
By transitivity, fixing $w_3$ and summing over the all other variables always gives the same number for any choice of $w_3$. Thus we may drop from the above sums the condition that $w_3\in a\cup b\cup c\cup d$ and add instead a multiplicative factor of $|a\cup b\cup c\cup d|/V$. Furthermore, up to a constant multiplicative factor, we may replace in the above sums $u'$ by $u$, $v'$ by $v$, $x'$ by $x$ and $y'$ by $y$. Hence,
\be\label{eq:OffWithoutNabla} \frac{V\nabla(a,b,c,d)}{|a\cup b\cup c\cup d|} \pe  \sum_{w_1,w_2}T_2(w_1-w_2)^2T_4(w_1-w_2)+\sum_{w_1,w_2}T_3(w_1-w_2)^2 T_2(w_1-w_2) \, .\ee
By transitivity the right hand side equals 
\[ V\sum_{w}T_2(w)^2T_4(w)+V\sum_w T_3(w)^2 T_2(w)\, .\]
We use Lemma \ref{lem:T345} to bound each of these sums. For the first, we use our assumption that  $\X(p)\geq V^{3/10}$ and $d\geq 7$ so that $T_4(w)\pe \X(p)^4/V$, hence,  we get %
\[ \sum_wT_2(w)^2T_4(w)\pe {\X(p)^4 \over V} \sum_{w}T_2(w)^2= {\X(p)^4 \over V} \sum_{w}T_4(0) \pe  {\X(p)^8 \over V^2} \, , \] %
For the second by Lemma \ref{lem:T345}, we have 
\[ \sum_w T_3(w)^2 T_2(w)\pe \sum_w (\langle w\rangle^{6-d}+\X(p)^3/V)^2(\langle w\rangle^{4-d}+\X(p)^2/V).\]
Then using $(a+b)^2\leq 2a^2+2b^2$, 
\[ \sum_w T_3(w)^2 T_2(w)\pe \sum_w(\langle w\rangle^{12-2d}+\X(p)^6/V^2)(\langle w\rangle^{4-d}+\X(p)^2/V).\]
We use $\sum_w\langle w\rangle^{12-2d}\pe V^{5/d}$, and $\sum_w\langle w\rangle^{4-d}\pe V^{4/d}$, and $\sum_w \langle w\rangle^{16-3d}\pe V^{2/d}$ since $d\geq 7$. Then using our assumption $\X(p)\geq V^{3/10}$ we get 
\[ \sum_w T_3(w)^2 T_2(w)\pe \Big (V^{2/d}+V^{4/d}\X(p)^6/V^2+V^{5/d}\X(p)^2/V+\X(p)^8/V^2 \Big) \pe \X(p)^8/V^2.\]

Therefore, going back to \eqref{eq:OffWithoutNabla} we obtain 
\[ \nabla(a,b,c,d)\pe |a\cup b\cup c\cup d|{\X(p)^8\over V^2} \, .\]
Hence going further back to \eqref{eq:SgoodNabla} we deduce 
\[ \E[S_{\text{good}}]-\E[S_1]^2 \pe  \sum_{\substack{a,b,c,d \\ |\{a,b,c,d\}|=4}}\Delta_{a,b}\Delta_{b,c}\Delta_{c,d}\Delta_{d,a}\proba(a,b,c,d\in \C_p)|a\cup b\cup c\cup d|{\X(p)^8 \over V^2} \, .\]
Finally, using that $|a \cup b\cup c\cup d| \leq V^{2/3}$, with probability at least $1-V^{-100}$ from \eqref{eq:AizenmanTreeGraphMax} and our assumption on $p$, we get
\[ \E[S_{\text{good}}]-\E[S_1]^2 \pe \E\left [ \sum_{\substack{a,b,c,d \in \comp_{p,M} \\ |\{a,b,c,d\}|=4}}\Delta_{a,b}\Delta_{b,c}\Delta_{c,d}\Delta_{d,a} \right ] V^{2/3} \X(p)^8/V^2=\E[S_1] { \X(p)^8 \over V^{4/3}} \, .\]
Since $\X(p)=o(V^{1/3})$ we get by \cref{lem:TrDelta4Main} that,
\[ \E[S_{\text{good}}]-\E[S_1]^2=o(\X(p)^8).\]
This together with \eqref{eq:SbadTrDelta4} and $|S_1^2-S_{\text{good}}|\leq 16S_{\text{bad}}$ yields
\[ \E[S_1^2]-\E[S_1]^2=o(\X(p)^8)+o(\E[S_1]^2),\]
and \cref{lem:TrDelta4Main} concludes the proof.
\end{proof}

\subsubsection{Proof of \cref{thm:Delta4}}

By Chebychev's inequality, \cref{lem:TrDelta4Main} and \cref{lem:Delta4VarS1} we deduce that $S_1 \leq (1 + o_m(1)) m^4\X(p_1)^4$ with probability $1-o(1)$. Since $\Tr(\Delta^4)=S_1+S_2$ it remains to handle $S_2$. It suffices to upper bound 
\[ \sum_{a=c,b\neq a,d\neq a} \Delta_{a,b}\Delta_{b,a}\Delta_{a,d}\Delta_{d,a} \leq \left (\max_{a}\sum_{d}\Delta_{a,d}^2\right )\left (\sum_{a\neq b} \Delta_{a,b}^2\right ).\]
Thus, by \cref{thm:MaxSumDelta} and \cref{lem:Delta2FirstMom} and our assumption on $p$, we deduce that $S_2 = o_{\proba}(\X(p)^4)$, which concludes the proof. \qed

\section*{Acknowledgments} We thank Nicolas Broutin for useful discussions. This research was supported by ERC consolidator grant 101001124 (UniversalMap) as well as ISF grants 1294/19 and 898/23.

\printbibliography

\end{document}